\documentclass[11pt, reqno]{amsart}
\usepackage{graphicx} 
\usepackage[a4paper,top=1in,bottom=0.9in,left=2.5cm,right=2.5cm]{geometry}
\usepackage{amsfonts,mathtools,amssymb,amsthm}
\allowdisplaybreaks

\usepackage{enumitem}
\usepackage{booktabs}
\usepackage{bm}
\usepackage{mathabx}
\usepackage{xcolor}
\usepackage{float}
\usepackage{longtable}
\usepackage{hyperref}
\usepackage{url}

\usepackage{comment}

\theoremstyle{plain}
\newtheorem{theo}{Theorem}[section]

\newtheorem{prop}[theo]{Proposition}
\newtheorem{lem}[theo]{Lemma}
\newtheorem{corol}[theo]{Corollary}
\newtheorem*{theo*}{Theorem}

\theoremstyle{definition}
\newtheorem{defi}[theo]{Definition}
\newtheorem{rem}[theo]{Remark}
\newtheorem{exa}[theo]{Example}

\newtheorem*{defi*}{Definition}

\newcommand{\Z}{\mathbb{Z}}

\newcommand{\Ps}{\mathbb{P}}

\newcommand{\Q}{\mathbb{Q}}
\newcommand{\OO}{\mathcal{O}}

\newcommand{\R}{\mathbb{R}}

\newcommand{\F}{\mathbb{F}}

\newcommand{\C}{\mathbb{C}}
\newcommand{\T}{\mathbb{T}}

\newcommand{\ga}{\gamma}
\newcommand{\G}{\mathbb{G}}

\newcommand{\be}{\beta}

\newcommand{\al}{\alpha}
\newcommand{\HH}{\mathbb{H}}

\newcommand{\Spec}{\operatorname{Spec}}

\newcommand{\Gal}{\text{Gal}}

\title{Hypergeometric motives from toric hypersurfaces}
\author{Asem Abdelraouf}
\address{Kyiv School of Economics, Mykoly Shpaka St, 3, Kyiv, Ukraine, 03113} 
\email{aabdelraouf@kse.org.ua}
\date{October 2026}

\begin{document}

\begin{abstract}
In this paper, we study two compactifications of general hypersurfaces defined by the vanishing of linear combinations of $d+2$ monomials in $d$-dimensional algebraic tori. 
We prove that the number of their $\F_q$-points is given by finite hypergeometric sums under certain general conditions. 
In the process, we introduce the notion of a gamma triple, which allows us to extend the classical definition of finite hypergeometric sums to prime powers corresponding to the cyclotomic field of definition of the associated monodromy representation. 
As a special case of our main results, we study the Dwork family and obtain a formula for the number of its $\F_q$-points. 
Our results generalise the work of Beukers, Cohen, and Mellit for finite hypergeometric sums defined over $\Q$.
\end{abstract}
\maketitle

\tableofcontents

\section{Introduction}
Throughout this paper, we denote by $\F_q$ the finite field of $q$ elements.
We fix a non-trivial additive character $\psi: \F_q \to \C^\times$ and a multiplicative character $\chi: \F_q^\times \to \C^\times$ of order $q^\times \coloneqq q-1$. For any positive integer $k$, we write $\zeta_k$ for the primitive $k$-th root of unity $ e^{2\pi i/k} \in \C$. \hfill

By \emph{hypergeometric parameters} $(\al \, ; \be)$, we mean two vectors of rational numbers $\al, \be \in \Q^{n}$ with $n \ge 0$ such that $\al_i-\be_j$ is not an integer for any $i$ or $j$.
Given hypergeometric parameters $(\al \, ; \be)$, let $q$ be a prime power such that  $q^\times \al,q^\times \be \in \Z^{n}$.  
For any $t \in \F_q^\times$, the \emph{finite hypergeometric sum}, or the \emph{hypergeometric trace}, is defined by
\begin{equation}\label{eq:intro-defi-hg}
F_q\left(\al \, ; \be \mid t\right) \coloneqq \frac{1}{1-q} \sum_{m=0}^{q^\times-1} \prod_{i=1}^n \frac{g(m+\alpha_iq^\times)g(-m-\beta_iq^\times)}{g(\alpha_iq^\times)g(-\beta_i q^\times)} \chi((-1)^nt)^m ,   
\end{equation}
where $g(m)$ is the Gauss sum defined in Section \ref{section:hypergeometric-function}. 
Such finite hypergeometric sums were introduced independently by Greene \cite{Greene-hypergeometric} and Katz \cite{Katz-expo-sums} around the same time, for different purposes and with different normalisations. We adopt the normalisation in
\cite[Definition 1.1]{BCM} and \cite[Equation 10.1]{hypergeometric-survey}.

The hypergeometric parameters $(\al\, ; \be)$ are said to be defined over a field $K$ if the coefficients of the polynomials
$\prod_{j=1}^n(T- e^{2\pi i \al_j}), \prod_{j=1}^n(T- e^{2\pi i \be_j})$ are in $K$. We call the smallest such $K$ that is a cyclotomic field the \emph{cyclotomic field of definition} of $(\al \,; \be)$.

\subsection{Motivation} Let $(\al \, ; \be)$ be hypergeometric parameters which are defined over $\Q$.
In \cite[Theorem 1.3]{BCM}, the authors prove that Definition \eqref{eq:intro-defi-hg} extends to all prime powers $q$ which are relatively prime to the common denominator of the $\al_i$'s and $\be_i$'s  as follows:
write 
\begin{equation}\label{eq:intro-defi-ga}
    \frac{\prod_{j=1}^n (T- e^{2\pi i \alpha_j})}{\prod_{j=1}^n (T- e^{2\pi i \beta_j}) } = \frac{ \prod_{\ga_j< 0} (T^{-\ga_j} - 1)}{ \prod_{\ga_j>0} (T^{\ga_j} - 1)} ,
\end{equation}
for some vector $\ga \in \Z^{d+2}_{\neq 0}$.
Let $s(m)$ be the multiplicity of the zero $e^{2 \pi i m/q^\times}$ in the greatest common divisor of the numerator and the denominator of the right-hand side of Equation \eqref{eq:intro-defi-ga}. 
Then,  we have the following equality
\begin{equation}\label{eq:intro-equality-of-hg-sums}
F_q\left(\al \, ; \be \mid t\right) = \frac{(-1)^{d+2}}{1-q}\sum_{m=0}^{q^\times-1} \prod_{j=1}^{d+2} g\left(-\gamma_j m \right)  q^{s(m)-s(0)} \chi^m(\gamma^{\gamma}t),
\end{equation}
where $\ga^{\ga} \coloneqq \prod_{j=1}^{d+2} \ga_j^{\ga_j}$. The right-hand side of \eqref{eq:intro-equality-of-hg-sums} is defined for all prime powers $q$ {that} are coprime to $\ga_1\cdots\ga_{d+2}$. This gives the desired extension of $F_q( \al \, ; \be \mid t)$.
We note that $\sum_{j=1}^{d+2} \ga_j=0$ {and} that $\gamma$ can be chosen so that $\gcd(\ga_1, \dots,\ga_{d+2})=1$. We call such a vector $\gamma$ a \emph{gamma vector}.
We denote the right-hand side of \eqref{eq:intro-equality-of-hg-sums} by $F_q(\ga,0,1 \mid t)$. 

In the same paper, the authors realise the finite hypergeometric sum \eqref{eq:intro-equality-of-hg-sums} in the point count of certain toric hypersurfaces:
let $\ga \in \Z^{d+2}_{\neq 0}$ be a gamma vector.
For $t \in \F_q^\times$, consider the affine variety $V_{t}$ defined by the homogeneous equations
\begin{align*}
    w_1+ \dots +w_{d+2} &= 0, \\
    w_1^{\gamma_1} \cdots w_{d+2}^{\gamma_{d+2}}  &=t, 
\end{align*}
and $ w_1\cdots w_{d+2} \neq 0$ in the projective space $\Ps^{d+1}$ with homogeneous coordinates $w_1, \dots, w_{d+2}$. 
Let $r$, respectively $s$, be the number of negative, respectively positive, components of $\ga$.
\begin{theo}\cite[Theorem 1.5]{BCM}\label{theo:BCM}
There is a compactification $\overline{V_{t}}$ of $V_{t}$ such that the number of $\F_q$-points of $\overline{V_{t}}$ is given by
\begin{align*}
\#\overline{V_{t}} (\mathbb{F}_q) &= \sum_{k=0}^{\min(r,s)-1} {\binom{r-1}{k}}{\binom{s-1}{k}} \frac{q^{d-k} - q^{k} }{q-1}\\
&+(-1)^{r+s-1} q^{\min(r,s)-1} F_q\left(\gamma,0,1 \mid t/\ga^{\ga}\right),
\end{align*}
whenever $q$ is a prime power that is coprime to $\ga_1\cdots\ga_{d+2}$.
\end{theo}
Theorem \ref{theo:BCM} provides a realisation of the finite hypergeometric sum $F_q(\ga,0,1 \mid t/ \ga^\ga)$ in the family of varieties $\overline{V_{t}}$. The variety $V_{t}$ is isomorphic to a hypersurface in the $d$-dimensional algebraic torus $\G_m^{d}$ defined by the vanishing of a Laurent polynomial, which is a linear combination of $d+2$ monomials. We study such hypersurfaces in general.
\subsection{Our results}
In this paper, we generalise the above results to finite hypergeometric sums not necessarily defined over $\Q$.
We consider rational functions of the form
\begin{equation}\label{eq:intro-gamma-triples}
    Q = \frac{ \prod_{\ga_j< 0} \left(T^{-\ga_j} - \zeta_{N}^{\delta_j}\right) }{ \prod_{\ga_j>0}\left( T^{\ga_j} - \zeta_{N}^{-\delta_j}\right)},
\end{equation}
where $\ga \in \Z^{d+2}_{\neq0 }$ is a gamma vector, $\delta \in \Z^{d+2}$ is a vector of integers, and $N$ is a positive integer. We call the triple $(\ga, \delta,N)$ a \emph{gamma triple}. Gamma triples define hypergeometric parameters by writing 
\begin{equation}\label{eq:intro-ga-tr-2}
Q = \frac{\prod_{j=1}^n (T- e^{2\pi i \al_j} )}{\prod_{j=1}^n (T- e^{2\pi i \be_j})},  
\end{equation}
where the numerator and the denominator of the right-hand side of Equation \eqref{eq:intro-ga-tr-2} have no common factors. We say that $(\ga,\delta,N)$ represents the hypergeometric parameters $(\al \, ; \be)$. 

Let $(\ga,\delta,N)$ be a gamma triple. 
For any prime power $q$, let $D_{\delta}(T)=D_{\ga,\delta,N}(T)$ be the greatest common divisor of the numerator and the denominator of the right-hand side of \eqref{eq:intro-gamma-triples}. Define $s_{\delta}(m)=s_{\ga,\delta,N}(m)$ to be the multiplicity of the zero $e^{2 \pi i m/q^\times}$ in $D_{\delta}(T)$.
Suppose that $q$ is coprime to $\ga_1 \cdots \ga_{d+2}$ and that $N$ divides $q^\times$. Then, the finite hypergeometric sum associated to the gamma triple $(\ga,\delta,N)$ is
\begin{equation}
F_q\left(\gamma, \delta,N \mid t\right) = \frac{1}{1-q}\sum_{m=0}^{q^\times-1} \prod_{j=1}^{d+2} \frac{g\left(-\gamma_j m +\frac{\delta_j}{N} q^\times\right)} {g\left(\frac{\delta_j}{N} q^\times\right)}  q^{s_{\delta}(-m)-s_{\delta}(0)} \chi^m(\gamma^{\gamma}t). \label{eq:intro-1}
\end{equation}
Our first result is the following:
\begin{theo} 
Suppose that $(\ga,\delta,N)$ represents the hypergeometric parameters $(\al \, ; \be)$. Then,
\[
F_q\left(\gamma, \delta,N \mid t\right)  = F_q( \al \, ; \be \mid t)
\]
for all prime powers $q$ for which both sides are defined.
\end{theo}
If $(\al \, ; \be)$ are defined over $\Q(\zeta_N)$, then there exists (Proposition \ref{prop:extension-field-field-of-definition}) a gamma triple $(\ga,\delta,N)$ which represents them.  Thus, we get an extension of $F_q(\al \, ; \be \mid t)$ to all prime powers $q$ such that $q^\times$ is divisible by $N$, and $q$ is coprime to the common denominator of the $\al_i$'s and $\beta_i$'s.

Conjecturally, to any hypergeometric parameters $(\al \,; \be)$ defined over $K$, one expects to be able to attach a family of pure Chow motives  $H(\al \, ; \be \mid t)$ defined over $K$, such that the trace of Frobenius at a prime power $q$ acting on the $\ell$-adic realisation of $H(\al \, ; \be \mid t)$ is\footnote{The choice of the generator $\chi$ of the multiplicative characters of $\F_q^\times$ is not canonical. To make sense of this equality, one must choose generators of the multiplicative characters of $\F_q^\times$ for each prime power $q$ that is a power of the norm of a prime ideal of $K$, in a consistent way. See, for example, \cite[Section 5]{katz-another-look}.} the finite hypergeometric sum $F_q(\al \, ; \be \mid t)$.
Our result above provides the needed extension of the hypergeometric trace when the field of definition is cyclotomic. We refer the reader to Section \ref{section:hypergeometric-function} for details. 

Let  $(\al \,; \be)$ be hypergeometric parameters such that $\sum_{i=1}^{n} (\al_i - \be_i)  \in \frac{1}{2} \Z$. Suppose that $(\al \, ; \be)$ are not defined over $\Q$. Then, there exists a gamma triple $(\ga,\delta,N)$ such that {the} greatest common divisor of the maximal minors of
\[
G= \begin{bmatrix}
\ga_1 &\ga_2 & \dots &\ga_{d+2} \\
\delta_1& \delta_2 &  \dots & \delta_{d+2} 
\end{bmatrix}
\]
is $1$, $\sum_{j=1}^{d+2} \delta_j =0$, and {$(\ga,k\delta,N)$ represents  $(\al \, ; \be)$} for some positive integer $0<k<N$.
For $t \in \F_q^\times$, consider the variety $V_{t}$ in $\Ps^{d+1} \times \G_m$ defined by the homogeneous (in the $w_i$'s) equations 
\begin{align*}
    w_1 + \dots +w_{d+2} &= 0, \\
    w_1^{\ga_1}\cdots w_{d+2}^{\ga_{d+2}} &= t,\\
    w_1^{\delta_1} \cdots w_{d+2}^{\delta_{d+2}} &= z^{N},
\end{align*}
and $ w_1\cdots w_{d+2} \neq 0,$ where $w_1, \dots, w_{d+2}$ are homogeneous coordinates on $\Ps^{d+1}$ and $z$ is a coordinate on $\G_m$. 
Let $r$, respectively $s$, be the number of negative, respectively positive, components of $\ga$.
For a vector $v \in \Z^{k}$, we write $g(v) \coloneqq \prod_{{i}=1}^{k} g(v_i)$.
Our second result is the following:
\begin{theo}\label{theo:intro-realisation}
Let $q$ be a prime power that is coprime to $\ga_1 \cdots\ga_{d+2}$ and such that $N$ divides $q^\times$. There exists a compactification $\overline{V_{t}}$ of $V_{t}$ such that
   \begin{align*}
    \#\overline{V_{t}}(\F_q) &=  \sum_{k=0}^{\min(r,s)-1} {\binom{r-1}{k}}{\binom{s-1}{k}} \frac{q^{d-k} - q^{k} }{q-1} \\
    &-\sum_{j=0}^{N-1} q^{s_{j\delta}(0)-1} g\left(j\frac{q^\times}{N} \delta \right) F_q\left(\gamma, j \delta, N \mid \frac{t}{\gamma^\gamma}\right).
    \end{align*}
\end{theo}

This gives a realisation of $F_q(\al ;\be | t/\ga^\ga)$ in $\overline{V}_t$ whenever $\sum_{i=1}^{n} (\al_i - \be_i) \in \frac{1}{2}\mathbb{Z}$. This condition is satisfied, for instance, whenever the hypergeometric parameters $(\al \,; \be)$ are defined over a totally real number field. 
Furthermore, it is equivalent to the special eigenvalue of the monodromy matrix at $1$ of the irreducible hypergeometric differential equation of the same parameters being $\pm 1$ (see \cite[Equation (2.8), Proposition 2.10]{Beukers-heckman}).

We obtain Theorem \ref{theo:intro-realisation} as an application of a systematic study of hypersurfaces in algebraic tori defined by linear combinations of $d+2$ monomials:
let $Z$ be the hypersurface in  the algebraic torus $\G_m^{d}= \Spec \F_q[x_1^{\pm},\dots,x_{d}^{\pm}]$ defined by the vanishing of a Laurent polynomial 
\[
f(x) = \sum_{j=1}^{d+2} u_j \prod_{i=1}^{d}x_i^{m_{ij}}, \qquad u_{j} \in \F_q^\times.
\]
We study such hypersurfaces in detail in Section \ref{section:source-varieties}.

All the point counting results of our paper are stated in terms of the following data (see the beginning of Section \ref{section:compactifications} for a summary of their definitions): 
\begin{enumerate}
    \item The \emph{Gale dual} $\gamma \in \Z^{d+2}$ of the Laurent polynomial $f$.
    \item A subset $\Lambda(q) \subset (\Z/q^\times\Z)^{d}$, together with a vector $\delta(\lambda) \in (\Z/q^\times\Z)^{d+2}$ for each $\lambda \in \Lambda(q)$.
    \item The constant $t = \prod_{j=1}^{d+2} u_j^{\gamma_j} \in \F_q^\times$, and additional constants $\sigma_1,\dots,\sigma_{d} \in \F_q^\times$  which record the coefficients of $f$.
\end{enumerate}
We restrict to the cases where the Gale dual $\ga$ has no zero components, and write 
$r$, respectively $s$, for the number of its negative, respectively positive, components. We study two compactifications of the hypersurface $Z$.
\subsection*{Compactification I }
Let $\Delta$ be the Newton polytope of the Laurent polynomial $f$, which we assume to be $d$-dimensional.  
We consider the closure $\overline{Z}$ of $Z$ in the toric variety $\Ps_{\Delta}$ associated to the Newton polytope $\Delta$. 
For a gamma triple $(\ga,\delta,N)$, define $\eta(m)=\eta_{\ga,\delta,N}(m)$ to be $1$ if $\exp(2\pi i m/q^{\times})$ is a root of $D_{\delta}(T)$ and $0$ otherwise. 
For two vectors $\lambda \in \left(\Z/q^\times\Z\right)^{d}, \sigma \in (\F_q^\times)^{d},$ we write $\chi^{\lambda}(\sigma)= \prod_{i=1}^{d}\chi^{\lambda_i}(\sigma_i)$.
Our main result for this compactification, which we prove in Subsection \ref{subsection:compactification-1}, is the following:
\begin{theo}\label{theo:intro-compcatification-1}
Let $\overline{Z}$ be the closure of the variety $Z$ in the toric variety $\Ps_{\Delta}$. Then,
\begin{align*}
    \#\overline{Z}(\F_q)&= \frac{q^{r+s-1}-1}{q-1}+ (q-1)^{r+s-2} -q^{r-1}(q-1)^{s-1}- q^{s-1}(q-1)^{r-1}\nonumber\\
     &+\sum_{\lambda \in \Lambda(q)}  \frac{\chi^{-\lambda}(\sigma)}{q-1}\sum_{m=0}^{q-2}q^{\eta_{\delta(\lambda)}(-m)-1} \prod_{j=1}^{d+2}g\left(-m\gamma_j + \delta(\lambda)_j\right) \chi^{m}\left(t\right).
\end{align*}
\end{theo}

As corollaries, we find some point counting results in terms of finite hypergeometric sums: 
let $q$ be a prime power that is coprime to $\ga_1\dots \ga_{d+2}$.
\begin{corol}\label{theo:intro-corol-1}
Suppose that $r=1$ and $s={d+1}$. Then,
  \begin{align*}
      \#\overline{Z}(\F_q)&= \frac{q^{d}-1}{q-1}-\sum_{\lambda \in \Lambda(q)} \chi^{-\lambda}(\sigma) q^{s_{\delta(\lambda)}(0)-1} g(\delta(\lambda))  F_q\left(\gamma,\delta(\lambda), q^\times \biggm\vert \frac{t}{\gamma^\gamma}\right).
  \end{align*}
\end{corol}
\begin{corol}\label{theo:intro-corol-2}
Suppose that $r=s=2$. Then,
\begin{align*}
      \#\overline{Z}(\F_q)&= q+1 -\sum_{\lambda \in \Lambda(q)} \chi^{-\lambda}(\sigma) q^{s_{\delta(\lambda)}(0)-1} g(\delta(\lambda))  F_q\left(\gamma,\delta(\lambda), q^\times\biggm\vert \frac{t}{\gamma^\gamma}\right).
  \end{align*}
\end{corol}
\begin{corol}\label{theo:intro-corol-3}
Suppose that $r=2$ and $s=3$. Then,
\begin{align*}
      \#\overline{Z}(\F_q)&= q^2+3q +1  -\sum_{\lambda \in \Lambda(q)} \chi^{-\lambda}(\sigma) q^{s_{\delta(\lambda)}(0)-1} g(\delta(\lambda))  F_q\left(\gamma,\delta(\lambda), q^\times \biggm\vert \frac{t}{\gamma^\gamma}\right).
  \end{align*}
\end{corol}
In particular, whenever the hypersurface $Z$ is a curve or a surface, the point count of the compactification $\overline{Z}$ in $\Ps_{\Delta}$ is given by finite hypergeometric sums.  Furthermore, in the cases of the corollaries, for all but finitely many primes $p$, when $ t \not \equiv \ga^{\ga} \pmod{p}$, the compactification $\overline{Z}$ is quasi-smooth, i.e., it only has finite quotient singularities. 
However, in general, the compactification $\overline{Z}$ is neither quasi-smooth nor can we write $\#\overline{Z}(\F_q)$ in terms of finite hypergeometric sums. To remedy this deficiency, we study a second compactification:
\subsection*{Compactification II} We define a partial desingularisation $\Ps'$ of $\Ps_{\Delta}${, and} consider the closure $W$ of $Z$ in $\Ps'$. 
Our main theorem for this compactification, which we prove in Subsection \ref{subsection:compactification-2}, is the following:
\begin{theo}\label{theo:intro-compactification-II}
Let $W$ be the closure of $Z$ in the toric variety $\Ps'$. Then,
\begin{align*}
\#W (\mathbb{F}_q) &= \sum_{k=0}^{\min(r,s)-1} {\binom{r-1}{k}}{\binom{s-1}{k}} \frac{q^{d-k} - q^{k} }{q-1}\\
&-\sum_{\lambda \in \Lambda(q)}\chi^{-\lambda}(\sigma) q^{s_{\delta(\lambda)} (0)-1} g\left(\delta(\lambda)\right) F_q\left(\gamma,\delta(\lambda),q^\times \biggm\vert \frac{t}{\gamma^\gamma}\right),
\end{align*}
whenever $q$ is a prime power that is coprime to $\ga_1\dots \ga_{d+2}$.
\end{theo}
Thus, $\#W(\F_q)$ is given in terms of finite hypergeometric sums and, unlike $\overline{Z}$, when $Z$ is smooth, $W$ is quasi-smooth for all but finitely many primes. The finite hypergeometric sums appearing in Theorem \ref{theo:intro-compactification-II} and Corollaries \ref{theo:intro-corol-1}, \ref{theo:intro-corol-2}, \ref{theo:intro-corol-3} above are no longer necessarily defined over $\Q$. It turns out that the only gamma triples $(\ga, \delta,N)$ which appear in such families satisfy
$ \sum_{j=1}^{d+2} \delta_j \equiv 0 \pmod{q^\times}$. Consequently, the associated hypergeometric parameters $(\al \, ; \be)$ {satisfy} $\sum_{i=1}^n (\al_i - \be_i) \in \frac{1}{2} \Z$, see Proposition \ref{prop:characterization-of-gamma-triple}. 

\subsection{Application: the Dwork family}
Theorem \ref{theo:intro-compcatification-1} can be used to recover many point counting results for hypersurfaces in projective space. We give, as an example, the renowned Dwork family $X_u$ defined by
\[
y_1^{d+1}+\dots + y_{d+1}^{d+1} - u^{-1} (d+1) y_1 \cdots y_{d+1} = 0
\]
in the projective space $\Ps^d$ with homogeneous coordinates $y_1, \dots, y_{d+1}$.
Our main result for this family is the following:
\begin{theo}\label{theo:Dwork}
Let $q$ be a prime power such that $\gcd(q,d+1)=1$, and let $e =\gcd(q-1,d+1)$. 
Let $\gamma= (-d-1,1,\dots,1) \in \Z^{d+2}$.
For a vector $\lambda =(\lambda_1,\dots,\lambda_{d-1}) \in (\Z/e\Z)^{d-1}$, let 
\[
\delta(\lambda) = (0,\lambda_1,\dots,\lambda_{d-1},-\lambda_1-\dots -\lambda_{d-1},0).
\]
Then, the number of $\F_q$-points of $X_u$ is given by
\[
\#X_u(\F_q) = \frac{q^{d}-1}{q-1} - \sum_{\lambda \in (\Z/e\Z)^{d-1}} g\left(\frac{q-1}{e}\delta(\lambda) \right) F_q\left(\gamma,\delta(\lambda), e \bigm\vert u^{d+1}\right).
\]   
\end{theo}
For each $\lambda \in \left(\Z/e\Z\right)^{d-1}$, the hypergeometric parameters $(\al \,;\be)$ represented by the gamma triple $(\ga,\delta(\lambda),e)$ which {appears} in Theorem \ref{theo:Dwork} can be computed by the following appealing procedure: let 
\begin{equation}\label{intro-Q-quotient}
Q= \frac{T^{d+1}-1}{(T- \zeta_{e}^{-\lambda_1}) \dots (T- \zeta_{e}^{-\lambda_{d-1}}) (T- \zeta_{e}^{\lambda_1+\dots +\lambda_{d-1}}) (T-1)}.
\end{equation}
Then, after cancellation, 
\[
Q = \frac{\prod_{j=1}^{n} \left(T- e^{2\pi i \al_j}\right) }{\prod_{j=1}^{n} \left(T- e^{2\pi i \be_j} \right)}.
\]
We note that $Q$ can be equal to $1$ for some $\lambda$, in which case $n=0$.
This happens only if $e=d+1$ and $d+1$ is odd; in that case, the number $N_{d+1}$ of such $\lambda$ is $d!$.
We remark that, heuristically, $N_{d+1}$ should be equal to the rank of the phantom cohomology $H^{d}_{ph}(X_1(\C),\C)$  for the restriction of the family to a small neighbourhood of $u=1$ (see \cite[Equation (2)]{matt-kerr}). 

In fact, using our results, we can handle the more general family 
\[
 y_1^{N}+\dots + y_{d+1}^{N} - u^{-1} (d+1) y_1^{n_1} \cdots y_{d+1}^{n_{d+1}} = 0
\]
in the projective space $\Ps^{d}$, where $n_1,\dots,n_{d+1} \ge 1$ and $n_1+\dots + n_{d+1}=N$. 
However, the results become awkward to state. We refer the reader to \cite{Dermot-Mccarthy} for results in this direction obtained using different methods. 

We give a few examples of Theorem \ref{theo:Dwork} to illustrate its concreteness. We write the results in terms of hypergeometric parameters, rather than gamma triples, for elegance. We use the classical notation  $F_q\biggl(\begin{matrix} \al_1, \ \ldots, \ \al_n \\ \be_1, \ \ldots, \ \be_n \end{matrix} \:\bigg|\:\: t\biggr) = F_q(\al \, ; \be \mid t)$.
\begin{theo}[$d=2$]\label{theo:dwork-2}
Let $X_u$ be the curve in $\Ps^2$ defined by
\[
y_1^{3}+y_2^{3} + y_3^3 -3u^{-1} y_1y_2y_3 =0.
\]
\begin{enumerate}
    \item   If $q\equiv 2 \pmod{3}$, then
    \begin{align*}
    \#X_u(\F_q) &=  q+1 - F_q\biggl(\begin{matrix} \frac{1}{3},  \frac{2}{3} \\ 1,   1 \end{matrix} \:\bigg|\:\: u^{3} \biggr).
    \end{align*}
    \item If $q\equiv 1 \pmod{3}$, then 
    \begin{align*}
    \#X_u(\F_q) &=  q+1 - F_q\biggl(\begin{matrix} \frac{1}{3},  \frac{2}{3} \\ 1,  1 \end{matrix} \:\bigg|\:\: u^{3} \biggr) - 2qF_q\biggl(\begin{matrix} \emptyset \\ \emptyset \end{matrix} \:\bigg|\:\: u^{3} \biggr) \\
    &=  q+1 - F_q\biggl(\begin{matrix} \frac{1}{3},  \frac{2}{3} \\ 1,  1 \end{matrix} \:\bigg|\:\: u^{3} \biggr) -2q \begin{cases}
    -1 &\text{ if } u^3 =1,\\
    0 &\text{ if } u^3 \neq 1. 
    \end{cases}
    \end{align*}
\end{enumerate}
\end{theo}
The last formula can be explained by the fact that when $u^3=1$, the fibres $X_u$ are reducible with three components. For instance, if $u=1$, then
\[
y_1^{3}+y_2^{3} + y_3^3 -3 y_1y_2y_3= (y_1+y_2+y_3)(y_1+\zeta y_2+\zeta ^2y_3)(y_1+\zeta ^2y_2+\zeta y_3),
\]
where $\zeta $ is a primitive third root of unity in $\F_q$, which exists since $q\equiv 1 \pmod{3}$. Thus, $X_1$ is the union of three projective lines. It is easy to check that any two of these lines meet at a single point. Moreover, the three points of intersection are $\F_q$-points, namely $[1,1,1],[\zeta ,\zeta^{2},1]$ and $[\zeta^{2},\zeta ,1]$. Thus, $\#X_1(\F_q) = 3 (q+1) -3 = 3q$. Note that this also gives the following evaluation
\[
F_q\biggl(\begin{matrix} \frac{1}{3},  \frac{2}{3} \\ 1,  1 \end{matrix} \:\bigg|\:\: 1 \biggr) =1, \qquad  \text{ when } q \equiv 1 \pmod{3}.
\]

\begin{theo}[$d=3$]\label{theo:dwork-3}
Let $X_u$ be the K3 surface in $\Ps^{3}$ defined by
\begin{align*}
  y_1^{4}+y_2^{4}+ y_3^{4} + y_{4}^{4} - 4 u^{-1} y_1y_2y_3y_{4}=0.
\end{align*}
\begin{enumerate}
    \item If $q\equiv 3 \pmod{4}$, then
    \begin{align*}
    \#X_u(\F_q)
    &= q^2+q+1 +  F_q\biggl(\begin{matrix} \frac{1}{4}, \frac{2}{4}, \frac{3}{4} \\ 1,  1, 1 \end{matrix} \:\bigg|\:\: u^{4} \biggr)- 3 q \cdot  F_q\biggl(\begin{matrix} \frac{1}{4},  \frac{3}{4} \\ \frac{1}{2},  1 \end{matrix} \:\bigg|\:\: u^{4} \biggr).
    \end{align*}
    \item If $q\equiv 1 \pmod{4}$, then
    \begin{align*}
     \#X_u(\F_q) 
     &= q^2+q+1 + F_q\biggl(\begin{matrix} \frac{1}{4}, \frac{2}{4}, \frac{3}{4} \\ 1,  1, 1 \end{matrix} \:\bigg|\:\: u^{4} \biggr)+ 3q \cdot F_q\biggl(\begin{matrix} \frac{1}{4},  \frac{3}{4} \\ \frac{1}{2},  1 \end{matrix} \:\bigg|\:\: u^{4} \biggr) \\
     &+12\chi(-1)^{q^{\times}/4} q\cdot  F_q\biggl(\begin{matrix} \frac{1}{2} \\ 1 \end{matrix} \:\bigg|\:\: u^{4} \biggr).
    \end{align*}
\end{enumerate}
\end{theo}
Theorem \ref{theo:dwork-3}, in this form, appeared before in \cite[Theorem 1.1]{goodson-hypergeometric-functions} when $q\equiv 1 \pmod{4}$, and for general prime powers in \cite[Proposition 3.4.1]{Salerno-Kelly}.

The next case ($d=4$) is more difficult to state, since the number of finite hypergeometric sums appearing in the point count is $5^3$. Nonetheless, Theorem \ref{theo:Dwork} provides the parameters of the finite hypergeometric sums that appear in the point counts for all prime powers coprime to $5$. The number of $\F_q$-points in this case was studied extensively in the literature in relation to {physics and mirror symmetry} (see, for instance, \cite{villegas-candelas-i} and \cite{Villegas-Candelas-Ossa}). 

We prove Theorem \ref{theo:Dwork} in Section \ref{section:dwork}. 
Theorems \ref{theo:dwork-2} and \ref{theo:dwork-3} are special cases of Theorem \ref{theo:Dwork}. To illustrate the computations involved, we give the proof of Theorem \ref{theo:dwork-3} in Section \ref{section:dwork}.

The question of realising finite hypergeometric sums as point counting functions of families of algebraic varieties has been studied extensively in recent years (see, for instance, \cite{BCM}, \cite{ams-finite-fields-hypergeometric}, \cite{voight-kelly}).
The significance of this problem lies partly in the fact that computing $L$-functions of algebraic varieties, especially those of high dimension, is computationally expensive, which makes it difficult to test conjectures for these $L$-functions in practice. 
Finite hypergeometric sums offer access to a big class of $L$-functions in a computationally feasible manner. Software packages for computing hypergeometric $L$-functions have been developed in PARI \cite{PARI2}, MAGMA \cite{MAGMA}, and Sage \cite{sagemath} through the contributions of Watkins \cite{watkins}, Rodriguez-Villegas, Cohen, Costa, Kedlaya, Roe \cite{kedlaya-costa-roe-1}, \cite{kedlaya-costa-roe-2}, and others. For a more detailed discussion of the motivic perspective and its relation to $L$-functions, we refer the reader to \cite{hypergeometric-survey}.

The computations of this paper have been implemented by the author in 
Sage \cite{sagemath}.
\subsection*{Structure of the paper}

In Section \ref{section:hypergeometric-function}, we recall the definitions of Gauss and finite hypergeometric sums.  
We introduce the notion of gamma triples to encode hypergeometric parameters, and study some of their properties. 
We define finite hypergeometric sums associated to gamma triples and relate them to the classical definition of finite hypergeometric sums.
In Section \ref{section:source-varieties}, we study a general hypersurface in an algebraic torus of dimension $d$ defined by a linear combination of $d+2$ monomials. We show that such a hypersurface is a cover of a ``primitive" hypersurface of the same form.  We count the number of $\F_q$-points on the restrictions of such a hypersurface to the faces of its Newton polytope. We end the section by studying some of its regularity properties.
In Section \ref{section:combinatorics}, we describe the combinatorics of the Newton polytope of such a hypersurface, its normal fan, and a refinement of the normal fan based on the construction of Beukers, Cohen, and Mellit.
In Section \ref{section:compactifications}, we count the number of $\F_q$-points of the {closure} of such a hypersurface in the toric variety associated to its Newton polytope; thus proving Theorem \ref{theo:intro-compcatification-1}. We also count the number of $\F_q$-points on the closure of $Z$ in the toric variety associated to the refinement constructed in Section \ref{section:combinatorics}; thus proving Theorem \ref{theo:intro-compactification-II}.
In Section \ref{section:reverse-engineering}, we prove Theorem \ref{theo:intro-realisation} and give two examples to illustrate how gamma triples give rise to different realisations of the same hypergeometric sum.
In Section \ref{section:dwork}, we apply the results of this paper to the Dwork family. 

\addtocontents{toc}{\protect\setcounter{tocdepth}{0}}
\section*{Acknowledgements}
\addtocontents{toc}{\protect\setcounter{tocdepth}{\the\value{tocdepth}}}
I am grateful to my PhD supervisor, Fernando Rodriguez Villegas, for many enlightening and inspiring conversations throughout my PhD. 
I am grateful to John Voight for reading parts of an earlier version of this manuscript and offering some corrections and valuable suggestions that improved the overall exposition. 
I want to thank Giulia Gugiatti, Matt Kerr, Tolibjon Ismoilov, Sohaib Khalid, and Vadym Kurylenko for many helpful discussions.  During my PhD, I was supported by the Simons Foundation through Award 284558FY19 to the ICTP.

\section{Finite hypergeometric sums over finite fields}\label{section:hypergeometric-function}
Recall that we have fixed a non-trivial additive character $\psi: \F_q \to \C^\times$ {and} a multiplicative character $\chi: \F_q^\times \to \C^\times$ of order $q^\times$. 

\begin{lem} \label{lem:orthogonal} We have the following orthogonality relations:
\begin{equation*}
    \sum_{t \in \F_q} \psi(tx) = \begin{cases}
        q \quad & \text{ if } x =0,\\
        0 \quad & \text{ if } x\neq 0.
    \end{cases}
\end{equation*}
\begin{equation*}
   \sum_{m=0}^{q^\times-1} \chi^{m}(x) =   \begin{cases}
        q^\times \quad & \text{ if }x =1,\\
        0 \quad & \text{ if } x \neq 1.
        \end{cases}
\end{equation*}
\end{lem}

\begin{defi}[Gauss sums] \label{defi:gauss-sums}
For $m \in \Z$, the \emph{Gauss sum} $g(m)$ is defined by
\[
g(m) \coloneqq \sum_{u \in \F_q^\times } \psi(u) \chi^m(u).
\]
For a vector $v \in \mathbb{Z}^{k}$, we define $g(v)\coloneqq \prod_{i=1}^k g(v_i)$.
\end{defi}
Note that the Gauss sum $g(m)$ lies in the field $\Q(\zeta_{p},\zeta_{q^\times})$, where $q= p^{k}$ for $p$ prime.
\begin{rem}[Dependence on $\psi$] \label{rem:dependence-gauss-sums}
The Gauss sum $g(m)$ depends on the additive character $\psi$ as follows:
Let $\psi':\F_q \to \C^\times$ be another non-trivial additive character. By  \cite[Proposition 2.5.4]{Cohen-book}, there is $a \in \F_q^{\times}$ such that $\psi'(x) = \psi(ax)$ for all $x$. Thus,
    \begin{align*}
        \sum_{u \in \F_q^\times } \psi'(u) \chi^m(u) &= \sum_{u \in \F_q^\times } \psi(au) \chi^m(u)\\
        &= \chi^{-m}(a)\sum_{u \in \F_q^\times } \psi(au) \chi^m(au) \\
        &= \chi^{-m}(a) \sum_{x \in \F_q^\times } \psi(x) \chi^m(x) =\chi^{-m}(a)g(m).
    \end{align*}

The sums we will consider in this paper do not depend on the choice of additive character $\psi,$ so we have chosen to suppress it from our notation for the Gauss sum.
\end{rem}

\begin{lem}\label{lem:gauss-sums-properties} Gauss sums satisfy the following properties:
    \begin{enumerate}
        \item $g(0)= -1$,
        \item $g(-m)= \overline{g(m)} \chi^{m}(-1)$,
        \item for $m \not\equiv 0 \pmod{q^\times}$, we have $g(m)g(-m) = \chi^m(-1) q$. 
    \end{enumerate}
\end{lem}
\begin{proof}
    See Lemma 2.5.8 and Proposition 2.5.9 in \cite{Cohen-book}.
\end{proof}
\begin{theo}[Hasse--Davenport]\label{Hasse-Davenport}
    Let $N$ be a positive integer dividing $q^\times$. For any integer $m,$ we have
    \[
    g(Nm) = - \chi(N)^{Nm} \prod_{j=0}^{N-1} \frac{g\left(m+j \frac{q^{\times}}{N}\right)}{g\left(j \frac{q^{\times}}{N}\right)} .
    \]
\end{theo}
\begin{proof}
  See \cite[Theorem 3.7.3]{Cohen-book}.
\end{proof}
Next, we define finite hypergeometric sums. For a review of their history, we refer the reader to the introductions of \cite{BCM} and \cite{voight-kelly}. {Recall that we call $(\al \, ; \be)$ \emph{hypergeometric parameters} if  $\al_i-\be_j$ is not an integer for any $i$ and $j$.}

\begin{defi}[Finite hypergeometric sums]\label{defi:finite-hg-function}
Let $(\alpha \, ; \beta) \in \Q^n \times \Q^n$, $n \ge 0$, be hypergeometric parameters and suppose that $q^\times \alpha, q^\times \beta \in \Z^n$.
For any $t \in \F_q^\times$, define
\[
F_q(\al\,;\be \mid t) \coloneqq \frac{1}{1-q} \sum_{m=0}^{q^\times-1} \prod_{i=1}^n \frac{g(m+\alpha_iq^\times)g(-m-\beta_iq^\times)}{g(\alpha_iq^\times)g(-\beta_i q^\times)} \chi((-1)^nt)^m .
\]
The sum $F_q(\al\,;\be \mid t)$ is called the \emph{finite hypergeometric sum} with parameters $(\al \, ;\beta)$.
\end{defi}
The finite hypergeometric sum $F_q(\al \, ; \be \mid t)$ does not depend on the choice of additive character $\psi$. Indeed, by Remark \ref{rem:dependence-gauss-sums}, changing the additive character $\psi$ to $x \mapsto \psi(ax)$, for some $a \in \F_q^\times,$ changes
\[
\frac{g(m+\alpha_iq^\times)g(-m-\beta_iq^\times)}{g(\alpha_iq^\times)g(-\beta_i q^\times)}
\]
by multiplication by
\[
\frac{\chi^{-m-\alpha_iq^\times}(a)\chi^{m+\beta_iq^\times}(a)}{\chi^{-\alpha_iq^\times}(a)\chi^{\beta_i q^\times}(a)} =1.
\]
Thus, for any $t \in \F_q^\times$, we have $F_q(\al \, ; \be \mid t) \in \mathbb{Q}(\zeta_{q^\times}\!)$.

\begin{rem}
Definition \ref{defi:finite-hg-function} above is equivalent to the definition of \cite[Equation 10.1] {hypergeometric-survey}. Indeed, by part 2 of Lemma \ref{lem:gauss-sums-properties}, we have 
\begin{align*}
    \prod_{i=1}^n \frac{g(m+\alpha_iq^\times)\overline{g(m+\beta_iq^\times)}}{g(\alpha_iq^\times)\overline{g(\beta_i q^\times)}}= \chi((-1)^n)^m \prod_{i=1}^n \frac{g(m+\alpha_iq^\times)g(-m-\beta_iq^\times)}{g(\alpha_iq^\times)g(-\beta_i q^\times)} .
\end{align*}
\end{rem}

The conditions on the prime powers $q$ in Definition \ref{defi:finite-hg-function} are quite restrictive. One of the main goals of this section is to find a different formula for this sum that eliminates some of the congruence conditions on $q$.

\subsection{Gamma triples}
We introduce a way to encode the hypergeometric parameters $(\alpha \, ;  \beta)$.  
\begin{defi}\label{def:gamma-triple}
    A \emph{gamma triple}
    \footnote{
    The notion of gamma triples is a generalisation of the notion of partial twists, which is the subject of a work in progress, joint with Giulia Gugiatti and Fernando Rodriguez Villegas \cite{partial-twists}.} 
    is a triple $(\gamma,\delta,N)$ such that 
    \begin{enumerate}
    \item $\gamma \in \Z^{d+2}_{\neq0}$  is a gamma vector,
    \item $\delta \in \Z^{d+2}$ is any vector, and
    \item $N$ is a positive integer.
\end{enumerate}
\end{defi}

\begin{defi}
    To a gamma triple $(\gamma,\delta,N)$, we associate hypergeometric parameters $(\al \,;\be) \in (\Q \cap (0,1])^n\times (\Q \cap (0,1])^n $, by writing
    \begin{equation}\label{eq:gamma-delta}
    \frac{\prod_{j=1}^n (T- e^{2\pi i \alpha_j})}{\prod_{j=1}^n (T- e^{2\pi i \beta_j}) } = \frac{ \prod_{\ga_j< 0} (T^{-\ga_j} - \zeta_N^{\delta_j})}{ \prod_{\ga_j>0} (T^{\ga_j} - \zeta_N^{-\delta_j})} ,
    \end{equation}
    where the greatest common divisor of the numerator and the denominator of the left-hand side is $1$. 
    The parameters $(\al \,;\be)$ are called the hypergeometric parameters associated to the gamma triple $(\gamma,\delta,N)$, and we say that $(\ga,\delta,N)$ represents $(\al \,;\be)$.
\end{defi}

We note that Equation \eqref{eq:gamma-delta} defines $(\al\, ;\be)$ only modulo $\Z$. {For convenience,} we choose $(\al\, ;\be)$ in $(\Q \cap (0,1])^n\times (\Q \cap (0,1])^n$. {Moreover,} $(\al \, ; \be)$ only depend on $\delta$ modulo $N$.

\begin{exa}\label{example-1}
Given any hypergeometric parameters $(\al \,;\be)$, there exists a gamma triple $(\ga,\delta,N)$ that represents them. 
Indeed, if $\al, \be $ are empty, then $(\ga,\delta,N) = ((1,-1),(0,0),1)$ 
represents $(\al,\be)$.
Otherwise, let $N$ be the least common multiple of the denominators of $\al_1,\dots,\al_n, \be_1,\dots,\be_n$, and let
\begin{align*}
    \ga&=(-1,\dots,-1;1,\dots,1) \in \Z^{n}\times \Z^n, \\
    \delta&= (N\al_1,\dots,N\al_n,-N\be_1,\dots,-N\be_n) \in \Z^{2n}.
\end{align*}
Then, $(\ga,\delta,N)$ represents $(\al \,;\be)$.
\end{exa}

The representation from Example \ref{example-1} is not unique.  For instance, since we have defined $\zeta_k$ to be the complex number $e^{2\pi i/k}$ rather than any primitive $k$-th root of unity, the gamma triples $(\gamma,\delta,N)$ and $ (\gamma,k \delta,kN)$ define the same hypergeometric parameters $(\al \,;\be)$ for any positive integer $k$. In some cases, we have more interesting representations:

\begin{exa}\label{exa:example-2}
The hypergeometric parameters $((1/3,2/3);(1,1))$ can be represented by the triples
\begin{align*}
    ((-3,1,1,1),(0,0,0,0),1), \\
    ((-1,-1,1,1),(1,-1,0,0),3).
\end{align*}
\end{exa} 
Let $\HH(\al \,;\be)$ be the irreducible hypergeometric local system (or monodromy representation) on $\Ps^1(\C)\setminus\{0,1,\infty\}$ associated to the hypergeometric parameters $(\al \,;\be)$ (see \cite[Equation~2.14]{Beukers-heckman}). The local system $\HH(\al \,;\be)$ is defined over the field obtained from $\Q$ by adjoining the coefficients of the numerator and the denominator of 
\begin{equation}
    Q= \frac{\prod_{j=1}^n (T- e^{2\pi i \alpha_j})}{\prod_{j=1}^n (T- e^{2\pi i \beta_j}) } \label{eq:quotient-Q}
\end{equation}
(see \cite[Corollary 3.6]{Beukers-heckman}). This motivates the following definition:

\begin{defi}
    Let $(\al\,;\be)$ be hypergeometric parameters. Let $K$ be the field obtained from $\Q$ by adjoining the coefficients of the numerator and the denominator of the right-hand side of Equation \eqref{eq:quotient-Q}. We call $K$ the field of definition of $(\al \, ; \be)$.
    For any field $L$ which contains $K$, we say that $(\al \, ; \be)$ can be defined over $L$.
\end{defi}

The third component of a gamma triple $(\gamma, \delta, N)$ records the cyclotomic field of definition of the associated hypergeometric parameters in the following sense:

\begin{prop}\label{prop:extension-field-field-of-definition}
If $(\al \,;\be)$ are associated to a gamma triple $(\gamma,\delta,N)$, then $(\al \,;\be)$ can be defined over $\Q(\zeta_N)$.  Conversely, if $(\al \,;\be)$ are {hypergeometric parameters} which can be defined over $\Q(\zeta_N)$ for some $N$, then there exists a gamma triple $(\gamma,\delta,N)$ which represents them.
\end{prop}

\begin{proof}
To prove the first statement, suppose $(\ga,\delta, N)$ is a gamma triple and $(\al \,; \be)$ are the associated hypergeometric parameters. Let $P/Q$ and $P'/Q'$ denote the left-hand and right-hand sides of Equation \eqref{eq:gamma-delta}, respectively. For any $\sigma \in \Gal(\overline{\Q}/\Q(\zeta_N))$, we have
\[
\sigma(P)/\sigma(Q)=\sigma(P/Q) = \sigma(P'/Q') = P'/Q' = P/Q.
\]
Since $P$ and $Q$ are monic, $\gcd(P, Q) = 1$ and $\gcd(\sigma(P), \sigma(Q)) = 1$, it follows that $\sigma(P) = P$ and $\sigma(Q) = Q$. Thus, the field of definition $K$ is contained in $\Q(\zeta_N)$.

The second statement follows by applying Part (a) of the next lemma to the numerator and the denominator of the right-hand side of Equation \eqref{eq:quotient-Q}.
\end{proof}

For a positive integer $n$ and a root of unity $x$, let $\nu_n(x)$ be the least integer $a \ge 1$ such that $x^a\in\mu_n$. Then, $x^b \in \mu_n$ if and only if $\nu_n(x) \mid b$. We call a binomial $T^a-\zeta_n^e$ \emph{reduced} if $\gcd(a,e,n)=1$. 

\begin{lem}\label{lem:generalizing-cyclotomics}
Let $n$ be a positive integer.
\begin{enumerate}
    \item[(a)] Let $A(x) \in \Q(\zeta_n)[x]$ be a monic polynomial whose roots are roots of unity. Then,
    \begin{equation}\label{eq:lem-decompos}
        A(x) = \frac{\prod_{j=1}^{l} \left(x^{k_j}-\zeta_n^{r_j}\right)}{\prod_{j=l+1}^{l+l'} \left(x^{k_j}-\zeta_n^{r_j}\right)}
    \end{equation}
    for some non-negative integer $l'$ and positive integers $l, k_1,\dots, k_{l+l'},r_1, \dots,r_{l+l'}$ such that
    \begin{enumerate}[label=(\arabic*), ref=\arabic*]
        \item $\sum_{j=1}^{l} k_j -  \sum_{j=l+1}^{l+l'} k_j= \deg A(x)$, \label{condition-1}
        \item $\gcd(k_1,\dots, k_{l+l'})=1$, and \label{condition-2}
        \item each $k_j$ is either $1$ or divides the order of some root of $A(x)$. \label{condition-3}
    \end{enumerate}
    \item[(b)] Let $S_1$ and $S_2$ be finite multisets of pairs $(a,e)\in\Z_{\ge1}\times\Z/n\Z$ with $\gcd(a,e,n)=1$. If $\prod_{S_1}(T^a-\zeta_n^e)=\prod_{S_2}(T^a-\zeta_n^e)$, then $S_1=S_2$.
\end{enumerate}
\end{lem}

\begin{proof}
We first show that if $T^a-\zeta_n^e$ is reduced, then the set 
$S_a= \{y   \mid y^a -\zeta_n^e=0 \text{ and } \nu_n(y)=a\}$ forms a non-empty $\Gal(\overline{\Q}/\Q(\zeta_n))$-orbit, and every other root $y$ of $T^a-\zeta_n^e$ has $\nu_n(y)$ a proper divisor of $a$.
The roots of $T^a-\zeta_n^e$ are of the form $y=\zeta_{an}^E$ with $E \equiv e \pmod n$. Moreover, for such a root $y$, $\nu_n(y)=a/\gcd(a,E)$. In particular, $\nu_n(y) \mid a$ with equality if and only if $\gcd(a,E)=1$.
To prove that $S_a$ is non-empty, we find $E$ such that $E\equiv e \pmod{n}$ and $\gcd(E,a)=1$. Let $E= e+jn$, where $j$ is the product of primes that divide $a$ but not $e$. 
Let $p$ be a prime divisor of $a$. 
If $p$ divides $e$, then $p$ does not divide $n$, since $\gcd(a,e,n)=1$. By definition of $j$, $p$ does not divide $j$, so $p$ does not divide $E$. 
If $p$ does not divide $e$, then $p$ divides $j$, so $p$ does not divide $E$. 
Thus, $\gcd(E,a)=1$.
The set $S_a$ is $\Gal(\overline{\Q}/\Q(\zeta_n))$-stable. To show that the Galois group acts on $S_a$ transitively, let $E\equiv E'\equiv e\pmod n$ with $\gcd(a,E)=\gcd(a,E')=1$, and choose $u\equiv 1\pmod n$ with $u E\equiv E'\pmod{an}$, which is possible since $E$ is invertible modulo $a$. Then, $u$ is prime to $an$, because $E'$ is prime to $a$, so $\zeta_{an}\mapsto\zeta_{an}^u$ is an element of $\Gal(\Q(\zeta_{an})/\Q(\zeta_n))$ that maps $\zeta_{an}^E$ to $\zeta_{an}^{E'}$.

Now, to prove (a), let $\xi$ be a root of unity, and let $a=\nu_n(\xi)$. We first show that $\xi$ is a root of a reduced binomial $T^a-\zeta_n^e$. Since $\xi^a\in\mu_n$, we can write $\xi=\zeta_{an}^E$ for some integer $E$. By the formula above, $\gcd(a, E)=1$. Moreover, $\xi^a=\zeta_n^e$ with $e\equiv E\pmod n$. Thus, $\gcd(a,e,n)$ divides $\gcd(a,E)=1$, so $T^a-\zeta_n^e$ is reduced. 

Let $f_\xi$ be the minimal polynomial of $\xi$ over $\Q(\zeta_n)$. Since $\xi\in S_a$ and $S_a$ is a $\Gal(\overline{\Q}/\Q(\zeta_n))$-orbit, the roots of $f_\xi$ are exactly the elements of $S_a$. Therefore,
\[
f_\xi(T)=\frac{T^a-\zeta_n^e}{B(T)},
\]
where $B(T)\in\Q(\zeta_n)[T]$ is monic, and every root $y$ of $B(T)$ satisfies $\nu_n(y)\mid a$ and $\nu_n(y)<a$.

We now show, by induction on $a$, that $f_\xi$ can be written in the form \eqref{eq:lem-decompos} with every $k_j$ dividing $a$. If $a=1$, then $B(T)=1$, and there is nothing to prove. Now, suppose that the statement holds for all roots of unity $y$ with $\nu_n(y)<a$. The polynomial $B(T)$ is a product of minimal polynomials $f_y$ of roots $y$ of $B(T)$. By the induction hypothesis, each $f_y$ can be written in the form \eqref{eq:lem-decompos} with every $k_j$ dividing $\nu_n(y)$, and hence dividing $a$. Since the exponent of $T^a-\zeta_n^e$ is $a$, the same holds for $f_\xi$.

Finally, $A(x)$ is a product of minimal polynomials $f_\xi$ of roots $\xi$ of $A(x)$. Thus, $A(x)$ can be written in the form \eqref{eq:lem-decompos} with every $k_j$ dividing $\nu_n(\xi)$ for some root $\xi$ of $A(x)$. Since $\nu_n(\xi)$ divides the order of $\xi$, condition \eqref{condition-3} holds. Condition \eqref{condition-1} holds by comparing degrees. If $\gcd(k_1,\dots,k_{l+l'})\neq 1$, we multiply the numerator and the denominator by $x-1=x-\zeta_n^n$. This gives condition \eqref{condition-2}, and the new exponent $1$ is allowed by condition \eqref{condition-3}.

To prove (b), we argue by induction on the degree of $\prod_{S_1}(T^a-\zeta_n^e)$. If the degree is $0$, then $S_1$ and $S_2$ are both empty. Otherwise, choose $(a,e)\in S_1\cup S_2$ with $a$ maximal. Since $T^a-\zeta_n^e$ is reduced, it has a root $\xi$ with $\nu_n(\xi)=a$. Let $(a',e')\in S_1\cup S_2$ be such that $\xi$ is a root of $T^{a'}-\zeta_n^{e'}$. Then, $\xi^{a'}\in\mu_n$, so $a\mid a'$. By maximality of $a$, we have $a'=a$. Moreover, $\zeta_n^{e'}=\xi^{a}=\zeta_n^{e}$, so $e'=e$ in $\Z/n\Z$. Thus, $T^a-\zeta_n^e$ is the only binomial in $S_1\cup S_2$ vanishing at $\xi$. Since all the binomials have simple roots, the multiplicity of $\xi$ as a root of $\prod_{S_i}(T^a-\zeta_n^e)$ is the number of times $(a,e)$ occurs in $S_i$, for $i=1,2$. Hence, $(a,e)$ occurs equally often in $S_1$ and in $S_2$. Removing one copy of $(a,e)$ from each, the result follows by induction.
\end{proof}

We end this subsection by considering gamma triples $(\gamma, \delta, N)$ which satisfy $\sum_{j=1}^{d+2} \delta_j \equiv 0 \pmod{N}$. Such triples are precisely the ones that arise in the families of algebraic varieties we will consider in this paper.

\begin{prop}\label{prop:characterization-of-gamma-triple}
Let $(\gamma,\delta,N)$ be a gamma triple such that $\sum_{j} \delta_j \equiv 0 \pmod{N}$. If $(\al \,;\be)$ are the hypergeometric parameters associated to $(\gamma,\delta,N)$, then $\sum_{i} (\al_i - \be_i) \in \frac{1}{2} \Z$. 
Conversely, if $(\al \,;\be)$ are hypergeometric parameters such that $\sum_{i} (\al_i - \be_i) \in \frac{1}{2}\Z$, then $(\al \,;\be)$ are represented by a gamma triple $(\ga,\delta,N)$ with $\sum_{j} \delta_j\equiv 0 \pmod{N}$.
\end{prop}
\begin{proof}
Assume that $\ga \in \Z_{\neq0}^{r+s}$. Suppose, without loss of generality, that $\ga_j <0$ for $j\le r$ and $\ga_j> 0 $ for $j > r$.
To prove the first statement, note that $\sum_{i} (\al_i - \be_i)$ is congruent modulo $\Z$ to $\sum_{j,i} (\al_{j,i}- \be_{j,i})$ where $\{\al_{j,i}\}$ and $\{\beta_{j,i}\}$ are the parameters ``without cancellation'' given by 
\begin{align*}
\al_{j,i} &= \frac{-\delta_j}{\ga_j N} + \frac{i}{-\ga_j}, \qquad j = 1,\dots,r, \ i = 1,\dots,-\ga_j, \\
\be_{j,i} &= \frac{-\delta_j}{\ga_j N} + \frac{i}{\ga_j}, \qquad j = r+1,\dots,r+s, \ i = 1,\dots,\ga_j.
\end{align*}

We compute:
\begin{align*}
    \sum_{j,i} \al_{j,i} &= \sum_{j=1}^r \sum_{i=1}^{-\ga_j} \left( \frac{-\delta_j}{\ga_j N} + \frac{i}{-\ga_j} \right) = \sum_{j=1}^r \frac{\delta_j}{N} + \sum_{j=1}^r \frac{1}{-\ga_j} \cdot \frac{(-\ga_j)(-\ga_j + 1)}{2} \\
    &= \sum_{j=1}^r \frac{\delta_j}{N} + \sum_{j=1}^r \frac{-\ga_j + 1}{2}, \text{ and}\\
    \sum_{j,i} \be_{j,i} &= \sum_{j=r+1}^{r+s} \sum_{i=1}^{\ga_j} \left( \frac{-\delta_j}{\ga_j N} + \frac{i}{\ga_j} \right) = \sum_{j=r+1}^{r+s} \frac{-\delta_j}{N} + \sum_{j=r+1}^{r+s} \frac{1}{\ga_j} \cdot \frac{\ga_j(\ga_j + 1)}{2} \\
    &= \sum_{j=r+1}^{r+s} \frac{-\delta_j}{N} + \sum_{j=r+1}^{r+s} \frac{\ga_j + 1}{2}.
\end{align*}
Now, since $\sum_j \ga_j = 0$ and $\sum_j \delta_j \equiv 0 \pmod{N}$, we have
\[
\sum_{j,i} (\al_{j,i} - \be_{j,i}) = \frac{1}{N} \sum_{j=1}^{r+s} \delta_j - \frac{1}{2} \sum_j \ga_j + \frac{r - s}{2} \equiv \frac{r - s}{2} \pmod{\Z}.
\]

To prove the converse, first note that the gamma triple $((-2,1,1),(0,0,N),2N)$ represents the empty parameters $(\emptyset\, ;\emptyset)$, for any positive integer $N$.
Now, suppose that $(\al \,; \be) \in \left(\Q \cap (0,1]\right)^{n} \times \left(\Q \cap (0,1]\right)^{n}$ are hypergeometric parameters with $\sum_i (\al_i - \be_i) \in \frac{1}{2} \Z$. 
If $n=0$, the triple $((1,-1),(0,0),1)$ represents $(\al \, ; \be)$, so assume $n\ge1$. Let $N$ be the least common multiple of the denominators of the $\al_i$'s and $\be_i$'s. 
Let
\begin{align*}
    \ga &= (-1, \dots, -1; 1, \dots, 1) \in \Z^{n}\times \Z^n, \\
    \delta&= (2N\al_1, \dots, 2N\al_n, -2N\be_1, \dots, -2N\be_n).
\end{align*}
Then, the gamma triple $(\ga, \delta, 2N)$ represents the parameters $(\al \, ; \be)$ and satisfies that the integer $\sum_j \delta_j$ is a multiple of $N$. If $\sum_j \delta_j$ is an even multiple of $N$, we are done. Otherwise, we consider the triple 
$(\tilde{\ga},\tilde{\delta},2N)= ((\ga;-2,1,1),(\delta;0,0,N),2N)$, which represents $(\al \, ; \be)$ and satisfies that $\sum_{j} \tilde{\delta}_j $ is an even multiple of $N$.
\end{proof}

\subsection{Finite hypergeometric sums from gamma triples}
In what follows, we would like to define the finite hypergeometric sum associated to a gamma triple $(\ga,\delta,N)$. The definition extends the definition of the finite hypergeometric sum associated to $(\al \,; \be)$ as we shall see below. First, we need some notation.

\begin{defi}\label{defi:s-delta}
    Let $(\ga,\delta, N)$ be a gamma triple, and let $q$ be a prime power. Let $D_{\ga,\delta, N}(T)$  be the greatest common divisor of the numerator and the denominator of the right-hand side of Equation \eqref{eq:gamma-delta}. 
    For any integer $m$, define $s_{(\ga,\delta,N)}(m)$ to be the multiplicity of  $e^{2 \pi i m/ q^\times}$ as a root of $D_{\ga,\delta,N}(T)$. We will abuse notation and write $s_{\delta}(m)$ instead of $s_{(\ga,\delta,N)}(m)$.
\end{defi}

\begin{defi}\label{defi:finite-hg-function-gamma}
Let $(\ga,\delta, N)$ be a gamma triple. 
Let $q$ be a prime power that is coprime to $\ga_1 \cdots \ga_{d+2}$ and such that $N$ divides $q^\times$. 
For $t \in \F_q^\times$, the \emph{finite hypergeometric sum} associated to  $(\ga,\delta,N)$ is
\begin{equation}\label{eq:defi-hgm-function-from-gamma-triple}
F_q\left(\gamma, \delta,N \mid t\right) := \frac{1}{1-q}\sum_{m=0}^{q^\times-1}\frac{g\left(-\gamma m +\frac{\delta}{N} q^\times\right)} {g\left(\frac{\delta}{N} q^\times\right)}  q^{s_{\delta}(-m)-s_{\delta}(0)} \chi^m(\gamma^{\gamma}t).
\end{equation}
More generally, this definition makes sense if $\delta_{j}q^\times/N $ is an integer for all $j$.
\end{defi}
As for Definition~\ref{defi:finite-hg-function}, $F_q(\ga,\delta,N\mid t)$ does not depend on $\psi$. We have the following result.
\begin{theo}\label{theo:equality-of-finite-hg-sums}
Let $(\ga,\delta,N)$ be a gamma triple, and let $(\al \,;\be) \in {\Q^n\times \Q^n}$ be the hypergeometric parameters associated to it.
Then, the following equality holds:
\begin{equation}\label{eq:equality-of-different-finit-hgs}
F_q\left(\gamma, \delta,N \mid t\right) = {F_q(\al \, ; \be \mid t)},
\end{equation}
whenever both sides are defined (Definitions \ref{defi:finite-hg-function}, \ref{defi:finite-hg-function-gamma}).
\end{theo}
\begin{proof}
Let $q$ be a prime power such that both sides of \eqref{eq:equality-of-different-finit-hgs} are defined, that is, $q$ is coprime to $\ga_1\cdots\ga_{d+2}$, $N$ divides $q^\times$, and $q^\times\al,q^\times\be\in\Z^n$. 

First, we make some reductions. Note that $F_q(\ga,\delta,N\mid t)$ and $(\al\,;\be)$ depend on $\delta$ and $N$ only through $\delta/N$ modulo $\Z$. Thus, replacing $\delta$ by $\frac{q^\times}{N}\delta$ and $N$ by $q^\times$, we may and will assume that $N=q^\times$ and write $\zeta=\zeta_{q^\times}$. Moreover, formula \eqref{eq:defi-hgm-function-from-gamma-triple} does not use the condition $\gcd(\ga_1,\dots,\ga_{d+2})=1$, and neither do the operations below. Thus, we drop this condition. We also allow the empty triple.

Next, we describe two operations on $(\ga,\delta,q^\times)$ which change neither $(\al\,;\be)$ nor the coefficients of $\chi^m(t)$ in \eqref{eq:defi-hgm-function-from-gamma-triple}. In both cases, we denote the new triple by $(\ga',\delta',q^\times)$, and we denote by $C_m$ and $C'_m$ the coefficients of $\chi^m(t)$ in \eqref{eq:defi-hgm-function-from-gamma-triple} for $(\ga,\delta,q^\times)$ and $(\ga',\delta',q^\times)$, respectively.
\begin{enumerate}
    \item[(i)] Suppose that $h=\gcd(\ga_j,\delta_j,q^\times)>1$ for some $j$. We replace the entry $(\ga_j,\delta_j)$ by the $h$ entries $(\ga_j/h,(\delta_j+iq^\times)/h)$, $0\le i<h$.

    To show that $(\al\,;\be)$ does not change, note that the numbers $\zeta^{\pm(\delta_j+iq^\times)/h}$, $0\le i<h$, are the $h$ distinct $h$-th roots of $\zeta^{\pm\delta_j}$. Thus,
    \[
    T^{|\ga_j|}-\zeta^{\pm\delta_j}=\prod_{i=0}^{h-1}\big(T^{|\ga_j|/h}-\zeta^{\pm(\delta_j+iq^\times)/h}\big),\qquad \pm=-\operatorname{sgn}(\ga_j).
    \]
    Hence, the numerator and the denominator of the right-hand side of \eqref{eq:gamma-delta} do not change. In particular, $(\al\,;\be)$ and $s_\delta$ do not change.

    To show that $C_m=C'_m$, note that by Theorem~\ref{Hasse-Davenport} (with $N=h$, which divides $q^\times$), we have
    \[
    \frac{g(-\ga_jm+\delta_j)}{g(\delta_j)}=\chi(h)^{-\ga_jm}\prod_{i=0}^{h-1}\frac{g\big(-\frac{\ga_j}{h}m+\frac{\delta_j+iq^\times}{h}\big)}{g\big(\frac{\delta_j+iq^\times}{h}\big)}.
    \]
    Moreover, the factor $\chi(\ga_j)^{\ga_jm}$ of $\chi^m(\ga^\ga)$ is replaced by $\prod_{i=0}^{h-1}\chi(\ga_j/h)^{\ga_jm/h}$. The remaining factors of $C_m$ and $C'_m$ agree. Thus,
    \[
    \frac{C_m}{C_m'} = \frac{g(-\ga_jm+\delta_j)\, \chi(\ga_j)^{\ga_jm}}{g(\delta_j)}\prod_{i=0}^{h-1} \frac{g\big(\frac{\delta_j+iq^\times}{h}\big)}{g\big(-\frac{\ga_j}{h}m+\frac{\delta_j+iq^\times}{h}\big)\, \chi(\ga_j/h)^{\ga_jm/h}}= \chi(h)^{-\ga_jm}\,\chi(h)^{\ga_jm}= 1.
    \]
    \item[(ii)] Suppose that there are entries $(-a,e)$ and $(a,e')$ with $a>0$ and $e'\equiv-e\pmod{q^\times}$. We remove both entries.

    To show that $(\al\,;\be)$ does not change, note that the entry $(-a,e)$ contributes the factor $T^a-\zeta^e$ to the numerator of the right-hand side of \eqref{eq:gamma-delta}. Moreover, the entry $(a,e')$ contributes the factor $T^a-\zeta^{-e'}=T^a-\zeta^e$ to the denominator. Thus, the right-hand side of \eqref{eq:gamma-delta} does not change, and neither does $(\al\,;\be)$.

    To show that $C_m=C'_m$, let $r(x)\in\{0,1\}$ be the multiplicity of $\zeta^x$ as a root of $T^a-\zeta^e$. Then, $r(x)=1$ if and only if $ax\equiv e\pmod{q^\times}$. The Gauss sums coming from the other entries are the same in $C_m$ and $C'_m$. Moreover, $\ga^\ga=(-a)^{-a}a^{a}\,\ga'^{\ga'}=(-1)^{a}\,\ga'^{\ga'}$. Thus,
    \[
    \frac{C_m}{C_m'} = \frac{g(am+e)\,g(-am+e')}{g(e)\,g(e')}\, \chi(-1)^{am}\, \frac{q^{s_\delta(-m)-s_\delta(0)}}{q^{s_{\delta'}(-m)-s_{\delta'}(0)}}.
    \]
    We compute the first and the last factor. Since $e'\equiv -e$, we have $g(-am+e')=g(-am-e)$ and $g(e')=g(-e)$. Moreover, $am+e\equiv0$ if and only if $r(-m)=1$, and $e\equiv0$ if and only if $r(0)=1$. Thus, Lemma~\ref{lem:gauss-sums-properties} gives
    \[
    \frac{g(am+e)\,g(-am+e')}{g(e)\,g(e')}=\frac{g(am+e)\,g(-am-e)}{g(e)\,g(-e)}=\chi(-1)^{am}\,q^{r(0)-r(-m)}.
    \]
    Next, the numerator and the denominator of the right-hand side of \eqref{eq:gamma-delta} for $(\ga,\delta,q^\times)$ are obtained from those for $(\ga',\delta',q^\times)$ by multiplying by $T^a-\zeta^e$. Thus, $D_{\ga,\delta,q^\times}=(T^a-\zeta^e)\,D_{\ga',\delta',q^\times}$, so $s_\delta=s_{\delta'}+r$. Hence,
    \[
    \frac{q^{s_\delta(-m)-s_\delta(0)}}{q^{s_{\delta'}(-m)-s_{\delta'}(0)}}=q^{r(-m)-r(0)}.
    \]
    Thus, since $\chi(-1)=\pm1$,
    \[
    \frac{C_m}{C_m'} = \chi(-1)^{am}\,q^{r(0)-r(-m)} \cdot \chi(-1)^{am}\, q^{r(-m)-r(0)} =1.
    \]
\end{enumerate}

We now prove the theorem. First, we apply (i) until no entry can be split. This terminates, since each step replaces an entry by entries with smaller $|\ga_j|$. Then, $\gcd(\ga_j,\delta_j,q^\times)=1$ for all $j$. Thus, every factor on the right-hand side of \eqref{eq:gamma-delta} is of the form $T^a-\zeta^e$ with $\gcd(a,e,q^\times)=1$. Since $q^\times\al,q^\times\be\in\Z^n$, the same holds for the factors $T-\zeta^{q^\times\al_i}$ and $T-\zeta^{q^\times\be_i}$ on the left-hand side.

Let $S_{\mathrm{num}}$ and $S_{\mathrm{den}}$ be the multisets containing one pair $(a,e)$ for each factor $T^a-\zeta^e$ of the numerator, respectively the denominator, of the right-hand side of \eqref{eq:gamma-delta}. Let $S_\al$ and $S_\be$ be the multisets of the pairs $(1,q^\times\al_i)$, respectively $(1,q^\times\be_i)$, for $1\le i\le n$. 
For multisets $X$ and $Y$, we write $X+Y$ for their union, in which multiplicities are added, and $X\subseteq Y$ if every element occurs in $X$ at most as often as in $Y$. 
Cross-multiplying \eqref{eq:gamma-delta}, Lemma~\ref{lem:generalizing-cyclotomics}(b) (with $n=q^\times$) gives $S_{\mathrm{num}}+S_\be=S_{\mathrm{den}}+S_\al$. Since $\al_i-\be_j$ is not an integer for any $i$ and $j$, the multisets $S_\al$ and $S_\be$ are disjoint. Thus, $S_\al\subseteq S_{\mathrm{num}}$. Writing $S_{\mathrm{num}}=S_\al+S$, we get $S_{\mathrm{den}}=S_\be+S$.

Let $(a,e)\in S$. Since $(a,e)\in S_{\mathrm{num}}$ and $(a,e)\in S_{\mathrm{den}}$, there are entries $(-a,e_1)$ and $(a,e_2)$ with $e_1\equiv e$ and $e_2\equiv -e\pmod{q^\times}$. Thus, (ii) applies to these two entries. Applying (ii) once for each element of $S$, counted with multiplicity, and changing $\delta$ modulo $q^\times$ if necessary, we are left with the triple
\[
\big((-1,\dots,-1,1,\dots,1),\,(q^\times\al_1,\dots,q^\times\al_n,-q^\times\be_1,\dots,-q^\times\be_n),\,q^\times\big)
\]
of Example~\ref{example-1}, with $N$ replaced by $q^\times$ (the empty triple if $n=0$). For this triple, $D_{\ga,\delta,q^\times}=1$, since $S_\al$ and $S_\be$ are disjoint. Thus, $s_\delta=0$. Moreover, $\ga^\ga=(-1)^{-n}=(-1)^n$. Hence, the right-hand side of \eqref{eq:defi-hgm-function-from-gamma-triple} becomes
\[
\frac{1}{1-q}\sum_{m=0}^{q^\times-1}\prod_{i=1}^n\frac{g(m+\al_iq^\times)\,g(-m-\be_iq^\times)}{g(\al_iq^\times)\,g(-\be_iq^\times)}\,\chi\big((-1)^nt\big)^m,
\]
which is exactly the formula in Definition~\ref{defi:finite-hg-function}. This proves the theorem.
\end{proof}
Note that if $\delta$ is the zero vector, then $ (\al \, ;  \beta) $ is defined over $\Q$, and this result recovers \cite[Theorem 1.3]{BCM}.
Indeed, the denominator $g\left(\frac{\delta}{N} q^\times\right)$ is equal to $(-1)^{d+2}$, and $s_\delta(m) = s_\delta(-m)$, since the parameters $(\al \, ; \be)$ are defined over $\Q$.
\begin{rem}
Note that if $n=0$ in Theorem \ref{theo:equality-of-finite-hg-sums}, then
\[
F_q(\ga,\delta, N \mid t) = F_q (\emptyset\, ; \emptyset \mid t) = 
\begin{cases}
-1 &  \text{ if } t=1 , \\
0 & \text{ otherwise,}
\end{cases}
\]
whenever the left-hand side is defined. 

\end{rem}
\begin{prop}\label{prop:uniqueness-extensions}
    Suppose that $(\al \,;\be)$ are hypergeometric parameters associated to two gamma triples $(\ga,\delta,N)$ and $(\ga',\delta',N').$ Let $M$ be the least common multiple of $N$ and $N'$. 
    Then,
    \begin{equation}\label{eq:cor-well-defined-extensions}
      F_q(\ga,\delta,N \mid t) = F_q(\ga',\delta',N' \mid t)  
    \end{equation}
    for all prime powers $q$ such that $q^\times$ is divisible by $M$ and $q$ is coprime to the $\ga_j$'s and $\ga_j'$'s.
\end{prop}
\begin{proof}
    Write $\sum_{m=0}^{q^\times-1} A_m \chi^{m}(t)$ for the left-hand side of Equation \eqref{eq:cor-well-defined-extensions} and $\sum_{m=0}^{q^\times-1} A_m' \chi^{m}(t)$ for {its} right-hand side. 
    If $h=\gcd(N,N'), $ then $M= NN'/h.$
    Consider the gamma triple
    \[
    \left((\ga,-\ga'),\left( \frac{N'}{h}\delta, -\frac{N}{h}\delta'\right), \frac{NN'}{h} \right).
    \]
    By definition, after inverting the Gauss sums coming from $-\ga'$ with Lemma~\ref{lem:gauss-sums-properties}, we have
    \begin{equation}
    F_q\left((\ga,-\ga'),\left( \frac{N'}{h}\delta, -\frac{N}{h}\delta' \right), \frac{NN'}{h}  \,  \bigg | \, t\right) = \frac{1}{1-q}\sum_{m=0}^{q^\times-1} \frac{A_m}{A_m'} \chi^{m}(t). \label{eq:cor-extensions-unique-1}
    \end{equation}
    On the other hand, by hypothesis, the hypergeometric parameters associated to 
    \[
    \left((\ga,-\ga'),\left( \frac{N'}{h}\delta, -\frac{N}{h}\delta' \right), \frac{NN'}{h} \right)
    \]
    are $(\emptyset \,;\emptyset)$, since $\left(\ga, \frac{N'}{h} \delta, \frac{NN'}{h}\right)$ represents $(\al \,;\be)$ while  $\left(-\ga', -\frac{N}{h} \delta', \frac{NN'}{h}\right)$ represents $(\be\, ;\al)$. 
    Thus, by Theorem \ref{theo:equality-of-finite-hg-sums}, we have 
    \begin{equation}
    F_q\left((\ga,-\ga'),\left( \frac{N'}{h}\delta, -\frac{N}{h}\delta' \right), \frac{NN'}{h}  \,  \bigg | \, t\right) = \frac{1}{1-q}\sum_{m=0}^{q^\times-1} 1\cdot  \chi^{m}(t). \label{eq:cor-extensions-unique-2}
    \end{equation}
    Equations \eqref{eq:cor-extensions-unique-1} and \eqref{eq:cor-extensions-unique-2} give two Fourier series (see Lemma \ref{lem:fourier-series} below) for the same function. By uniqueness of Fourier coefficients, we have $\frac{A_m}{A_m'}=1$ for all $m$. In other words, $A_m =A_m'$ for all $m$, which proves the proposition.
\end{proof}

\begin{defi}
Let $q$ be a prime power. We redefine the finite hypergeometric sum $F_q(\al \, ; \be \mid t)$ {to be} the finite hypergeometric sum $F_q(\ga,\delta,N \mid t)$ for any gamma triple $(\ga,\delta,N)$ which represents $(\al \, ; \be)$ and for which $F_q(\ga,\delta,N \mid t)$ is defined at $q$. 
\end{defi}
This definition makes sense since there is at least one gamma triple which represents $(\al \,; \be)$ (Example \ref{example-1}), and for any two gamma triples which represent $(\al\, ;\be)$ the associated finite hypergeometric sums coincide whenever they are defined, by Proposition \ref{prop:uniqueness-extensions}. 

This definition provides an extension of Definition \ref{defi:finite-hg-function} since for the prime powers $q$ for which $F_q(\al \,; \be \mid t)$ is already defined, the two definitions coincide by Theorem \ref{theo:equality-of-finite-hg-sums} and Example \ref{example-1}.

Let $N$ be the smallest positive integer such that given hypergeometric parameters $(\al \,; \be)$ can be defined over $\Q(\zeta_N)$. By Proposition \ref{prop:extension-field-field-of-definition}, there is a gamma triple $(\ga,\delta,N)$ which represents $(\al \, ; \be)$. 
Moreover, by condition \eqref{condition-3} of Lemma~\ref{lem:generalizing-cyclotomics}(a), this gamma triple can be chosen such that every prime dividing some $\ga_j$ divides the denominator of some $\al_i$ or $\be_i$.
The associated finite hypergeometric sum $F_q(\ga,\delta,N \mid t)$ is defined for all prime powers $q$ such that $q^\times$ is divisible by $N$, minus, possibly, the powers of the finitely many primes which divide the denominators of the $\al_i$'s and $\beta_i$'s.

We summarise the results of this subsection in the following theorem:

\begin{theo}\label{theo:extension-of-finite-hg-function}
Let $(\al \,; \be)$ be hypergeometric parameters, and let $N$ be the smallest positive integer such that they can be defined over $\Q(\zeta_N)$.
The finite hypergeometric sum $F_q(\al \,; \be \mid t)$ is defined for all prime powers $q$ such that $q^\times$ is divisible by $N$, minus, possibly, the powers of the finitely many primes which divide the denominators of the $\al_i$'s and $\beta_i$'s.
\end{theo}

\begin{exa}\label{example-extension-of-primes}
The finite hypergeometric sum $F_q((1/3,2/3) \, ; (1,1) \mid t)$ is defined for all prime powers $q$ which are coprime to $3$, since $((-3,1,1,1),(0,0,0,0),1)$ represents $((1/3,2/3)\, ; (1,1))$. This sum was previously defined (Definition \ref{defi:finite-hg-function}) only when $q\equiv 1 \pmod{3}$. 
\end{exa}

\begin{rem}
The flexibility in choosing different gamma triples representing some hypergeometric parameters $(\al \,; \be)$ allows for the realisation of the finite hypergeometric sum $F_q(\al \,; \be \mid t)$ in different families of varieties (see Section \ref{section:reverse-engineering}).
\end{rem}

\section{Hypersurfaces of algebraic tori}\label{section:source-varieties}
Let $k$ be an arbitrary field.
In this section, we consider affine hypersurfaces of the $d$-dimensional algebraic torus over $k$ defined by the vanishing of Laurent polynomials that are linear combinations of exactly $d+2$ monomials. Throughout, $Z$ is the hypersurface of $\G_m^d = \Spec(k[x_1^{\pm},\dots, x_d^{\pm}])$ defined by the vanishing of the Laurent polynomial
\begin{equation}\label{eq:main-laurent-polynomial}
f(x) = \sum_{j=1}^{d+2} u_j \prod_{i=1}^d x_i^{m_{ij}} ,
\end{equation}
where $u_1,\dots,u_{d+2} \in k^\times$. 
Assume that the exponent vectors $m_{*j}=(m_{1j},\dots,m_{dj}) \in \Z^{d}$ for $j=1, \dots, d+2$ are distinct and do not lie in any affine hyperplane of $\R^{d}$. We denote by $\Delta$ the Newton polytope of $f$ in $\R^{d}$. It is $d$-dimensional by the condition on the vectors {$m_{*1}, \dots, m_{*,d+2}$.}

Consider the matrix
\begin{equation}\label{eq:matrix-M}
M=\begin{pmatrix}
1 & 1 &\cdots &1 \\
m_{11} & m_{12} & \cdots & m_{1,d+2} \\
m_{21} & m_{22} & \cdots & m_{2,d+2} \\
\vdots & \vdots & \ddots & \vdots \\
m_{d1} & m_{d2} & \cdots & m_{d,d+2}
\end{pmatrix}. 
\end{equation}
We denote its first row by $\mathbf{1}$, and its $(i+1)$-th row by $m_i$. The columns of $M$ are $d+2$ vectors in $\R^{d+1}$; thus they are linearly dependent. Moreover, since the exponent vectors {$m_{*1}, \dots, m_{*,d+2}$} do not lie in any affine hyperplane in $\R^{d}$, the rank of $M$ is $d+1$. It follows that the lattice of linear relations between the columns of $M$ is of rank $1$. Thus, there is a vector  $ \gamma =(\gamma_1, \dots, \gamma_{d+2})\in \mathbb{Z}^{d+2}$, unique up to sign, such that  
\begin{align*}
    \sum_{j=1}^{d+2} \gamma_j &= 0, \qquad \gcd(\gamma_1, \dots, \gamma_{d+2}) = 1, \text{ and }\\
    \sum_{j=1}^{d+2} m_{ij} \gamma_j &= 0 \text{, }\qquad  i = 1, \dots, d. 
\end{align*}
We will disregard the sign ambiguity in the choice of $\ga$, as it does not affect the results under consideration.

\begin{defi}\label{def:gale-dual}
The vector $\gamma \in \Z^{d+2}$ is called the \emph{Gale dual} {of $f$, $Z$}, or  $\Delta$. 
\end{defi}

Let $K$ be the kernel of the map $\mathbb{Z}^{d+2} \to \mathbb{Z}$ given by $x \mapsto \sum_{j=1}^{d+2} \gamma_j x_j$. Since $\ga$ is non-zero, $K$ is of rank $d+1$. Let $K'$ be the lattice spanned by the row vectors of the matrix $M$ in $\mathbb{Z}^{d+2}$. Since $M$ is of rank $d+1$, the lattice $K'$ is a finite index sublattice of $K$.

\begin{defi}
We call the index $[K\!:\!K']$ the \emph{covering degree} of $Z$, and denote it by $\deg(Z)$.
If $\deg(Z)=1$, equivalently if $K'=K$, we say that $f$, {$Z$, and $\Delta$ are \emph{primitive}. }
\end{defi}

\begin{rem}
  Gale duality is an important concept in the combinatorics of polytopes. It is defined for general \emph{vector configurations} (see \cite[\S 4.1]{triangulations-book}).  The cases we are interested in in this paper are those whose Gale dual is one-dimensional. 
\end{rem}

\begin{lem}\label{lem:projective-to-toric}
   If $Z$ is primitive, then it is isomorphic to the subvariety of $\Ps^{d+1}$ given by 
   \begin{align*}
    w_1+\dots +w_{d+2} &=0, \quad     w_1\cdots w_{d+2} \neq 0,\\
    \prod_{j=1}^{d+2} w_j^{\gamma_j} &=t,
    \end{align*}
    where  $w_1,\dots, w_{d+2}$ are homogeneous coordinates on $\Ps^{d+1}$ and $t = \prod_{j=1}^{d+2} u_j^{\gamma_j}$. Furthermore, such an isomorphism is given by 
    \[
    w_j = u_j \prod_{i=1}^{d} x_i^{m_{ij}},  \quad j=1,\dots, {d+2} .
    \]
\end{lem}
\begin{proof}
A proof can be found in \cite[\S 3]{hypergeometric-survey} when $\ga_j \neq 0$ for all $j$. However, the same argument works if $\ga_j =0$ for some $j$.
\end{proof}

\begin{rem}\label{rem:primitive-dehomogenization}
The proof of Lemma \ref{lem:projective-to-toric} implies that, if $Z$ is primitive, then the variety in $\Ps^{d+1}$ defined by $w_1\cdots w_{d+2} \neq 0$ 
and $\prod_{j=1}^{d+2} w_j^{\gamma_j} =t$ is isomorphic to $\G_m^{d} = \Spec k[x_1^{\pm},\dots,x_d^{\pm}]$ via the coordinate map
\[
w_j = u_j \prod_{i=1}^{d} x_i^{m_{ij}}, \quad j=1,\dots, d+2.
\]
\end{rem}

If $Z$ is not primitive, then it is a covering of a primitive hypersurface:
\begin{prop}\label{prop:cover}
Let $Z^{\mathrm{prim}}$ be any primitive hypersurface with Gale dual $\ga$ and the same coefficients $u_1, \dots, u_{d+2}$. Then, there is a covering map 
\[
Z \to Z^{\mathrm{prim}}.
\]
Furthermore, the degree of the covering map is the covering degree $\deg(Z)$.
This map is \'etale if the degree of the map is invertible in $k.$ 
\end{prop}
\begin{proof}
Fix an isomorphism of the kernel $K$ with $\mathbb{Z}^{d+1}$ by choosing an integral basis $f_0,f_1,\dots,f_{d}$ of $K$ such that $f_0 = \mathbf{1}$.
Let $Z^{\mathrm{prim}}$ be the hypersurface in $\G_m^{d}$ defined by
\begin{equation}\label{eq:primitive-coefficients}
\sum_{j=1}^{d+2} u_j \prod_{i=1}^d {x_i}^{f_{ij}}=0.   
\end{equation}
By definition, $Z^{\mathrm{prim}}$ is primitive with Gale dual $\gamma$. We note that any primitive hypersurface arises by choosing such a basis.  
Write $m_i = \sum_{k=1}^d c_{ik} f_k +r_i f_0$ for $i=1,\dots,d$, and consider {the} coefficient matrix
\[
C=\begin{pmatrix}
1 & 0 & 0 & \cdots & 0 \\
r_1&c_{11} & c_{12} & \cdots & c_{1d} \\
r_2 & c_{21} & c_{22} & \cdots & c_{2d} \\
\vdots & \vdots & \vdots & \ddots & \vdots \\
r_d & c_{d1} & c_{d2} & \cdots & c_{dd}
\end{pmatrix}.
\]
Let $\mathcal{H}$ be the Hermite normal form of $C$. Then, $\mathcal{H} = A C$ for some $A \in \operatorname{GL}(d+1,\mathbb{Z})$. Furthermore,
\[
A=\begin{pmatrix}
1 & 0 & 0 & \cdots & 0 \\
s_1& A_{11} & A_{12} & \cdots & A_{1d} \\
s_2 & A_{21} & A_{22} & \cdots & A_{2d} \\
\vdots & \vdots & \vdots & \ddots & \vdots \\
s_d & A_{d1} & A_{d2} & \cdots & A_{dd}
\end{pmatrix}\text{, and }
\mathcal{H}=\begin{pmatrix}
1 & 0 & 0 & \cdots & 0 \\
0& \mathcal{H}_{11} & \mathcal{H}_{12} & \cdots & \mathcal{H}_{1d} \\
0 & 0 & \mathcal{H}_{22} & \cdots & \mathcal{H}_{2d} \\
\vdots & \vdots & \vdots & \ddots & \vdots \\
0 & 0 & 0 & \cdots & \mathcal{H}_{dd}
\end{pmatrix}.
\]
In particular, for all $l,k \in \{1,\dots,d\}$, we have $\sum_{i=1}^d A_{ki}c_{il} = \mathcal{H}_{kl}$. 
Note that $\det(\mathcal{H})$ is equal to the covering degree $\deg(Z)$.
We make the change of variables
\[
x_i= \prod_{k=1}^d y_k^{A_{ki}}, \quad i=1, \dots, d.
\]
\begingroup
\allowdisplaybreaks
Then,
\begin{align*}
\prod_{i=1}^d x_i^{m_{ij}} &= \prod_{i=1}^d \prod_{k=1}^d y_k^{A_{ki}m_{ij}} =\prod_{k=1}^d y_k^{\sum_{i=1}^d A_{ki} m_{ij}}\\
&=\prod_{k=1}^d y_k^{\sum_{i=1}^d \left(\sum_{l=1}^d A_{ki}c_{il} f_{lj} +A_{ki}r_i \right) }   &(m_{ij} = \sum_{l=1}^d c_{il} f_{lj} +r_i )\\
&=  \prod_{k=1}^d y_k^{\sum_{l=1}^d\sum_{i=1}^dA_{ki} c_{il} f_{lj} } \cdot\prod_{k=1}^d y_k^{\sum_{i=1}^d r_i A_{ki}} \\
&= \prod_{k=1}^d y_k^{\sum_{l=1}^d \mathcal{H}_{kl} f_{lj}}\cdot\prod_{k=1}^d y_k^{\sum_{i=1}^d r_i A_{ki}} .
\end{align*}
\endgroup
We set $U = \prod_{k=1}^d y_k^{\sum_{i=1}^d r_i A_{ki}}$, and note that it is independent of $j$. Thus,
\[
U^{-1}f(x) = \sum_{j=1}^{d+2} u_j \prod_{i=1}^d  y_i ^{\sum_{l=1}^d \mathcal{H}_{il}f_{lj}} \eqqcolon F(y).
\]
Now, note that the hypersurface defined by $F=0 $ in $\G_m^{d} \coloneqq \Spec k[y_1^\pm,\dots, y_d^\pm]$ is isomorphic to the variety in $\G_m^d \times \G_m^d \coloneqq \Spec k[y_1^\pm,\dots, y_d^\pm] \times \Spec k[z_1^\pm,\dots, z_d^\pm]$ defined by the vanishing of the equations
\begin{align}
    \sum_{j=1}^{d+2} u_j \prod_{i=1}^d z_i^{f_{ij}}&=0 \label{eq:primitive},\\
     z_k -\prod_{i=1}^k y_i^{\mathcal{H}_{ik}}  &= 0,\qquad    k=1,\dots,d.\label{eq:primitive-other}
\end{align}
Equation \eqref{eq:primitive} defines a hypersurface in $\Spec k[z_1^\pm,\dots, z_d^\pm]$ that is isomorphic to $Z^{\mathrm{prim}}$. Furthermore, Equations \eqref{eq:primitive-other} define a covering of $\Spec k[z_1^\pm,\dots, z_d^\pm]$ since the matrix $\mathcal{H}$ has non-zero determinant. Furthermore, it is an \'etale covering if and only if $\det(\mathcal{H})=[K:K']$ is invertible in $k$.
\end{proof}

\begin{corol}\label{cor:cover1}
There exist $d$ vectors $\rho_1,\dots,\rho_d \in \mathbb{Z}^{d+2}$ such that for $k=1,\dots,d$,
\[
z_k = \prod_{j=1}^{d+2} \left( \prod_{i=1}^d z_i^{f_{ij}} \right) ^{\rho_{kj}}, \text{ and }  \quad \sum_{j=1}^{d+2} \rho_{kj}=0 .\]
Furthermore, if we define
\[
\sigma_k \coloneqq \prod_{j=1}^{d+2} u_j^{\rho_{kj}}, \qquad k=1,\dots, d,
\] then the hypersurface $Z$ is isomorphic to the variety given by
\begin{align*}
    w_1+\dots +w_{d+2} &= 0, \quad w_1\cdots w_{d+2} \neq 0,\\
    \prod_{j=1}^{d+2} w_j^{\gamma_j} &= t, \quad  t=\prod_{j=1}^{d+2}u_j^{\gamma_j},\\
    \prod_{j=1}^{d+2} w_j^{\rho_{kj}} &= \sigma_k \prod_{i=1}^d y_i^{\mathcal{N}_{ki}},    \qquad k=1,\dots,d 
\end{align*}
in $\mathbb{P}^{d+1} \times \G_m^{d}$, for some matrix $\mathcal{N} \in M_{d\times d}(\Z)$ with $\det(\mathcal{N})=\deg(Z)$. 
Here, $w_1,\dots,w_{d+2}$ are projective coordinates on $\mathbb{P}^{d+1}$ and $y_1,\dots ,y_{d}$ are coordinates on $\G_m^{d}$.
\end{corol}
\begin{proof}
The first statement follows since the rows of the matrix $M$ above (with $m_{ij}=f_{ij}$) span $K$. Thus, the columns of $M$ span $\Z^{d+1}$, so we can choose $\rho_k$ with $\sum_{j=1}^{d+2} f_{ij}\rho_{kj}=\delta_{ik}$ for $i=0,\dots,d$.
For the second statement, let $\mathcal{N} \coloneqq \mathcal{H}^{T}$, where $\mathcal{H}$ is as in the proof of Proposition \ref{prop:cover}. Then, $\det(\mathcal{N})=\deg(Z)$, and Equations \eqref{eq:primitive-other} read $z_k=\prod_{i=1}^d y_i^{\mathcal{N}_{ki}}$. The statement now follows from Equations \eqref{eq:primitive}, \eqref{eq:primitive-other} and Lemma \ref{lem:projective-to-toric}.
\end{proof}
For a Laurent polynomial $g(x)= \sum_{n \in \Z^{d}} c_{n} \prod_{i=1}^{d} x_i^{n_i}$ and a subset $S \subset \R^d$,  denote by  $g_{|S}$ the restriction of $g$ to $S$ defined as
\[
g_{|S}(x)= \sum_{n\in S\cap \Z^{d}} c_{n} \prod_{i=1}^{d} x_i^{n_i}.
\]
A slight modification of the proof of Corollary \ref{cor:cover1} gives:
\begin{corol}\label{cor:cover}
Let $F$ be a face of $\Delta$. Let $S= \left\{ j \in \{1,\dots, d+2\} \mid m_{*j} \notin F \right\}$; in particular, $ F= \operatorname{conv}( m_{*j} \mid j \notin S )$.  Then, the subvariety $V(f_{|F}) \subset \G_m^{d}$ is isomorphic to the variety in $\mathbb{P}^{d+1} \times \G_m^{d}$ defined by
\begin{align*}
   \sum_{j\notin S} w_j&= 0, \quad  w_1\cdots w_{d+2} \neq 0,\\
    \prod_{j=1}^{d+2} w_j^{\gamma_j} &=t, \quad  t=\prod_{j=1}^{d+2}u_j^{\gamma_j},\\
    \prod_{j=1}^{d+2} w_j^{\rho_{kj}} &=  \sigma_k \prod_{i=1}^d z_i^{\mathcal{N}_{ki}}  \qquad k=1,\dots,d.
\end{align*}
\end{corol}
\subsection{Main count}
In what follows, we assume that $k= \F_q$ for some prime power $q.$ 
We want to count the number of $\F_q$-points of the hypersurfaces defined by the various restrictions of the Laurent polynomial $f$ to the faces of the Newton polytope $\Delta$ of $f$.  

Let $S$ be any subset of the set $\{1,\dots,d+2\}$, and let $W_{S}$ be the subvariety of $\mathbb{P}^{d+1} \times \G_m^{d}$ defined by $ w_1\cdots w_{d+2} \neq 0$ and the equations
\begin{align*}
    \sum_{j \notin S} w_{j} &= 0,\\
    w^{\gamma}\coloneqq\prod_{j=1}^{d+2} w_j^{\gamma_j} &=t, \\
    w^{\rho_k} z^{-\mathcal{N}_k} \coloneqq \prod_{j=1}^{d+2} w_j^{\rho_{kj}} \prod_{i=1}^d z_i^{-\mathcal{N}_{ki}} &= \sigma_k, \qquad k=1,\dots,d,
\end{align*}
where $t, \sigma_1,\dots,\sigma_d \in \F_q^\times$. 
The number of $\F_q$-points of $W_S$ is ultimately related to the set of solutions $\Lambda(q)$ of the equations 
\begin{align*}
    \lambda_1 \mathcal{N}_{11}+ \dots +\lambda_d \mathcal{N}_{d1}&=0,\\
    \vdots \\
    \lambda_1 \mathcal{N}_{1d}+ \dots +\lambda_d \mathcal{N}_{dd}&=0
\end{align*}
in $\left(\mathbb{Z}/q^\times\mathbb{Z}\right)^{d}$. 
For every solution $\lambda \in \Lambda(q)$, define 
\begin{align*}
    \delta(j,\lambda) &\coloneqq \sum_{i=1}^d \lambda_i \rho_{ij} \in \Z/q^\times\Z \quad \text{ for } j=1,\dots, d+2 ,\\
    \delta(\lambda) &\coloneqq (\delta(1,\lambda),\dots,\delta(d+2,\lambda)) \in \left(\Z/q^\times \Z\right)^{d+2}.
\end{align*}
For $x \in \Z/q^\times\Z$, define 
\[
\epsilon(x) \coloneqq \begin{cases}
     1 & \text{ if } x =0,\\ 
     0 & \text{ if } x \neq 0.
 \end{cases}
\]
We have the following result.
\begin{prop}\label{prop:main-count}
The number of $\F_q$-points of $W_S$ is given by
\begin{align*}
     \#W_{S} (\mathbb{F}_q) &= q^{-1}(q-1)^{d} + (-1)^{|S|}q^{-1}(q-1)^{|S|-1} \\
     &\qquad\cdot \sum_{\lambda \in \Lambda(q)}   \sum_{m=0}^{q^\times-1} \prod_{j\in S} \epsilon\left( -m \gamma_j -\delta(j,\lambda) \right) g\left(-m\gamma - \delta(\lambda)\right) \chi^{m}(t) \prod_{k=1}^d \chi^{\lambda_k}(\sigma_k).
\end{align*}
\end{prop}
We need the following well-known lemma \cite[Chapter VI, Section 1.2]{serre}.
\begin{lem}[Finite Fourier Transform]\label{lem:fourier-series}
Let $g: \left(\F_{q}^{\times}\right)^{d} \to \mathbb{C}$ be a complex valued function. Then, for all $x \in \left(\F_{q}^{\times}\right)^{d}$, we have
\[
g(x) = \sum_{\substack{k_i =0 \\ i=1,\dots,d}}^{q^\times -1} a_k \prod_{i=1}^d \chi^{k_i}(x_i),
\]
where 
\[
a_k = \frac{1}{(q-1)^{d}}\sum_{x \in (\F_q^{\times})^{d} } g(x) \prod_{i=1}^d\chi^{-k_i}(x_i).
\]
Furthermore, the \emph{Fourier coefficients} $a_k$ are unique. 
\end{lem}
\begin{proof}[Proof of Proposition \ref{prop:main-count}]
Let $V$ be the subvariety in $\G_m^{d+2} \times \G_m^d$ defined by the equations
\begin{align}
f_0(w,z)&\coloneqq w^{\gamma} =t\eqqcolon \sigma_0, \label{eq:proof-main-count-1}\\
f_k(w,z)&\coloneqq w^{\rho_k} z^{-\mathcal{N}_k}=  \sigma_k, \qquad k=1,\dots,d.\label{eq:proof-main-count-2}
\end{align}
Note that $V$ is a cone over the variety $V'$ defined by the same homogeneous (in $w_1, \dots, w_{d+2}$) equations in $\Ps^{d+1} \times \G_m^d$. By the choice of the $\rho_k$ in the proof of Corollary \ref{cor:cover1}, the vectors $\ga,\rho_1,\dots,\rho_d$ form a basis of $\{a \in \Z^{d+2} \mid \sum_{j} a_j = 0\}$, so we can choose $v_1,\dots,v_{d+2} \in \F_q^\times$ with $\prod_{j=1}^{d+2} v_j^{\ga_j}=t$ and $\prod_{j=1}^{d+2} v_j^{\rho_{kj}}=\sigma_k$ for $k=1,\dots,d$. Setting 
\[
w_j = v_j \prod_{i=1}^{d} x_i^{f_{ij}}, \qquad j=1,\dots,d+2,
\]
Equation \eqref{eq:proof-main-count-1} becomes a tautology, and Equations \eqref{eq:proof-main-count-2} become
\[
x_k z^{-\mathcal{N}_k}=  1, \qquad k=1,\dots,d.
\]
Thus, the coordinates $x_1,\dots,x_d$ are determined by $z_1,\dots,z_d$. Furthermore, the morphism defined by this substitution is an isomorphism by Remark \ref{rem:primitive-dehomogenization}. 
Thus, $V'$ is a torus of dimension $d$ and $V$ is a torus of dimension $d+1$.

We dehomogenise the equations of $W_S$ by setting $y_j = w_{j}/w_{d+2}$ for $j=1,\dots,d+2$. For $k=0,\dots,d$, we define $g_k$ by the equation
\[
g_k(y_1,\dots,y_{d+1},z)= f_k(y_1,\dots,y_{d+1},1,z).
\]
By Lemma \ref{lem:orthogonal}, we have:
\begin{align*}
     q  \# W_{S} (\mathbb{F}_q)  &= \sum_{
     \substack{w_{d+2}\in \F_q \\ y_1,\dots,y_{d+1},z_1,\dots,z_d \in \F_q^{\times} \\ g_k(y,z)=\sigma_k, \ k=0,\dots,d }} \psi\left(w_{d+2}\sum_{j \notin S} y_j\right) \\
     &= \sum_{
     \substack{y_1,\dots,y_{d+1},z_1,\dots,z_d \in \F_q^{\times} \\ g_k(y,z)=\sigma_k, \ k=0,\dots,d}} 1 + \sum_{
     \substack{w_{d+2}\in \F_q^\times \\ y_1,\dots,y_{d+1},z_1,\dots,z_d \in \F_q^\times \\g_k(y,z)=\sigma_k, \ k=0,\dots,d}} \psi\left(w_{d+2}\sum_{j \notin S} y_j\right).
\end{align*}
We make the change of variables $w_j = w_{d+2}y_j$,  $j=1,\dots, d+1$, then
\begin{align*}
      q  \# W_{S} (\mathbb{F}_q)  &= \sum_{
     \substack{y_1,\dots,y_{d+1},z_1,\dots,z_d \in \F_q^{\times} \\ g_k(y,z)=\sigma_k, \ k=0,\dots,d}} 1 + \sum_{
     \substack{w_1,\dots, w_{d+2}, z_1, \dots, z_d \in \F_q^\times \\f_k(w,z)=\sigma_k, \ k=0,\dots,d}} \psi\left(\sum_{j\notin S} w_j\right)\!.
\end{align*}
The first term is the number of points of the dehomogenisation of the equations defining $V$; thus it is the number of points of a torus of dimension $d$. Thus, 
\begin{align*}
       q  \# W_{S} (\mathbb{F}_q)  &= (q-1)^{d} + \sum_{
     \substack{w_1,\dots, w_{d+2}, z_1, \dots, z_d \in \F_q^\times \\f_k(w,z)=\sigma_k, \ k=0,\dots,d}} \psi\left(\sum_{j\notin S} w_j\right)\!\\
    &= (q-1)^{d} + \sum_{(w,z) \in V(\F_q) } \psi\left(\sum_{j \notin S} w_j\right).
    \end{align*}
The last sum depends on $t, \sigma_1,\dots, \sigma_d$, so we consider $q \# W_{S}(\F_q)$ as a function of the variables $t, \sigma_1,\dots, \sigma_d$ and compute its Fourier expansion.
By Lemma \ref{lem:fourier-series}, we have 
\begin{align*}
 q \# W_{S}(\F_q) = (q-1)^{d} + \frac{1}{ (q-1)^{d+1}}\sum_{(m,\lambda) \in \left(\Z/q^\times\Z\right)^{d+1}} c_{m,\lambda} \cdot   \chi^{m}(t) \prod_{k=1}^d \chi^{\lambda_k} (\sigma_k),
\end{align*}
where 
\begingroup
\allowdisplaybreaks
\begin{align*}
    c_{m,\lambda}  &= \sum_{(t,\sigma) \in \left(\F_q^\times\right)^{d+1}} \sum_{(w,z) \in V(\F_q)} \psi \left(\sum_{j\notin S} w_j\right) \chi^{-m}(t) \prod_{k=1}^d \chi^{-\lambda_k} (\sigma_k) \\
    &=  \sum_{(w,z) \in \left(\F_q^\times\right)^{2d+2} } \psi \left(\sum_{j\notin S} w_j\right) \chi^{-m}(w^\gamma) \prod_{k=1}^d \chi^{-\lambda_k}\left( w^{\rho_k}z^{-\mathcal{N}_k}\right) \\
    &=  \prod_{j \notin S}  \sum_{w_j \in \F_q^\times} \psi(w_j) \chi^{-m \gamma_j -\sum_{k=1}^d \lambda_k\rho_{kj}}(w_j) \cdot \prod_{j \in S}  \sum_{w_j \in \F_q^\times} \chi^{-m \gamma_j -\sum_{k=1}^d \lambda_k\rho_{kj}}(w_j) \\
    &\qquad \cdot  \prod_{i=1}^d \sum_{z_i \in \F_q^\times} \chi^{\sum_{k=1}^d \lambda_k\mathcal{N}_{ki}}(z_i) \\
    &=  \prod_{j \notin S}   g\left(-m\gamma_j -  \sum_{k=1}^d \lambda_k \rho_{kj}\right) \cdot \prod_{j\in S} (q-1)\epsilon\left( -m \gamma_j -\sum_{k=1}^d \lambda_k\rho_{kj} \right) \\
    & \qquad \cdot  \prod_{i=1}^d (q-1) \epsilon \left( \sum_{k=1}^d \lambda_k\mathcal{N}_{ki} \right) \\
    &= (-1)^{|S|} (q-1)^{d+|S|}\cdot \prod_{j=1}^{d+2} g\left(-m\gamma_j -  \sum_{k=1}^d \lambda_k \rho_{kj}\right)   \\
    &\qquad \cdot\prod_{j\in S} \epsilon\left( -m \gamma_j -\sum_{k=1}^d \lambda_k\rho_{kj} \right) \cdot \prod_{i=1}^d \epsilon \left( \sum_{k=1}^d \lambda_k\mathcal{N}_{ki} \right).
\end{align*}
\endgroup
The product $\prod_{i=1}^d \epsilon \left( \sum_{k=1}^d \lambda_k\mathcal{N}_{ki} \right)$ is $1$ if $\lambda \in \Lambda(q)$, and is $0$ if $\lambda \notin \Lambda(q)$.  Thus, we have 
\begin{align*}
    q \# W_{S} (\mathbb{F}_q) &= (q-1)^{d} + (-1)^{|S|}(q-1)^{|S|-1}\\
    &\qquad\cdot\sum_{\lambda \in \Lambda(q)}  \sum_{m=0}^{q^\times -1}\prod_{j\in S} \epsilon\left( -m \gamma_j -\delta(j,\lambda) \right)\prod_{j=1}^{d+2}g\left(-m\gamma_j - \delta(j,\lambda)\right) \chi^{m}(t) \chi^{\lambda}(\sigma),
\end{align*}
as desired.
\end{proof}
\begin{rem}
The size of the set $\Lambda(q)$ varies with $q$. For instance, if  $\deg(Z)$ divides $q^\times$, then the maximum possible value of $|\Lambda(q)|$ is attained, namely $\deg(Z)$. The computation done in this way allows us to count the number of $\F_q$-points on $W_S$ for all prime powers $q$ without any congruence restrictions on $q.$
\end{rem}

\begin{rem}\label{rem:zero-components}
Note that for any $\lambda$, we have
\begin{align*}
&\sum_{m=0}^{q^\times -1}\prod_{j\in S} \epsilon\left( -m \gamma_j -\delta(j,\lambda) \right)\prod_{j=1}^{d+2}g\left(-m\gamma_j - \delta(j,\lambda)\right) \chi^{m}(t) \chi^{\lambda}(\sigma)\\
&= \prod_{\ga_j =0} g(-\delta(j,\lambda))\cdot  \sum_{m=0}^{q^\times-1}\prod_{j\in S} \epsilon\left( -m \gamma_j -\delta(j,\lambda) \right)\prod_{\ga_j \neq 0}g\left(-m\gamma_j - \delta(j,\lambda)\right) \chi^{m}(t) \chi^{\lambda}(\sigma).
\end{align*}
Thus, the zero components of the Gale dual $\ga$ contribute to the point count of $W_S$ by multiplication by fixed Gauss sums. 
\end{rem}

\subsection{Regularity} We discuss some regularity properties of the hypersurface $Z$ that are useful for studying the regularity of the compactifications we will consider.

\begin{defi}\label{def:Delta}\cite[Definition 2.6]{Denef-Loeser}
 The Laurent polynomial $f$ and the hypersurface $Z$ are called \emph{$\Delta$-regular} if, for each face $F$ of $\Delta$, the variety $V(f_{|F}) \subset\G_m^{d}$ is smooth (over $k$).  
\end{defi}
\begin{prop}\label{prop:cover-smooth}
Suppose that $\ga_j \neq 0$ for all $j$, and let $L= \operatorname{lcm}\{ \gcd( \ga_i, \ga_j)\mid \ga_i \ga_j <0 \}$. Suppose that $(q,L \deg(Z))=1$, and consider the identity
\begin{equation}\label{eq:critical}
\prod_{\ga_j>0}\ga_j^{\ga_j}\prod_{\ga_j<0}u_j^{-\ga_j} \;=\; \prod_{\ga_j<0}\ga_j^{-\ga_j}\prod_{\ga_j>0}u_j^{\ga_j}.
\end{equation}
Then, we have the following properties.
\begin{enumerate}
    \item If \eqref{eq:critical} does not hold, then the hypersurface $Z$ is $\Delta$-regular.
    \item If \eqref{eq:critical} holds, then the hypersurface $Z_{\overline{\F}_q}$ has exactly $\deg(Z)$ ordinary double points.
\end{enumerate}
\end{prop}
\begin{proof}
Both statements are geometric, so we work over $\overline{\F}_q$.

We first reduce to the primitive case. Suppose that $\Delta$ is not primitive. Since $(q,\deg(Z))=1$, the covering
\[
Z \to Z^{\operatorname{prim}} = V(f^{\operatorname{prim}})
\]
from Proposition~\ref{prop:cover} is \'etale of degree $\deg(Z)$. Moreover, $Z^{\operatorname{prim}}$ has the same $\ga$ and the same $u_j$, so it satisfies the hypotheses of the proposition with the same identity \eqref{eq:critical}. Furthermore, for any face $F$ of $\Delta$, the hypersurface $V(f_{|F})$ is an \'etale covering of the hypersurface defined by the restriction of $f^{\operatorname{prim}}$ to the corresponding face of its Newton polytope. Since \'etale morphisms are smooth and preserve the nature of singularities, (1) follows from the primitive case. In case (2), the singular locus of $Z$ is the fibre over the unique ordinary double point of $Z^{\operatorname{prim}}$. This fibre consists of exactly $\deg(Z)$ ordinary double points, which proves (2).

From now on, we assume that $\Delta$ is primitive. In characteristic $0$, a proof is given in \cite[Proposition 3.17]{asem-giulia}. The same argument works in characteristic $p>0$, once we know that $p$ does not divide any $\ga_j$ in case (2), and that $V(f_{|F})$ is smooth for every proper face $F$ of $\Delta$.

Since $(q,L)=1$, the prime $p$ cannot divide both $\ga_i$ and $\ga_j$ for a pair with $\ga_i\ga_j<0$. Hence, since $u_1,\dots,u_{d+2} \in \F_q^\times$, if $p \mid \ga_{j_0}$ for some $j_0$, then exactly one side of \eqref{eq:critical} vanishes, and we are in case (1). Thus, in case (2), all the $\ga_j$ are prime to $p$.

We now show that the proper faces are smooth. For $S \subsetneq \{1,\dots,d+2\}$, write $S^{\pm} \coloneqq \{j \in S \mid \operatorname{sgn}(\ga_j)= \pm 1 \}$ and $F_S \coloneqq \operatorname{conv}(m_{*j} \mid j \notin S)$, and let $\mathcal{S}$ be the set of such $S$ with $S^{+}$ and $S^{-}$ non-empty. By Proposition~\ref{prop:structure-polytope}, which we prove in Section~\ref{section:combinatorics}, the proper faces of $\Delta$ are the $F_S$ with $S \in \mathcal{S}$, and $m_{*j} \in F_S$ if and only if $j \notin S$. Let $S \in \mathcal{S}$, and let $M_S$ be the submatrix of $M$ consisting of the columns indexed by $j \notin S$. If $x \in \G_m^{d}(\overline{\F}_q)$ is a singular point of $V(f_{|F_S})$, then, by the Jacobian criterion, $(u_j x^{m_{*j}})_{j \notin S}$ is a non-zero vector in the kernel of $M_S$ over $\overline{\F}_q$. Thus, it suffices to show that $M_S$ has full column rank modulo $p$.

By Remark~\ref{rem:faces-are-simplices}, $F_S$ is a simplex, whose vertices are the points $m_{*j}$ with $j \notin S$. Let $\Lambda_S$ be the subgroup of $\Z^{d+1}$ generated by the columns of $M_S$. The order of the torsion subgroup of $\Z^{d+1}/\Lambda_S$ is the normalised volume $\operatorname{Vol}(F_S)$ of $F_S$ in the lattice of its affine span, and $M_S$ has full column rank modulo $p$ if and only if $p \nmid \operatorname{Vol}(F_S)$.

Denote by $e_1,\dots,e_{d+2}$ the standard basis of $\Z^{d+2}$, and write $\Z^{d+2} = \Z^{S} \oplus \Z^{S^c}$, where $\Z^{S}$, respectively $\Z^{S^c}$, is the subgroup spanned by the $e_j$ with $j \in S$, respectively $j \notin S$. Accordingly, we write $\ga = \ga_S + \ga_{S^c}$, so that $\ga_S = \sum_{j \in S} \ga_j e_j$. Note that $\Lambda_S = M(\Z^{S^c})$. Since $\Delta$ is primitive, the map $M \colon \Z^{d+2} \to \Z^{d+1}$ is surjective with kernel $\Z\ga$. Thus,
\[
\Z^{d+1}/\Lambda_S = \Z^{d+1}/M(\Z^{S^c}) \cong \Z^{d+2}\big/\big(\Z^{S^{c}}+\Z\ga\big) \cong \Z^{S}/\Z\ga_S \cong \Z^{|S|-1} \oplus \Z/g\Z,
\]
where $g \coloneqq \gcd(\ga_j \mid j \in S)$. Hence,
\begin{equation}\label{eq:volume-of-faces}
\operatorname{Vol}(F_S)=\gcd(\ga_j \mid j \in S).
\end{equation}
Choosing $i \in S^{-}$ and $j \in S^{+}$, we see that $\operatorname{Vol}(F_S)$ divides $\gcd(\ga_i,\ga_j)$, hence divides $L$. Since $(q,L)=1$, we get $p \nmid \operatorname{Vol}(F_S)$, as required.
\end{proof}

We will often appeal to the following well-known lemma.
\begin{lem}\label{lem:quasi-smooth-compactification}
Let $f$ be a Laurent polynomial with Newton polytope $\Delta$, and let $Z= V(f)$ be the hypersurface it defines in $\mathbb{G}_m^d$. Suppose that $f$ is $\Delta$-regular.  Let $\Sigma$ be the normal fan of the Newton polytope $\Delta$, and let $\Sigma'$ be any refinement of $\Sigma$.  Denote by $\mathbb{P}_{\Sigma'}$ the toric variety associated to the fan $\Sigma'$. 
Let $\overline{Z}$ be the closure of $Z$  in $\mathbb{P}_{\Sigma'}.$ If $\mathbb{P}_{\Sigma'}$ is quasi-smooth, i.e., has only finite quotient singularities, then $\overline{Z}$ is quasi-smooth. 
\end{lem}
\begin{proof}
This follows immediately from the proof of \cite[Lemma 2.3]{Denef-Loeser}.
\end{proof}

\section{Combinatorics of the Newton polytope} \label{section:combinatorics}
To ease exposition, we postpone all proofs to the end of the section. Recall that we are considering a Laurent polynomial 
\[f(x) =\sum_{j=1}^{d+2} u_j \prod_{i=1}^d x_i^{m_{ij}},\]
with Gale dual $\gamma \in \mathbb{Z}^{d+2}$, which satisfies $\sum_{j=1}^{d+2} \gamma_j=0$ and $\gcd(\ga_1,\dots,\ga_{d+2})=1$. 
We assume that the Newton polytope $\Delta$ of $f$ is $d$-dimensional. 
One of the nice aspects of this situation is that the combinatorics of $\Delta$ can be described in terms of the Gale dual $\ga$. Recall that the matrix $M$ was defined by
\begin{equation}\label{eq:matrix-M-combinatorics}
M=\begin{pmatrix}
1 & 1 &\cdots &1 \\
m_{11} & m_{12} & \cdots & m_{1,d+2} \\
m_{21} & m_{22} & \cdots & m_{2,d+2} \\
\vdots & \vdots & \ddots & \vdots \\
m_{d1} & m_{d2} & \cdots & m_{d,d+2}
\end{pmatrix}. 
\end{equation}

For simplicity, we will assume that $\ga_j \neq 0$ for all $j$ (see Remark \ref{rem:zero-components}).
By reordering the monomials, we may suppose that  $\gamma_1,\dots, \gamma_r < 0$ and $\gamma_{r+1},\dots,\gamma_{r+s} >0$, where $r+s =d+2.$ 
Denote by $\ga_- = \{ j \mid \ga_j <0\} =\{1,\dots,r\} $ and  $\ga_+ = \{ j \mid \ga_j >0\}  = \{r+1,\dots,r+s\}.$ 
For any subset $S \subset \{1,\dots,d+2\}$, we define $S^{+}\coloneqq S \cap \ga_{+}$ and $S^{-} \coloneqq S \cap \ga_{-}$.
Let $\mathcal{S}$ be the set of proper subsets $S$ of $\{1,\dots,d+2\}$ such that $S^+$ and $S^{-}$ are both non-empty.  

We have the following description of the proper faces\footnote{
We adopt the convention that the empty set is not a face of  $\Delta$.
} of the polytope $\Delta$.
\begin{prop}\label{prop:structure-polytope}
The maps
\begin{align*}
S &\mapsto F_S\coloneqq \operatorname{conv} \big((m_{1j},\dots,m_{dj}) \mid j \notin S \big),\\
F &\mapsto S_F \coloneqq\{ j \mid (m_{1j},\dots,m_{dj}) \notin F \}
\end{align*}
are bijections between the set $\mathcal{S}$ and the set of proper faces of $\Delta$, and are inverse to one another. Furthermore, if $S \subset S'$ are elements of $\mathcal{S}$,  then $F_S \supset F_{S'}$, and vice versa.
\end{prop}
We illustrate Proposition \ref{prop:structure-polytope} in dimension $2$. 
Let $m_{*1},\dots,m_{*4} \in \Z^{2}$ be four distinct lattice points in the plane, and let $\Delta$ be their convex hull. 
We assume that the four points do not lie on a line, and that the associated Gale dual $\ga$ has no zero components.
The following are the only possible cases.
\begin{enumerate}
    \item The four points are vertices of $\Delta$, in which case $\Delta$ is a quadrilateral.
    \item Only three of the points are vertices of $\Delta$, so $\Delta$ is a triangle, and the last point lies in $\Delta$.  We have two sub-cases:
    \begin{enumerate}
        \item the last point lies in the interior of the triangle; or
        \item the last point lies in the relative interior of one of the edges. In this case, there would be an affine linear relation between the points lying on this edge, and the Gale dual would contain a 0; thus, this case does not occur.
    \end{enumerate}
\end{enumerate}
The signs of the Gale dual in the two cases above can be chosen (black and white) as in the following figure:
\begin{figure}[H]
    \centering
    \includegraphics[width=6cm, height=5cm]{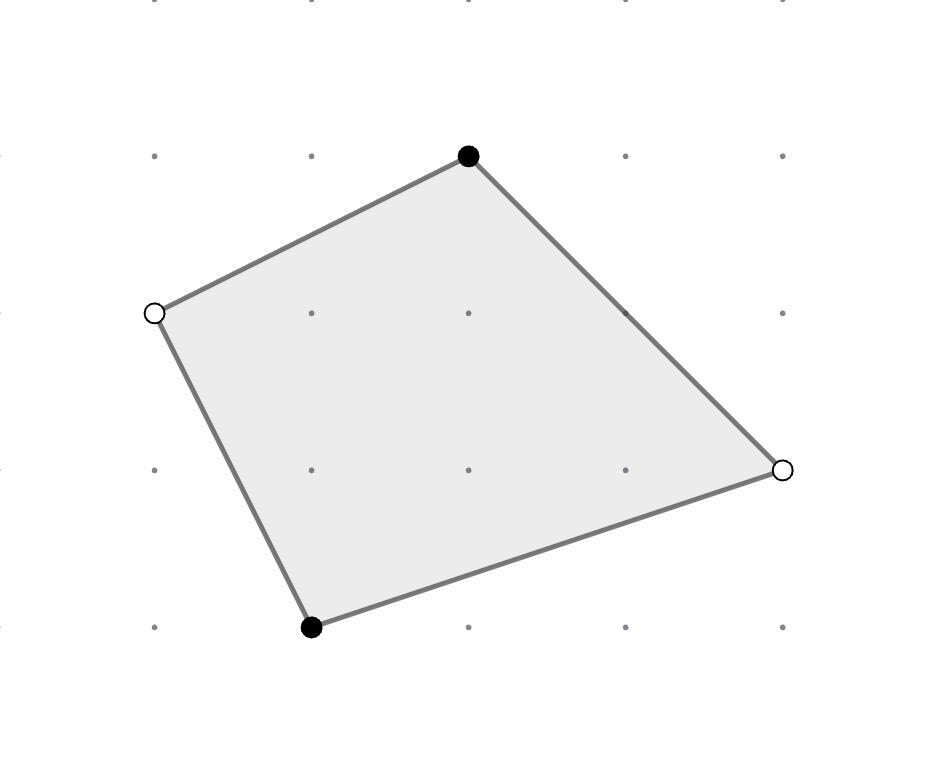}
    \includegraphics[width=6cm, height=5cm]{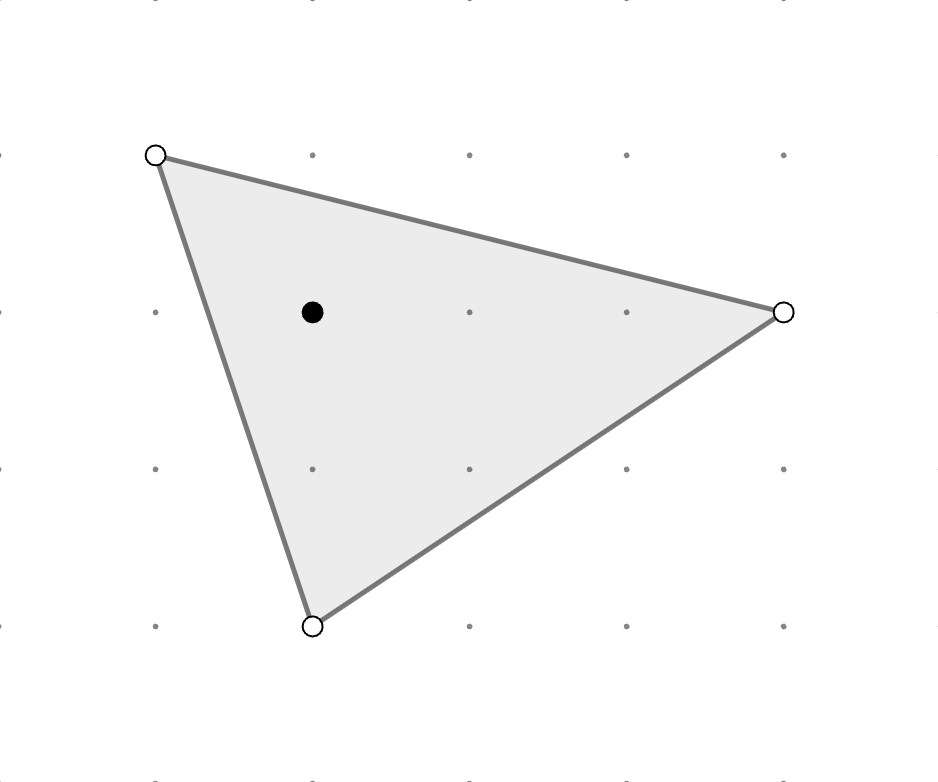}
    \caption{{Case} $d=2$}
\end{figure}
In both cases of the diagram, it is clear that the convex hull of two points is an edge if and only if their complement contains both a white and a black dot. Furthermore, a point $m_{*k}$ is a vertex if and only if its complement contains white and black dots.

\begin{rem}\label{rem:faces-are-simplices}
Many nice properties of the polytope $\Delta$ can be deduced from Proposition \ref{prop:structure-polytope}. For example, the polytope $\Delta$ is simplicial, i.e., all its proper faces are simplices. Indeed, the facets are the convex hulls of the points indexed by the complements of the sets $S \in \mathcal{S}$ of cardinality $2$.
Thus, the $(d-1)$-dimensional faces are the convex hull of $d+2-2$ points, so they are simplices. Since any proper face of $\Delta$ is a face of some facet, all proper faces of $\Delta$ are simplices.
\end{rem}
\subsection*{The normal fan} 
Our goal now is to describe the normal fan $\Sigma=\Sigma(\Delta)$ of the polytope $\Delta$. A good reference for the normal fan of a lattice polytope is \cite[\S 2.3]{cox}.
We recall that to each face $F$ of $\Delta$ of dimension $k$, there is an associated cone $\sigma_F$ of $\Sigma$ of dimension $d-k$, called the normal cone to $F$. 
Conversely, each cone $\sigma$ of $\Sigma$ is the normal cone of some face $F_{\sigma}$ of $\Delta$.
The normal cone to the whole polytope $\Delta$ is, by definition, the zero cone.
Furthermore, the rays of $\Sigma$, i.e., the one-dimensional cones, are spanned by the (inward) normals to the facets of {$\Delta$}. 
By Proposition \ref{prop:structure-polytope}, they are in one-to-one correspondence with pairs $(i,j)$, where $1 \le i \le r$ and $r+1 \le j \le r+s$. 
We abuse notation and write $(i,j)$ for such a ray. Denote by $e_i$ the $i$-th standard basis vector of $\R^{d+2}$.  
We have the following description of the inward normals to the facets of $\Delta$.

\begin{lem}\label{lem:normals}
Let $i$ and $j$ be such that $1 \le i \le r$ and $r+1 \le j \le r+s$, and define $x^{(i,j)} \coloneqq -\ga_i e_j + \ga_j e_i$.
Then, there is a unique solution $(\al_0^{(i,j)},\al^{(i,j)}) \in \R\times\R^{d}$ of the matrix equation $(\al_0^{(i,j)},\al^{(i,j)}) M = x^{(i,j)}.$  Furthermore, $\al^{(i,j)}$ is an inward normal to the facet corresponding to $(i,j)$.
\end{lem}

By construction of the normal fan, the normal cone to a face $F$ is the convex cone generated by the inward normals to the facets that contain $F$.
Let $F$ be a face of $\Delta$, and let $S_F = \{ j \mid m_{*j} \notin F \}$; in particular, $S_\Delta = \emptyset$.
By Proposition \ref{prop:structure-polytope}, the set of facets which contain $F$ is in bijection with the {set} $\{ (i,j) \mid i \in S_F^{-} \text{ and } j \in S_F^{+}\}$. 
Finally, note that $ \sigma_F \subset \sigma_E$ if and only if  $E \subset F$, which is true if and only if $S_F \subset S_E$. To summarise, we have the following proposition:
\begin{prop}\label{prop:normal-fan} \hfill
\begin{enumerate}
    \item The rays of the normal fan are in one-to-one correspondence with pairs $(i,j)$, where $1 \le i \le r$ and $r+1 \le j \le r+s$. 
    \item The proper faces of $\Delta$ are in bijection with the non-zero cones of $\Sigma$ via
    \begin{align*}
        F \mapsto {\operatorname{cone}}\big( (i,j) \mid i \in S_F^{-} \text{ and } j \in S_F^{+} \big).
    \end{align*}
    \item If $\sigma_F$ and $\sigma_E$ are two cones of $\Sigma$ corresponding to faces $F$ and $E$, respectively, {then}
    \[
    \sigma_F \subset \sigma_E \iff S_F \subset S_E.
    \] 
\end{enumerate}
\end{prop}
The descriptions of the normal fan $\Sigma$ and the Newton polytope $\Delta$ above do not depend on whether the polytope is primitive or not. This is key to generalising the results of \cite{BCM}, as it allows us to describe, as we shall see below, the boundary components for refinements of $\Sigma$ when $\Delta$ is not a primitive polytope. 

\begin{rem}
From the description above, it is not hard to see that the fan $\Sigma(\Delta)$ is not simplicial\footnote{Recall that a fan is called simplicial if all of its cones are simplicial.} unless the Gale dual of $\Delta$ is a vector $\gamma$ such that $r=1$, $s=1$, or $r=s=2$. 
The latter case corresponds to $\Delta$ being a quadrilateral in the plane, while the former cases correspond to $\Delta$ being a simplex. 

To give an example in which the fan $\Sigma$ is not simplicial, consider a polytope $\Delta$ with Gale dual $(\ga_1,\dots,\ga_{5})=(-30,-1,6,10,15)$ (see \cite[\S 3]{villegas-integral-factorial}).
The normal cone to the vertex $F$ with $S_F= \{1,2,3,4\}$ is a $3$-dimensional cone, but its minimal generators are the four rays $(1,3),(1,4),(2,3)$ and $(2,4)$. Thus, the fan is not simplicial.
\end{rem}

\subsection*{Refinement of the normal fan} 
Now, we want to describe a simplicial refinement of $\Sigma$. 
We use the indexing of the ``staircase'' refinement from \cite{BCM} to construct a refinement of $\Sigma$, which we denote by $\Sigma'$.  
Let $\mathcal{C}$ be the set of all sequences $C=((i_1,j_1),\dots, (i_l,j_l))$ of pairwise-distinct pairs satisfying $1\le i_1 \le \dots \le i_l \le r$ and $r+1\le j_1 \le \dots \le j_l \le r+s$.  
The length of a sequence $C \in \mathcal{C}$ is $l_C=l$. The support $S_C$ of $C$ is the set $\{i_1, \dots, i_l,j_1,\dots, j_l\}$.  
Associated to a sequence $C \in \mathcal{C}$ is the cone $\sigma_C \subset \R^{d}$ generated by the rays $(i_1,j_1),\dots, (i_l,j_l)$. 
We associate the zero cone to the empty sequence.
A sequence $C$, or cone $\sigma_C$, is said to be maximal if $S_C =\{1,\dots,d+2\}$.
\begin{defi}\label{defi:fan-refinement}
We define $\Sigma'$ to be the collection of cones $\sigma_C$ that are not maximal.
\end{defi}

\begin{prop}\label{prop:sigma-is-a-fan}
    The collection of cones $\Sigma'$ is a simplicial fan.
\end{prop}

\begin{rem}\label{rem:first-meet}
Given a cone $\sigma_{C}$ of $\Sigma'$ associated to a sequence $C=((i_1,j_1),\dots, (i_l,j_l))$, in Step (3) of the proof of Proposition \ref{prop:sigma-is-a-fan} below, we show that {$\sigma_C\subset \sigma_F$} with $S_F = \{i_1,\dots,i_l,j_1,\dots,j_l\}$. 
In fact, by Proposition \ref{prop:normal-fan}, every cone of $\Sigma$ which contains $\sigma_{C}$ contains $\sigma_F$. Thus, $\sigma_F$ is the smallest cone of $\Sigma$ containing $\sigma_C$.
Note that this implies that the first meet locus of $\sigma_C$, as defined in \cite[\S 2.4]{Denef-Loeser}, is the face $F$ of $\Delta$.  Indeed, the relative interior of $\sigma_C$ is contained in the relative interior of $\sigma_F$, and the first meet locus of a point in the relative interior of $\sigma_F$ is $F$. 
\end{rem}

\begin{prop}\label{prop:sigma-prime-is-a-refinement}
   The fan $\Sigma'$ is a refinement of $\Sigma$.
\end{prop}

\begin{rem}
  The fan $\Sigma'$ being simplicial implies that the proper toric variety  $\Ps_{\Sigma'}$ associated to $\Sigma'$ is quasi-smooth, i.e., it has only finite quotient singularities (see \cite{batyrev-cox}). 
  This is not so far from being smooth; for instance, over $\C$, the cohomology group $H^k(\Ps_{\Sigma'}(\C),\C)$ is a pure Hodge structure of weight $k$, for each $k$ (see \cite{steenbrink-vanishing}). 
\end{rem}

The rest of this section is devoted to proving the results above.
We note that Proposition \ref{prop:structure-polytope} is well-known; it is a special case of \cite[Lemma 4.1.38 (iii)]{triangulations-book} whose proof uses \emph{circuits and cocircuits}. We provide our own elementary proof, since we need the notation and ideas to prove the rest of the statements. 

\begin{proof}[Proof of Proposition \ref{prop:structure-polytope} and Lemma \ref{lem:normals}]
Let $F$ be a proper face of $\Delta$. 
Then, $F= \Delta \cap H$ for some supporting hyperplane $H$ which is given by
\[
 \al \cdot (y-m_{*j_0}) =0
\]
for some non-zero $\al \in \R^{d},$ which we assume to be an inward normal to $\Delta$, and $m_{*j_0} \in F$. Define  $x \in \R^{d+2}$ by $x_j \coloneqq \al \cdot (m_{*j}-m_{*j_0})$ for $j=1,\dots,d+2$. Then, 
\begin{align*}
    x_j= 0\quad &\text{ if and only if } m_{*j} \in F, \\
    x_j> 0\quad  &\text{ if and only if } m_{*j} \notin F.
\end{align*}
Thus, $x \in \R^{d+2}_{\ge0}$ and is non-zero and $x_{j_0}=0$. Furthermore, 
\begin{align*}
\sum_{j=1}^{d+2} \ga_j x_j &= \sum_{j=1}^{d+2} \ga_j \al \cdot (m_{*j}-m_{*j_0})\\
&= \sum_{i=1}^d \al_i \sum_{j=1}^{d+2}  \ga_j (m_{ij}-m_{ij_0})=0.
\end{align*}
Conversely, let $x \in \R^{d+2}_{\ge 0}$ be non-zero, with $x_{j_0}=0$ for some $j_0$ and $\sum_{j=1}^{d+2} \ga_jx_j =0$. 
We call such a vector $x$ a relation. 
Since the rows $\mathbf{1}, m_1,\dots, m_d$ of $M$ are an $\R$-basis of the kernel of $z \mapsto \ga\cdot z$,  there exists a unique $(\al_0^x,\al^x) \in \R\times \R^{d}$ such that
\[
(\al_0^x,\al^{x}) M = x.
\]
We associate to $x$ the supporting hyperplane
\[
H_x:  \al^x \cdot (y-m_{*j_0}) =0.
\]
Then, by definition of $\al^x$,
\begin{align*}
    \al^x \cdot (m_{*j}-m_{*j_0})= 0\quad &\text{ if and only if } m_{*j} \in H_x \cap \Delta, \\
    \al^x \cdot (m_{*j}-m_{*j_0})>0\quad  &\text{ if and only if } m_{*j} \notin H_x\cap \Delta.
\end{align*}
Note that two relations $x$ and $x'$ define the same face if and only if  $x$ and $x'$ have the same zero components. 
Two such relations are said to be equivalent. Thus, there is a one-to-one correspondence between the set of relations up to equivalence and the set of proper faces of $\Delta$  given by 
$x \mapsto H_x \cap \Delta.$
We denote the inverse of this map by $F \mapsto [x_F]$.

Now, note that since $x_j \ge 0$ for all $j$ and $\sum_{j=1}^{d+2} \ga_j x_j =0,$ we have $x_j > 0 $ for some $j \le r$ and $x_j >0$ for some $j >r$. 
Conversely, every $S \in \mathcal{S}$ is the set of indices of the non-zero components of some relation: since each $x^{(i,j)}$ is non-negative and satisfies $\ga\cdot x^{(i,j)}=0$, the vector $x=\sum_{i \in S^-,\, j \in S^+} x^{(i,j)}$ is a relation, and its non-zero components are exactly those indexed by $S$.
Furthermore, if $F \subset E$ are faces of $\Delta$, then the set of indices $ i$ such that $(x_F)_i=0$ is a subset of the set of indices $ i$ such that $(x_E)_i=0$, and vice versa. This proves Proposition \ref{prop:structure-polytope}.

Finally, note that since the facets of the polytope are maximal with respect to inclusion, we get that facets correspond to relations $x$ with exactly two non-zero components. The facet corresponding to $(i,j)$ is the facet corresponding to the class $[x^{(i,j)}]$. 
Altogether, this proves Lemma \ref{lem:normals}.
\end{proof}

\begin{lem}\label{lem:staircase}
Let $x \in \R^{d+2}_{\ge 0}$ with $\sum_{j=1}^{d+2} \ga_j x_j = 0$. Then, there exist unique $l \ge 0$, $c_1,\dots,c_l > 0$ and $C = ((i_1,j_1),\dots,(i_l,j_l)) \in \mathcal{C}$ such that
\[
x = \sum_{k=1}^{l} c_k\, x^{(i_k,j_k)}.
\]
Moreover, $S_C = \{ j \mid x_j \neq 0 \}$.
\end{lem}

\begin{proof}
Since $x^{(i,j)}$ has entries $\ga_j$ at $i$ and $|\ga_i|$ at $j$, and zeros elsewhere, any such expression is non-zero exactly at the indices in $S_C$; this gives the last claim. If $x=0$, then $C$ is empty. Otherwise, let $i(x)$ and $j(x)$ be the smallest indices in $\ga_-$ and in $\ga_+$ at which $x$ is non-zero (they both exist since $\sum_j \ga_j x_j = 0$), and let
\[
c(x) = \min\left\{ \frac{x_{i(x)}}{|\ga_{j(x)}|}, \frac{x_{j(x)}}{|\ga_{i(x)}|}\right\}.
\]
Then, $y \coloneqq x - c(x)\, x^{(i(x),j(x))}$ is non-negative, satisfies $\sum_j \ga_j y_j = 0$, and vanishes at $i(x)$ or $j(x)$; if $y \neq 0$, then $i(x) \le i(y)$ and $j(x) \le j(y)$, with $(i(y),j(y)) \neq (i(x),j(x))$. Repeating with $y$ in place of $x$, we obtain, after finitely many steps, an expression of $x$ as in the statement.

For uniqueness, let $x = \sum_{k=1}^{l} c_k\, x^{(i_k,j_k)}$ be such an expression with $x \neq 0$. Since the pairs are weakly increasing and $x$ is non-zero exactly at the indices in $S_C$, we have $(i_1,j_1) = (i(x),j(x))$. Since the pairs are also distinct, $i_1$ or $j_1$ occurs only in the first pair. Both ratios in the definition of $c(x)$ are at least $c_1$, with equality for that index, so $c_1 = c(x)$. Thus, the first term is determined by $x$, and induction on $l$, applied to $y$, determines the rest.
\end{proof}

\begin{proof}[Proof of Proposition \ref{prop:sigma-is-a-fan}]
The proof is divided into a few steps.

\noindent\textbf{(1)} $\Sigma'$ is simplicial:
    let $\sigma_C \in \Sigma'$ be the cone associated to the sequence 
    \[
    C=((i_1,j_1), \dots, (i_l,j_l)).
    \]
    We prove that $\sigma_C$ is $l$-dimensional. To do so, we prove that the rays $(i_1,j_1),\dots,(i_l,j_l)$ are linearly independent.
    Let $\al^{(i_k,j_k)}$ be the inward normal to the facet $(i_k,j_k)$ corresponding to the relation $x^{(i_k,j_k)}$. Recall that $(\al_0^{(i_k,j_k)}, \al^{(i_k,j_k)})\in \R^{d+1}$ is the unique vector such that 
    \[
    \al_0^{(i_k,j_k)} \mathbf{1} + \sum_{t=1}^{d} \al^{(i_k,j_k)}_tm_t = x^{(i_k,j_k)}.
    \]
    Suppose for contradiction that $\al^{(i_1,j_1)},\dots,\al^{(i_l,j_l)}$ are linearly dependent. Without loss of generality, we can assume that there exist non-zero $r_1,\dots,r_l$ such that
    \[
    \sum_{k=1}^{l} r_k {\al}^{(i_k,j_k)}=0.
    \]
    Then, we have
    \begin{align*}
         \sum_{k=1}^{l} r_kx^{(i_k,j_k)}= \sum_{k=1}^l r_k \left(\al_0^{(i_k,j_k)} \mathbf{1} + \sum_{t=1}^{d} \al^{(i_k,j_k)}_t m_t\right) =
         \sum_{k=1}^l r_k \al_0^{(i_k,j_k)} \mathbf{1}.
    \end{align*}
    Thus, for each $s \in \{1,\dots,d+2\}$, we have 
    \begin{equation}\label{eq:simplicial-proof}
        \sum_{k=1}^{l} r_kx^{(i_k,j_k)}_s = c \coloneqq \sum_{k=1}^l r_k \al_0^{(i_k,j_k)} .
    \end{equation}
    We have two cases:
    \begin{enumerate}
        \item[(a)] If $c=0,$ then since $(i_1,j_1) \neq (i_2,j_2)$, either $i_2>i_1$ or $j_2>j_1$. Suppose, without loss of generality, that $i_2>i_1$. Then, the index $i_1$ only appears once in the sequence $C$.  Thus,
        for $s=i_1$ in Equation \eqref{eq:simplicial-proof}, we have 
        \[
        \sum_{k=1}^{l} r_kx^{(i_k,j_k)}_{i_1} =  r_{1} x^{(i_1,j_1)}_{i_1} = 0.
        \]
        Since $x^{(i_1,j_1)}_{i_1} \neq 0,$ we have $r_{1}=0$, which is a contradiction.
        \item[(b)] If $c \neq 0$, then for each $s \in
         \{1 , \dots , d+2\}$, there is $k$ such that $x_s^{(i_k,j_k)} \neq 0$. This implies that $\bigcup_{k=1}^{l} \{i_k,j_k\} = \{1,\dots,d+2\}$, i.e., $C$ is maximal, which is a contradiction.
    \end{enumerate}
     
\noindent \textbf{(2)} $\Sigma'$ is closed under intersections and faces: any subsequence of $C \in \mathcal{C}$ is also in $\mathcal{C}$. Furthermore, a face of $\sigma_C$ is given by $\sigma_{D}$ for some subsequence $D$ of $C$. 
Now, let $D \in \mathcal{C}$ be any non-maximal sequence, and let $\al \in \sigma_C \cap \sigma_D$. As in (1), its expressions in the rays of $C$ and $D$ lift to non-negative vectors that differ by a multiple of $\mathbf{1}$; since $C$ and $D$ are not maximal, both have a zero component, so they are equal. By Lemma \ref{lem:staircase}, both expressions use the same pairs with non-zero coefficients, so $\al \in \sigma_{C \cap D}$, where $C \cap D$ is the subsequence of common pairs. Hence $\sigma_C \cap \sigma_D = \sigma_{C \cap D}$, a face of both.
    
\noindent \textbf{(3)} Each cone of $\Sigma'$ is strongly convex: we prove this by proving that each cone $\sigma_C$ of $\Sigma'$ is contained in some cone of $\Sigma$. This proves that $\sigma_C$ is strongly convex since it is a subset of a strongly convex cone. Suppose that $C= ((i_1,j_1), \dots, (i_l,j_l))$. We claim that $\sigma_C \subset \sigma_F$, where $F$ is the face of $\Delta$ defined by $S_F= \{i_1,\dots,i_l,j_1, \dots,{j_l}\}$ (note that $S_F$ is a proper subset of $\{1,\dots,d+2\}$ since $C$ is not maximal).  By Proposition \ref{prop:normal-fan}, for each $k$, the ray $(i_k,j_k)$ is contained in $\sigma_F$; thus, $\sigma_{C} \subset \sigma_{F}$.
\end{proof}

\begin{proof}[Proof of Proposition \ref{prop:sigma-prime-is-a-refinement}]
We need to show that each cone of $\Sigma$ is a union of cones of $\Sigma'$.
Since $\Sigma$ is the normal fan of a polytope, we have $\bigcup_{\sigma \in \Sigma} \sigma = \R^{d}$.
It suffices to show that
\[
\bigcup_{\sigma_C \in \Sigma'} \sigma_C =\R^{d}.
\]
Let $\al \in \R^{d}$ be a non-zero vector. We want to show that $\al$ is in some cone $\sigma_C$. There exists a unique $\al_0$ such that 
\begin{equation}\label{eq:proof-refinement-1}
    \al_0 \mathbf{1} + \sum_{i=1}^{d} \al_i m_i = x \in \R^{d+2}_{\ge 0}
\end{equation}
with $x_j \ge 0$ for all $j$, $\min_{j} x_j =0$, and $ \sum_{j=1}^{d+2}\ga_j x_j=0$.
Suppose that
\begin{equation} \label{eq:proof-refinement-2}
x= c_1 x^{(i_1,j_1)}+ \dots +c_{l} x^{(i_l,j_l)},
\end{equation}
for $c_k \ge 0 ,k=1,\dots, l$, and $i_1 \le \dots \le i_l$ and $j_1 \le \dots \le j_l$. We claim that $\al \in \sigma_C$, where $C = ((i_1,j_1),\dots,(i_l,j_l)).$
Indeed, writing the right-hand side in terms of the basis elements $ \mathbf{1}, m_1, \dots,m_d$, we get
\[
\al_k= \sum_{r=1}^{l} c_r \al_k ^{(i_r,j_r)}  \quad \text{ for } k=1,\dots, d. 
\]
Thus, $\al = \sum_{r=1}^l c_r \al^{(i_r,j_r)}$, as desired.
By Lemma \ref{lem:staircase}, $x$ has such an expression, and $C$ is not maximal since $S_C = \{ j \mid x_j \neq 0\}$ and $\min_j x_j = 0$.
\end{proof}

\section{Compactifications}\label{section:compactifications}
As before, let $Z$ be the hypersurface defined by a Laurent polynomial
\[f(x) =\sum_{j=1}^{d+2} u_j \prod_{i=1}^d x_i^{m_{ij}}\]
Let $q$ be a prime power. 
We associated to $f$ the following data:
\begin{enumerate}
    \item The {Gale dual} $\ga \in \Z ^{d+2} =\Z^{r+s}$ which we assume satisfies 
    \[\ga_1\le \dots \le \ga_r <0 < \ga_{r+1}\le\dots\le \ga_{r+s}.\]
    \item A matrix $\mathcal{N}$ and vectors $\rho_{1}, \dots, \rho_d \in\Z^{d+2}$, satisfying $\sum_{j=1}^{d+2} \rho_{kj}=0 $ for all $k$, such that $Z$ is isomorphic to the variety given by
    \begin{align*}
    \sum_{j=1}^{d+2} w_j&= 0, \quad  w_1\cdots w_{d+2} \neq 0,\\
        \prod_{j=1}^{d+2} w_j^{\gamma_j} &=t,\quad  t \coloneqq \prod_{j=1}^{d+2} u_j^{\gamma_j},\\
        \prod_{j=1}^{d+2} w_j^{\rho_{kj}} &=  \sigma_k \prod_{i=1}^d z_i^{\mathcal{N}_{ki}},  \qquad k=1,\dots,d,
    \end{align*}
    in $\mathbb{P}^{d+1} \times \G_m^{d}$, where $\sigma_1,\dots,\sigma_d$ are the numbers defined in item (5) below.
    \item The set $\Lambda(q)$ of solutions of the system of equations \begin{align*}
    \lambda_1 \mathcal{N}_{11}+ \dots +\lambda_d \mathcal{N}_{d1}&=0,\\
    \vdots \\
    \lambda_1 \mathcal{N}_{1d}+ \dots +\lambda_d \mathcal{N}_{dd}&=0
    \end{align*}
    in $(\Z/q^\times\Z)^{d}$.
    \item For each $\lambda \in \Lambda(q)$, we defined $\delta(\lambda) \in \left(\Z/q^\times\Z\right)^{d+2}$ by
    \begin{align*}
        \delta(j,\lambda) &\coloneqq \sum_{i=1}^d \lambda_i \rho_{ij} \in \Z/q^\times\Z \quad \text{ for } j=1,\dots, d+2 ,\\
        \delta(\lambda) &\coloneqq (\delta(1,\lambda),\dots,\delta(d+2,\lambda)) .
    \end{align*}
    \item  The numbers $\sigma_k \in \F_q^\times$, $k=1,\dots,d$, defined by 
    \[
    \sigma_k= \prod_{j=1}^{d+2} u_j^{\rho_{kj}}.
    \]  
\end{enumerate}
We want to count the number of $\F_q$-points of two compactifications of $Z$.

\subsection{Compactification I}\label{subsection:compactification-1}
The toric variety $\Ps_{\Delta}$  (see \cite{Denef-Loeser} and \cite{Danilov-Khovanskii}) is the toric variety associated to the normal fan $\Sigma$ of $\Delta$. 
It is a compactification of $\G_m^d$ and admits a decomposition into disjoint locally closed subvarieties
\begin{equation}\label{eq:toric-strati}
\Ps_{\Delta}= \bigcup_{F \preceq \Delta } \T_F,
\end{equation}
where $\T_F$ is a torus of dimension $\dim(F)$,  and the union is over all faces of the polytope, including $\Delta$ itself.

Let $\overline{Z}$ be the closure of $Z$ in $\Ps_{\Delta}$. For each face $F$, define $Z_F\coloneqq \overline{Z} \cap \T_F $. Then, $\overline{Z}$ admits a decomposition into disjoint subvarieties
\[
\overline{Z} = \bigcup_{F \preceq \Delta }  Z_F.
\]
Furthermore, if $f_{|F}$ is the restriction of the Laurent polynomial $f$ to the face $F$, then the variety $V(f_{|F}) \subset \G_m^{d}$ is isomorphic to the product  $\G_m^{d-\dim(F)} \times Z_F$ over $\F_q$.
Thus, to count the number of $\F_q$-points of $Z_F,$ we count the number of $\F_q$-points of $V(f_{|F})$ and divide it by $(q-1)^{d-\dim(F)}$.

Now, let $F$ be a proper face of $\Delta$ corresponding to a set $S_F$. The restriction of $f$ to the face $F$ is given by
\[
f_{|F} = \sum_{j \notin S_F} u_j \prod_{i=1}^d x_i^{m_{ij}}.
\]
Thus, by Corollary \ref{cor:cover}, $V(f_{|F})$ is isomorphic to the variety given by
\begin{align*}
   \sum_{j\notin S_F} w_j&= 0, \quad  w_1\cdots w_{d+2} \neq 0,\\
    \prod_{j=1}^{d+2} w_j^{\gamma_j} &=t, \quad t = \prod_{j=1}^{d+2} u_j^{\ga_j},\\
    \prod_{j=1}^{d+2} w_j^{\rho_{kj}} &= \sigma_k \prod_{i=1}^d z_i^{\mathcal{N}_{ki}},  \qquad k=1,\dots,d,
\end{align*}
in $\mathbb{P}^{d+1} \times \G_m^{d}$.

To ease notation, we set $ G_{m,\lambda}\coloneqq \prod_{j=1}^{d+2}g\left(-m\gamma_j + \delta(j,\lambda)\right) \chi^{m}(t) \chi^{-\lambda}(\sigma)$.
If $F$ is  a proper face of $\Delta$, then, by Proposition \ref{prop:main-count} with $\lambda$ replaced by $-\lambda$, we have
\begin{align}
   \#Z_F(\F_q) &= \#V(f_{|{F}})(\F_q) (q-1)^{-|S_F|+1} \nonumber\\
    &= q^{-1} (q-1)^{d-|S_F|+1} + q^{-1}\sum_{\lambda \in \Lambda(q)}  \sum_{m=0}^{q-2} G_{m,\lambda }(-1)^{|S_F|} \prod_{j\in S_F}\epsilon\left( -m \gamma_j + \delta(j,\lambda) \right).\label{eq:30}
\end{align}
Furthermore, the variety $Z_{\Delta}$ {coincides with} $Z$, since it is the intersection of $\overline{Z}$ with the maximal torus of $\Ps_{\Delta}$. Thus, by Proposition \ref{prop:main-count} (with $S$ the empty set, and again with $\lambda$ replaced by $-\lambda$), the number of $\F_q$-points of $Z_{\Delta}$ is given by
\begin{align}
    \#Z_\Delta(\F_q) &= q^{-1}(q-1)^{d} +q^{-1}(q-1)^{-1} \sum_{\lambda \in \Lambda(q)}  \sum_{m=0}^{q-2} G_{m,\lambda }.
\end{align}
Since the varieties $Z_F$ are disjoint, we have
\begin{equation}\label{eq:index-compactification-I}
  \#\overline{Z}(\F_q) = \sum_{F} \#Z_F(\F_q)= \#Z_{\Delta}(\F_q)+ \sum_{F\precneq\Delta} \#Z_F(\F_q).  
\end{equation}
Recall that a non-empty proper set $S \subsetneq \{1,\dots,d+2\}$ defines a face of $\Delta$ if and only if $S^- = S \cap \{1,\dots,r\}$ and $S^{+} =S \cap \{r+1,\dots,r+s\}$ are non-empty. Thus,
\begin{align*}
\sum_{F\precneq\Delta} \#Z_F(\F_q) &=  \sum_{ \substack{S \subsetneq \{1,\dots,d+2\}\\  2 \le|S| \le d+1 \\ |S^{+}|,|S^{-}|>0} } \#Z_{F_S}(\F_q).
\end{align*}
The number of faces $F$ of $\Delta$ for which $|S_F|=N$ is given by
\[
\sum_{\substack{m,n\ge1\\ m+n =N}} {\binom{r}{m}}{\binom{s}{n}}.
\]
Thus, we have
\begin{align}
\sum_{F\precneq\Delta} \#Z_F(\F_q) &= A_1 +A_2, \text{ where } \\
A_1 &\coloneqq\sum_{ |S|=2 }^{d+1} q^{-1} (q-1)^{d-|S|+1}\sum_{\substack{m,n\ge1\\ m+n =|S|}} {\binom{r}{m}}{\binom{s}{n}},\nonumber \\
A_2 &\coloneqq q^{-1}\sum_{\lambda \in \Lambda(q)}  \sum_{m=0}^{q-2} G_{m,\lambda }
\sum_{ \substack{S \subsetneq \{1,\dots,d+2\}\\  2 \le|S| \le d+1 \\ |S^{+}|,|S^{-}|>0} } 
\prod_{j\in S}\epsilon\left( -m \gamma_j + \delta(j,\lambda) \right) (-1)^{|S|} .
\end{align}
For fixed $m$ and $\lambda$, let $S_{m,\lambda}$  be the set of all $j$ for which  $-m \gamma_j +\delta(j,\lambda)=0 \pmod{q^\times}$. Then,
\begin{align}
A_2 &= q^{-1}\sum_{\lambda \in \Lambda(q)}  \sum_{m=0}^{q-2} G_{m,\lambda } \sum_{\substack{2\le |S|\le d+1, \\ S \subset S_{m,\lambda}, \\ |S^{+}|,|S^{-}|>0  }}(-1)^{|S|} \nonumber\\
&= q^{-1}\sum_{\lambda \in \Lambda(q)}  \sum_{m=0}^{q-2} G_{m,\lambda }\sum_{\substack{|S|=2}}^{\min(d+1,\,|S_{m,\lambda}^+|+|S_{m,\lambda}^-|)} (-1)^{|S|} \sum_{\substack{a,b\ge1\\ a+b =|S|}} { \binom{|S_{m,\lambda}^-| }{a}}{\binom{|S_{m,\lambda}^+|}{b}}.\nonumber
\end{align}
Note that if $S_{m,\lambda}^{+}$ or $S_{m,\lambda}^{-}$ is empty, then the sum over $|S|$ is $0$.  We define 
\begin{equation}\label{eq:definition-eta}
\eta_{\delta(\lambda)}(-m) \coloneqq \begin{cases}
    0 &  \text{ if } \min( |S_{m,\lambda}^{-}|,|S_{m,\lambda}^{+}|) = 0,\\
    1 & \text{ otherwise.}
\end{cases}
\end{equation}
Note that if $S_{m,\lambda}=\{1,\dots,d+2\}$, then  $(m,\lambda)=(0,0)$ (by the argument in the proof of Corollary \ref{cor:curves-hgm}); in that case the sum over $|S|$ stops at $|S|=d+1$, so by \eqref{eq:alternate} it equals $1-(-1)^{d}$ instead of $1$. Since $G_{0,0}=(-1)^{d+2}$, this accounts for the $-q^{-1}$ below.
Then, by Lemma \ref{lem:identities} below, we have 
\begin{align}
A_2 &=  q^{-1}\sum_{\lambda \in \Lambda(q)}  \sum_{m=0}^{q-2}G_{m,\lambda} \cdot \eta_{\delta(\lambda)}(-m) - q^{-1}.
\end{align}
On the other hand, by Lemma \ref{lem:identities} below, we have:
\begin{align*}
    A_1 &= \sum_{ |S|=2 }^{d+1} q^{-1} (q-1)^{d-|S|+1}\sum_{\substack{m,n\ge1\\ m+n =|S|}} {\binom{r}{m}}{\binom{s}{n}}\nonumber \\
    &= q^{-1}(q-1)^{-1} \big(q^{r}-(q-1)^{r}\big)\big(q^{s}-(q-1)^{s}\big) - q^{-1}(q-1)^{-1}.
\end{align*}
Putting everything together, we get:
\begin{align*}
    \#\overline{Z}(\F_q) &=  \#Z_\Delta(\F_q) + \sum_{F\precneq\Delta} \#Z_F(\F_q)\\
    &=q^{-1}(q-1)^{-1} \sum_{\lambda \in \Lambda(q)}  \sum_{m=0}^{q-2} G_{m,\lambda }+q^{-1}\sum_{\lambda \in \Lambda(q)}\sum_{m=0}^{q-2}G_{m,\lambda} \cdot \eta_{\delta(\lambda)}(-m) - q^{-1}\\
    &\qquad+q^{-1}(q-1)^{d} +q^{-1}(q-1)^{-1} \big(q^{r}-(q-1)^{r}\big)\big(q^{s}-(q-1)^{s}\big) - q^{-1}(q-1)^{-1}\\
    &=  \sum_{\lambda \in \Lambda(q)} \sum_{m=0}^{q-2}  q^{-1}(q-1)^{-1}\cdot \left( 1+  (q-1) \cdot\eta_{\delta(\lambda)}(-m) \right) G_{m,\lambda} \\
    &\qquad+(q-1)^{r+s-2}+ q^{r+s-1}(q-1)^{-1} -q^{r-1}(q-1)^{s-1}- q^{s-1}(q-1)^{r-1}\\ 
    &\qquad- q^{-1}(q-1)^{-1} - q^{-1}\\
    &=\sum_{\lambda \in \Lambda(q)} \sum_{m=0}^{q-2}(q-1)^{-1}  q^{\eta_{\delta(\lambda)}(-m)-1} G_{m,\lambda}\\
    &\qquad+\frac{q^{r+s-1}-1}{q-1}+ (q-1)^{r+s-2} -q^{r-1}(q-1)^{s-1}- q^{s-1}(q-1)^{r-1}.
\end{align*}
Thus, we have proven the following:
\begin{theo}\label{theo:point-count-p-delta}
    Let $\overline{Z}$ be the closure of the variety $Z$ in the toric variety $\Ps_{\Delta}$. Then,
\begin{align}
    \#\overline{Z}(\F_q)&= \frac{q^{r+s-1}-1}{q-1}+ (q-1)^{r+s-2} -q^{r-1}(q-1)^{s-1}- q^{s-1}(q-1)^{r-1}\nonumber\\
     &\qquad+\sum_{\lambda \in \Lambda(q)}  \frac{\chi^{-\lambda}(\sigma)}{q-1}\sum_{m=0}^{q-2}q^{\eta_{\delta(\lambda)}(-m)-1} \prod_{j=1}^{d+2}g\left(-m\gamma_j + \delta(j,\lambda)\right) \chi^{m}\left(t\right). \label{eq:point-count-p-delta}
\end{align}
Recall that $\eta$ was defined in Equation \eqref{eq:definition-eta}.
\end{theo}

\begin{rem}
The function $\eta$ defined above is easily computed as follows:
Recall that $\delta(j,\lambda)\coloneqq \sum_{i=1}^d \lambda_i \rho_{ij}$.
For a fixed $\lambda$, write 
\begin{equation}\label{eq:defi-eta}
 \frac{ \prod_{\gamma_j<0} \left(T^{-\ga_j}- \zeta_{q^\times}^{\delta(j,\lambda)}\right)}{ \prod_{\gamma_j>0} \left(T^{\ga_j}- \zeta_{q^\times}^{-\delta(j,\lambda)}\right)}.   
\end{equation}
If $e^{2\pi i (-m)/q^\times}$ is a common zero of the numerator and the denominator of Equation \eqref{eq:defi-eta}, then $\eta_{\delta(\lambda)}(-m)=1$. Otherwise, it is equal to $0$.
\end{rem}

\begin{lem}\label{lem:identities} We have the following identities. 
    \begin{equation}\label{eq:alternate}
        \sum_{|S|=2}^{r+s} (-1)^{|S|} \sum_{\substack{m,n\ge 1\\ m+n=|S|}} {\binom{r}{m}}{\binom{s}{n}} =1
   \end{equation}
   \begin{equation}\label{eq:normal}
       \sum_{|S|=2}^{r+s} (q-1)^{r+s-|S|}  \sum_{\substack{m,n\ge 1\\ m+n=|S|}} {\binom{r}{m}}{\binom{s}{n}} = \big(q^{r}-(q-1)^{r}\big)\big(q^{s}-(q-1)^{s}\big)
   \end{equation}  
\end{lem}
\begin{proof}
Let $P(x)=((1+x)^{r}-1)\cdot ((1+x)^{s}-1)$. Note that 
\begin{equation}\label{eq:lem-identity}
 P(x)=\sum_{k=2}^{r+s} \sum_{\substack{m,n\ge 1\\ m+n=k}} {\binom{r}{m}}{\binom{s}{n}}x^{k}.   
\end{equation}
Substituting $x=-1$ in Equation \eqref{eq:lem-identity} gives \eqref{eq:alternate}. 
Substituting $x=(q-1)^{-1}$ in Equation \eqref{eq:lem-identity} and multiplying by $(q-1)^{r+s}$ gives \eqref{eq:normal}.
\end{proof}

We have a few corollaries of Theorem \ref{theo:point-count-p-delta}.
\begin{corol}\label{cor:simplices-hg}
    Suppose that $r=1$ and $s={d+1}$.
    If $q$ is coprime to $\ga_1\cdots\ga_{d+2}$, then
  \begin{align*}
      \#\overline{Z}(\F_q)&= \frac{q^{d}-1}{q-1}-\sum_{\lambda \in \Lambda(q)} \chi^{-\lambda}(\sigma) q^{s_{\delta(\lambda)}(0)-1} g(\delta(\lambda))  F_q\left(\gamma,\delta(\lambda),q^\times \mid \frac{t}{\gamma^\gamma}\right).
  \end{align*}
\end{corol}
\begin{proof}
Since $r=1$, for each $m$, we have $s_{\delta(\lambda)}(-m)\le 1$. Thus, $\eta_{\delta(\lambda)}(-m)= s_{\delta(\lambda)}(-m)$. 
Recall that 
\[
 G_{m,\lambda}\coloneqq \prod_{j=1}^{d+2}g\left(-m\gamma_j + \delta(j,\lambda)\right) \chi^{m}(t) \chi^{-\lambda}(\sigma).
\]
We have
\begin{align*}
     \frac{1}{q-1}\sum_{m=0}^{q-2}q^{\eta_{\delta(\lambda)}(-m)-1} G_{m, \lambda} &= G_{0,\lambda} q^{s_{\delta(\lambda)}(0)-1} \frac{1}{q-1}\sum_{m=0}^{q-2}q^{s_{\delta(\lambda)}(-m)-s_{\delta(\lambda)}(0)} \frac{G_{m, \lambda}}{G_{0,\lambda}}\\
   &= - G_{0,\lambda} q^{s_{\delta(\lambda)}(0)-1} \frac{1}{1-q}\sum_{m=0}^{q-2}q^{s_{\delta(\lambda)}(-m)-s_{\delta(\lambda)}(0)} \frac{G_{m, \lambda}}{G_{0,\lambda}}\\
   &= - \chi^{-\lambda}(\sigma) q^{s_{\delta(\lambda)}(0)-1} g(\delta(\lambda)) F_q\left(\gamma,\delta(\lambda),q^\times \mid \frac{t}{\gamma^\gamma}\right).
\end{align*}
\end{proof}
\begin{rem}
In the case of Corollary \ref{cor:simplices-hg}, when $\deg(Z)$ is coprime to $q$, the toric variety $\Ps_{\Delta}$ is quasi-smooth; in fact, it is a quotient of a weighted projective space by a finite abelian group. If we assume that $\prod_{j=1}^{d+2}u_j^{\gamma_j} \neq \prod_{j=1}^{d+2} \gamma_j^{\gamma_j}$, then $\overline{Z}$ is also quasi-smooth (see Proposition \ref{prop:cover-smooth} and Lemma \ref{lem:quasi-smooth-compactification}). 
Furthermore, $\overline{Z}$ is an ample divisor since it is the zero locus of a section of the ample line bundle $\OO_{\Ps_\Delta}(\Delta)$ (see \cite{Danilov-geometry-of-toric}). 
By the Lefschetz hyperplane theorem, for $k=0,1, \dots, 2(d-1)$ and $k \neq d-1$,  we have
\[
H^{k}_{\text{\'et}} (\overline{Z}\otimes \overline{\F}_q, \Q_{\ell} ) =
\begin{cases}
 \Q_\ell(-k/2)  & \text{ if }  k \text{ is even,}\\ 
 0  & \text{ if }  k \text{ is odd.}
 \end{cases}
\]
Furthermore,  letting $H^{d-1}_{\text{\'et},\text{prim}}$ denote primitive cohomology, we have
\[
 H^{d-1}_{\text{\'et},\text{prim}}(\overline{Z}\otimes \overline{\F}_q, \Q_\ell) =
\begin{cases}
    H^{d-1}_{\text{\'et}}(\overline{Z}\otimes \overline{\F}_q, \Q_\ell)  &\text{ if } d-1 \text{ is odd,}\\
     H^{d-1}_{\text{\'et}}(\overline{Z}\otimes \overline{\F}_q, \Q_\ell)/\Q_{\ell}\left(-\frac{d-1}{2}\right) &\text{ if } d-1 \text{ is even.}
\end{cases}
\]
By the Grothendieck--Lefschetz trace formula \cite{SGA412}, we have
\[
(-1)^{d-1}\operatorname{Tr} ( \operatorname{Fr} \mid H^{d-1}_{\text{\'et},\text{prim}}(\overline{Z} \otimes \overline{\F}_q,\mathbb{Q}_\ell)) =  -\sum_{\lambda \in \Lambda(q)} \chi^{-\lambda}(\sigma) q^{s_{\delta(\lambda)}(0)-1} g(\delta(\lambda))  F_q\left(\gamma,\delta(\lambda),q^\times \mid \frac{t}{\gamma^\gamma}\right),
\]
where $\operatorname{Fr}$ is the induced morphism on cohomology of a geometric Frobenius element.
Thus, we view the finite hypergeometric sums on the right-hand side as (twisted) Frobenius traces on the primitive cohomology group $H^{d-1}_{\text{\'et},\text{prim}}(\overline{Z}\otimes\overline{\F}_q,\Q_\ell)$.
\end{rem}

\begin{rem}\label{rem:hecke-grossen}
The factors
\[
q^{s_{\delta(\lambda)}(0)-1} g(\delta(\lambda)) 
\]
which appear in Corollary \ref{cor:simplices-hg} are essentially harmless. Indeed, we have 
\[
\sum_{j=1}^{d+2} \delta(j,\lambda) = {\sum_{j=1}^{d+2}}\sum_{i=1}^d \lambda_i \rho_{ij} ={ \sum_{i=1}^d \lambda_i \sum_{j=1}^{d+2} \rho_{ij} }=0 \pmod{q^\times}.
\]
Therefore, the product $q^{s_{\delta(\lambda)}(0)-1} g(\delta(\lambda))$ can be identified, as a Hecke-Grossencharacter over $\Q(\zeta_n)$ for some $n$, when varying $q$ (see \cite[Theorem 4.3.4 and Remark 4.4.4]{voight-kelly}).

If the Newton polytope $\Delta$ of $f$ is primitive, then $\Lambda(q)=\{0\}$ and 
\[q^{s_{\delta(\lambda)}(0)-1} g(\delta(\lambda)) = (-1)^{r+s} q^{s_{0}(0)-1}.\]
If $\Delta$ is not primitive, then it is not hard to write down examples where $g(\delta(\lambda))$ is a non-trivial product of Gauss sums. Thus, to interpret the sum, one would need to appeal to Hecke-Grossencharacters.
\end{rem}

Note that in the case $r=1,s=3$, when $\deg(Z)$ is coprime to $q$, the corresponding toric variety $\Ps_{\Delta}$ is smooth away from the points corresponding to the vertices of the polytope. If $Z$ is smooth and $\Delta$-regular, then the curve $\overline{Z}$ is smooth, since it does not pass through the torus-fixed points of $\Ps_\Delta$. More generally, if $q$ is coprime to $\deg(Z)$ and to $\ga_1\cdots \ga_{d+2}$, we get a smooth compactification when $Z$ is a smooth curve.
Furthermore, we have the following:
\begin{corol} \label{cor:curves-hgm}
Suppose that $r=s=2$.
If $q$ is coprime to $\ga_1\cdots\ga_{d+2}$, then 
\begin{align*}
      \#\overline{Z}(\F_q)&= q+1 -\sum_{\lambda \in \Lambda(q)} \chi^{-\lambda}(\sigma) q^{s_{\delta(\lambda)}(0)-1} g(\delta(\lambda))  F_q\left(\gamma,\delta(\lambda),q^\times \mid \frac{t}{\gamma^\gamma}\right).
  \end{align*}
\end{corol}
\begin{proof}
Note that in this case, we have $s_{0}(0)= 2$. We show that $s_{\delta(\lambda)}(-m) = \eta_{\delta(\lambda)}(-m)$ unless $m=0$ and $\lambda=0$.
Since $s_{\delta(\lambda)}(-m) \le 2$ for all $\lambda$ and $m$, it suffices to show that $s_{\delta(\lambda)}(-m) \neq 2$ unless $m=0$ and $\lambda=(\lambda_1,\lambda_2)=0$. 
Suppose that
\begin{equation}\label{eq:proof-corol-curves}
-m \gamma_j +\sum_{i=1}^2 \lambda_i\rho_{ij}= 0 \pmod{q^\times} \text{ for } j=1,\dots,4.    
\end{equation}
Recall that  
\[
\sum_{j=1}^{4} \rho_{ij}f_{kj} = \epsilon_{k}(i)\coloneqq \begin{cases}
    1 & \text{ if } i=k,\\
    0 & \text{ if } i \neq k,
\end{cases}
\]
where $f_{kj}$ are as in \eqref{eq:primitive-coefficients}. Thus, {for} $k=1,2$, we have  
\begin{align*}
 \lambda_k = \sum_{i=1}^2 \lambda_i \epsilon_{k}(i)  &=\sum_{j=1}^{4} -m \gamma_j f_{kj} +\sum_{i=1}^2 \lambda_i\sum_{j=1}^{4}  \rho_{ij} f_{kj} \\
 &= \sum_{j=1}^{4} \left( -m \ga_j + \sum_{i=1}^{2}\lambda_i\rho_{ij} \right) f_{kj} = 0 \pmod{q^\times}.
\end{align*}
By Equation \eqref{eq:proof-corol-curves}, we have $-m \gamma_j =0 \pmod{q^\times}$ for $j=1, \dots,4$. Thus, $m=0 \pmod{q^\times}$ since $\gcd(\gamma_1,\dots,\gamma_{4})=1$.

Now, note that 
\[
  \frac{q^{r+s-1}-1}{q-1}+ (q-1)^{r+s-2} -q^{r-1}(q-1)^{s-1}- q^{s-1}(q-1)^{r-1}  =q+2.
\]
Thus, by Theorem \ref{theo:point-count-p-delta}, we have
\begin{align*}
 \#\overline{Z}(\F_q) &= q+2 + \sum_{\lambda \in \Lambda(q)}  \frac{1}{q-1}\sum_{m=0}^{q-2}q^{\eta_{\delta(\lambda)}(-m)-1} G_{m, \lambda}\\
 &= q+1 + 1 + \frac{q^{\eta_{0}(0)-1}-q^{s_{0}(0)-1}}{q-1}G_{0,0} +\sum_{\lambda \in \Lambda(q)}  \frac{1}{q-1}\sum_{m=0}^{q-2}q^{s_{\delta(\lambda)}(-m)-1} G_{m, \lambda}\\
 &=q+1 +1 +\frac{q^0-q^1}{q-1} -\sum_{\lambda \in \Lambda(q)} \chi^{-\lambda}(\sigma) q^{s_{\delta(\lambda)}(0)-1} g(\delta(\lambda))  F_q\left(\gamma,\delta(\lambda), q^\times \mid \frac{t}{\gamma^\gamma}\right)\\
 &=q+1 -\sum_{\lambda \in \Lambda(q)} \chi^{-\lambda}(\sigma) q^{s_{\delta(\lambda)}(0)-1} g(\delta(\lambda))  F_q\left(\gamma,\delta(\lambda),q^\times \mid \frac{t}{\gamma^\gamma}\right).
\end{align*}
\end{proof}
Next, we would like to consider surfaces, that is, $r+s=5$. We have already considered the case $r=1,s=4$ in Corollary \ref{cor:simplices-hg}, so we consider the case $r=2,s=3$. The remaining cases are handled by symmetry in $r$ and $s$.

\begin{corol}\label{cor:surfaces-hgm}
Suppose that $r=2, s=3$.  If $q$ is coprime to $\ga_1\cdots\ga_{d+2}$, then
\begin{align*}
      \#\overline{Z}(\F_q)&= q^2+3q +1  -\sum_{\lambda \in \Lambda(q)} \chi^{-\lambda}(\sigma) q^{s_{\delta(\lambda)}(0)-1} g(\delta(\lambda))  F_q\left(\gamma,\delta(\lambda), q^\times \mid \frac{t}{\gamma^\gamma}\right).
  \end{align*}
\end{corol}

\begin{proof}
Again, we would like to prove that $s_{\delta(\lambda)}(-m) = \eta_{\delta(\lambda)}(-m)$ unless $m=0$ and $\lambda=0$. 
Note that if $s_{\delta(\lambda)}(-m)=2$, then it must be that $-m \ga_j + \delta(j,\lambda)=0 \pmod{q^\times}$ for all $j$. Indeed, we have
\[
\sum_{j=1}^5 \big( -m \ga_j + \delta(j,\lambda) \big) =  0 \pmod{q^\times}
\]
(see Remark \ref{rem:hecke-grossen}),
and $s_{\delta(\lambda)}(-m)=2$  implies that four of the terms in the sum are $0$; thus the last term is also $0$.
The argument in the proof of Corollary \ref{cor:curves-hgm} then shows that $m=0$ and $\lambda=0$.
Simple algebraic manipulation, similar to that in the proof of Corollary \ref{cor:curves-hgm}, gives the result. 
\end{proof}

\begin{rem}
    When $r=2$ and $s=3$, {if} $\prod_{j=1}^{d+2} u_j^{\gamma_j} \neq \prod \gamma_j^{\gamma_j}$ and $q$ is coprime to $\deg(Z)$ and to $\ga_1\cdots\ga_{d+2}$,
    then the closure $\overline{Z}$ is, in fact, quasi-smooth. This can be seen as follows:
    The normal fan of the Newton polytope $\Delta$ has $3$ maximal cones that are not simplicial, and the refinement $\Sigma'$ (Definition \ref{defi:fan-refinement}) only subdivides these three cones. Hence the toric morphism $\Ps_{\Sigma'} \to \Ps_{\Delta}$ is an isomorphism away from the three corresponding torus-fixed points. By $\Delta$-regularity (see Proposition \ref{prop:cover-smooth}), $\overline{Z}$ does not meet these points, so $\overline{Z}$ is isomorphic to the closure $W$ of $Z$ in $\Ps_{\Sigma'}$. Since $\Sigma'$ is simplicial (Proposition \ref{prop:sigma-is-a-fan}), $W$ is quasi-smooth by Proposition \ref{prop:cover-smooth} Lemma \ref{lem:quasi-smooth-compactification}. This proves the claim.
\end{rem}

\subsection{Compactification II}\label{subsection:compactification-2}
We consider the toric variety $\Ps_{\Sigma'}$ (see \cite{Denef-Loeser}) associated to the fan $\Sigma'$ (Definition \ref{defi:fan-refinement}).
It is again a compactification of $\G_m^d$, and admits a decomposition by the cones of $\Sigma'$: every $l$-dimensional cone $\sigma$ defines a subvariety $\T_\sigma$ of $\Ps_{\Sigma'}$ which is a torus of dimension $d-l$. Furthermore, there is a decomposition of $\Ps_{\Sigma'}$ into disjoint subvarieties
\[
\Ps_{\Sigma'} = \bigcup_{\sigma} \T_{\sigma}.
\]
Let $W$ be the {closure} of $Z$ in $\Ps_{\Sigma'}$, and define $W_\sigma\coloneqq \T_\sigma \cap W $. Then, we have a decomposition of $W$ into the union of disjoint subvarieties:
\[
W = \bigcup_{ \sigma \text{ cone of }  \Sigma' } W_\sigma.  
\]
Furthermore, if $F_\sigma$ is the first meet locus of the cone $\sigma$
and $f_{|F_\sigma}$ is the restriction of $f$ to the face $F_{\sigma}$, {then} the variety $V(f_{|F_\sigma})\subset \G_m^d$ is isomorphic to the product $ \G_m^{l} \times W_\sigma$ over $\F_q$.
Thus, the number of $\F_q$-points of $W_{\sigma}$ is given by
\[
\#W_{\sigma}(\F_q) = \#V(f_{|F_{\sigma}})(\F_q) (q-1)^{-l}.
\]
Let $\sigma_C$ be the cone of $\Sigma'$ defined by the sequence $C= ((i_1,j_1),\dots ,(i_l,j_l))$. Recall that $\sigma_C$ is the {cone generated by} the rays $(i_1,j_1),\dots,(i_l,j_l)$, the cone $\sigma_C$ is of dimension $l$, and the first meet locus $F_{C}$ of $\sigma_C$ is the face of $\Delta$ corresponding to the set $S_C = \{i_1,\dots,i_l,j_1,\dots,j_l\}$.
The restriction of $f$ to the face $F_C$ is given by 
\[
f_{|C} \coloneqq \sum_{j \notin S_C} u_j \prod_{i=1}^d x_i^{m_{ij}}.
\]
The hypersurface $V(f_{|C})$ is isomorphic, by Corollary \ref{cor:cover}, to the variety given by 
\begin{align*}
   \sum_{j\notin S_C} w_j&= 0, \quad  w_1\cdots w_{d+2} \neq 0,\\
    \prod_{j=1}^{d+2} w_j^{\gamma_j} &=t, \quad t = \prod_{j=1}^{d+2} u_j^{\gamma_j},\\
    \prod_{j=1}^{d+2} w_j^{\rho_{kj}} &= \sigma_k \prod_{i=1}^d z_i^{\mathcal{N}_{ki}},  \qquad k=1,\dots,d.
\end{align*}
The number of $\F_q$-points  of $W_{C}\coloneqq W_{\sigma_C}$ {satisfies}
\begingroup\interdisplaylinepenalty=10000
\begin{align}
    q  \#W_{C} (\mathbb{F}_q) &= (q-1)^{d-l_C} \nonumber \\
    &+ (-1)^{|S_C|}(q-1)^{|S_C|-l_C-1}\cdot\sum_{\lambda \in \Lambda(q)}  \sum_{m=0}^{q-2}\prod_{j\in S_C} \epsilon\left( -m \gamma_j +\delta(j,\lambda) \right)G_{m,\lambda}, \label{eq:contribution-of-cone}
\end{align}
\endgroup
where, as before, $ G_{m,\lambda}\coloneqq \prod_{j=1}^{d+2}g\left(-m\gamma_j + \delta(j,\lambda)\right) \chi^{m}(t) \chi^{-\lambda}(\sigma)$.
Since the varieties $W_C$ are disjoint, we have
\begin{equation}
    \#W(\F_q) = \sum_{ \substack{C \\ |S_C| \le d+1}}  \#W_C(\F_q).
\end{equation}
Here the sum is taken over all sequences  $C= ((i_1,j_1),\dots ,(i_l,j_l))$ with $|S_C| \le d+1$, including the {empty} sequence.
To shorten the formulas, we set
\[
A_{m,\lambda,C}\coloneqq (-1)^{|S_C|}(q-1)^{|S_C|-l_C-1}
  \prod_{j\in S_C}\epsilon\left( -m \gamma_j +\delta(j,\lambda) \right).
\]
Then,
\[
q\#W_C(\F_q) = (q-1)^{d-l_C} + \sum_{m=0}^{q-2} \sum_{\lambda \in \Lambda(q)} A_{m,\lambda,C} G_{m,\lambda}.
\]
Thus,
\begin{align*}
    q \# W(\F_q) &= A_1 + A_2, \\
    A_1 &\coloneqq \sum_{ \substack{C \\ |S_C| \le d+1}} (q-1)^{d-l_C},  \\
    A_2 &\coloneqq \sum_{ \substack{C \\ |S_C| \le d+1}}\sum_{m=0}^{q-2} \sum_{\lambda \in \Lambda(q)} A_{m,\lambda,C} G_{m,\lambda}. 
\end{align*}
By \cite[Propositions 5.6 and 5.7]{BCM}, we have that 
\begin{align}
    A_1 &=  (q-1)^{-1} \sum_{k=0}^{\min(r,s)-1} {\binom{r-1}{k}}{\binom{s-1}{k}} \left( q^{d+1-k} - q^{k} \right).
\end{align}
To compute $A_2$, 
we fix $m$ and $\lambda$ and consider the sum $\sum_C A_{m,\lambda,C}$.
Let $S_{m,\lambda}$  be the set of all $j$ for which  $-m \gamma_j +\delta(j, \lambda)=0 \pmod{q^\times}$,    and  $r_{m,\lambda} = \min(|S_{m,\lambda}^+|, |S_{m,\lambda}^-|)$. Then, $r_{m,\lambda} =s_{\delta(\lambda)} (-m)$ (Definition \ref{defi:s-delta}), since
\[ -m \gamma_j +\delta(j, \lambda) =0 \pmod{q^\times} \iff  e^{\frac{2 \pi i (-m)}{q^\times }} \text{ is a root of } T^{-\gamma_j}-  \zeta_{q^\times}^{\delta(j, \lambda) }.\]
By \cite[Proposition 5.5]{BCM}, we have
\begin{align*}
\sum_{ \substack{C \\ |S_C| \le d+2}} A_{m,\lambda,C}& = \sum_{ \substack{C \\ S_C \subset S_{m,\lambda}}} (-1)^{|S_C|}(q-1)^{|S_C|-l_C-1} \\
&=(q-1)^{-1} q^{s_{\delta(\lambda)} (-m)}.
\end{align*}
This sum includes maximal cells, which are not part of the fan $\Sigma'$; we subtract their contribution. 
By the argument in the proof of Corollary \ref{cor:curves-hgm}, we have that 
$\prod_{j=1}^{d+2} \epsilon \left( -m \gamma_j +\delta(j,\lambda)\right) =1$ if and only if $ m=0$ and $\lambda=0$. Thus,
\begin{align*}
  \sum_{ \substack{C \\ |S_C| =d+2}} A_{m,\lambda,C}&=    \sum_{ \substack{C \\ |S_C| =d+2}} (-1)^{d+2}(q-1)^{d+1-l_C}
  \prod_{j=1}^{d+2} \epsilon \left( -m \gamma_j +\delta(j,\lambda)\right) \\
  &=   \sum_{ \substack{C \\ |S_C| = d+2}} (-1)^{d+2}(q-1)^{d+1-l_C}  \epsilon(m)\epsilon(\lambda).
\end{align*}
Thus,
\begin{align*}
    A_2 &=  \sum_{ \substack{C \\ |S_C| \le d+2}}\sum_{m=0}^{q-2} \sum_{\lambda \in \Lambda(q)} A_{m,\lambda,C} G_{m,\lambda} -\sum_{ \substack{C \\ |S_C| = d+2}}\sum_{m=0}^{q-2} \sum_{\lambda \in \Lambda(q)} A_{m,\lambda,C} G_{m,\lambda} \\
    &= - \sum_{\lambda \in \Lambda(q)} \sum_{m=0}^{q-2} \sum_{ \substack{C \\ |S_C| = d+2}}  \epsilon(m)\epsilon(\lambda)  A_{m,\lambda,C}    G_{m,\lambda} + \sum_{\lambda \in \Lambda(q)}(q-1)^{-1} \sum_{m=0}^{q-2}  q^{s_{\delta(\lambda)} (-m)}G_{m,\lambda} \\
     &= -   \sum_{ \substack{C \\ |S_C| = d+2}}  A_{0,0,C}   G_{0,0}  +\sum_{\lambda \in \Lambda(q)}(q-1)^{-1}\sum_{m=0}^{q-2}  q^{s_{\delta(\lambda)} (-m)}G_{m,\lambda} \\
    &= - \sum_{ \substack{C \\ |S_C| = d+2}} (-1)^{d+2}(q-1)^{d+1-l_C}G_{0,0} + \sum_{\lambda \in \Lambda(q)} (q-1)^{-1}\sum_{m=0}^{q-2}  q^{s_{\delta(\lambda)} (-m)}G_{m,\lambda} \\
    &= B_2 G_{0,0} + \sum_{\lambda \in \Lambda(q)} q^{s_{\delta(\lambda)} (0)}G_{0,\lambda} \ (q-1)^{-1} \sum_{m=0}^{q-2}  q^{s_{\delta(\lambda)} (-m) -s_{\delta(\lambda)} (0)}\frac{G_{m,\lambda}}{G_{0,\lambda}}, \text{ where }\\
    B_2 &= (-1)^{d-1}\sum_{ \substack{C \\ |S_C| = d+2}} (q-1)^{d+1-l_C}.
\end{align*}
Finally, by \cite[Proposition 5.7]{BCM},
\begin{align*}
   A_1+ B_2G_{0,0} &= A_1 +B_2 (-1)^{d+2}\\
   &=A_1- \sum_{k=0}^{\min(r,s)-1}  {\binom{r-1}{k}}{\binom{s-1}{k}}q^k \\
    &=(q-1)^{-1} \sum_{k=0}^{\min(r,s)-1} {\binom{r-1}{k}}{\binom{s-1}{k}} \left( q^{d+1-k} - q^{k} -(q-1) q^k \right) \\
    &=q \sum_{k=0}^{\min(r,s)-1} {\binom{r-1}{k}}{\binom{s-1}{k}} \frac{q^{d-k} - q^{k} }{q-1}.
\end{align*}
Altogether, we have proven the following:
\begin{theo}\label{theo:main-desingularization}
Let $W$ be the closure of $Z$ in $\Ps_{\Sigma'}$. 
If $q$ is coprime to $\ga_1\cdots\ga_{d+2}$, then 
\begin{align*}
\#W (\mathbb{F}_q) &= \sum_{k=0}^{\min(r,s)-1} {\binom{r-1}{k}}{\binom{s-1}{k}} \frac{q^{d-k} - q^{k} }{q-1}\\
&-\sum_{\lambda \in \Lambda(q)}\chi^{-\lambda}(\sigma) q^{s_{\delta(\lambda)} (0)-1} g\left(\delta(\lambda)\right) F_q\left(\gamma,\delta(\lambda), q^\times \mid \frac{t}{\gamma^\gamma}\right).
\end{align*}
\end{theo}

\begin{rem}
    If $Z$ is primitive, then we recover the main result of \cite[Theorem 1.5]{BCM}.
    Indeed, if $Z$ is primitive, then $\Lambda(q)= \{0\}$, $s_{0}(0) =\min(r,s)$ and $g\left(\delta(0)\right) =(-1)^{r+s}$. Thus,
    \begin{align*}
    \#W (\mathbb{F}_q) &= \sum_{k=0}^{\min(r,s)-1} {\binom{r-1}{k}}{\binom{s-1}{k}} \frac{q^{d-k} - q^{k} }{q-1}\\
    &+(-1)^{r+s-1} q^{\min(r,s) -1}  F_q\left(\gamma,0, q^\times \bigm\vert  \frac{t}{\gamma^\gamma}\right).
    \end{align*}
\end{rem}
Unlike the compactification $\overline{Z}$, the compactification $W$ is always quasi-smooth whenever $Z$ is, provided that $q$ is coprime to $\deg(Z)$ and to $\ga_1\cdots \ga_{d+2}$.  
\begin{prop}\label{prop:bcm-quasi-smooth}
Assume that $q$ is coprime to $\deg(Z)$ and to $\ga_1\cdots\ga_{d+2}$, and that $\prod_{j=1}^{d+2} u_{j}^{\ga_j} \neq \prod_{j=1}^{d+2} \ga_j^{\ga_j}$. Then, $W$ is quasi-smooth. 
\end{prop}
\begin{proof}
By Proposition \ref{prop:cover-smooth}, the hypersurface $Z$ is smooth and $\Delta$-regular. Since $\Sigma'$ is simplicial, the toric variety $\Ps_{\Sigma'}$ is quasi-smooth. By Lemma \ref{lem:quasi-smooth-compactification}, the compactification $W$ is quasi-smooth.
\end{proof}

\section{Reverse engineering}\label{section:reverse-engineering}
Theorem \ref{theo:main-desingularization} gives a compactification $W$ of $Z$ whose point count is given by sums of finite hypergeometric sums (plus a polynomial).  Note that finite hypergeometric sums $F_q(\ga,\delta,N)$ arising in such varieties satisfy $\sum_{j} \delta_j =0 \pmod{N}$.
By Proposition \ref{prop:characterization-of-gamma-triple}, this implies that the associated hypergeometric parameters $(\al \,;\be)$ satisfy $\sum_{i} (\al_i - \be_i) \in \frac{1}{2}\Z$.    

In this section, we want to show that the converse also holds: 
namely, for any hypergeometric parameters $(\al \, ; \be)$ such that $\sum_{i=1}^{n} (\al_i - \be_i) \in \frac{1}{2}\mathbb{Z}$, the finite hypergeometric sum $F_q(\al \, ; \be \mid t )$ appears in the point count of a toric hypersurface of the form above.

Let $(\ga, \delta,N) \in \Z^{d+2}_{\neq0}\times \Z^{d+2} \times \Z_{\ge1}$ be a gamma triple with $\sum_{j=1}^{d+2} \delta_j =0$ and suppose that the greatest common divisor of the maximal minors of the matrix
\[
G= \begin{bmatrix}
\ga_1 &\ga_2 & \dots &\ga_{d+2} \\
\delta_1& \delta_2 &  \dots & \delta_{d+2} 
\end{bmatrix}
\]
is $1.$ 
Consider the variety
\footnote{The idea of associating these cyclic covers to the complex hypergeometric local systems arising from "partial twists" is due to Giulia Gugiatti. The paper \cite{partial-twists} will address the geometric realisation of these local systems in such covers.} 
$V_t$ in $\Ps^{d+1} \times \G_m$ defined by
\begin{align*}
    w_1 + \dots +w_{d+2} &= 0, \qquad w_1\dots w_{d+2} \neq 0,\\
    w_1^{\ga_1}\dots w_{d+2}^{\ga_{d+2}} &= t,\\
    w_1^{\delta_1} \dots w_{d+2}^{\delta_{d+2}} &= z^{N},
\end{align*}
where $w_1, \dots, w_{d+2}$ are {homogeneous} coordinates on $\Ps^{d+1}$ and $z$ is a coordinate on $\G_m$. 
Let $r$, respectively $s$, be the number of negative, respectively positive, components of $\ga$.
\begin{theo}\label{theo:realisation}
Let $q$ be a prime power that is coprime to $\ga_1 \cdots\ga_{d+2}$ and such that $N$ divides $q^\times$.  Then, there exists a compactification $\overline{V_t}$ of $V_t$ such that
   \begin{align*}
    \#\overline{V_t}(\F_q) &=  \sum_{k=0}^{\min(r,s)-1} {\binom{r-1}{k}}{\binom{s-1}{k}} \frac{q^{d-k} - q^{k} }{q-1} \\
    &-\sum_{j=0}^{N-1} q^{s_{j\delta}(0)-1} g\left(j\frac{q^\times}{N} \delta \right) F_q\left(\gamma, j \delta, N \mid \frac{t}{\gamma^\gamma}\right).
    \end{align*}
\end{theo}
\begin{proof}
To prove that such a compactification exists, we prove that $V_t$ is isomorphic to a hypersurface in $\G_m^{d}$ given by the vanishing of a linear combination of $d+2$ monomials and apply Theorem \ref{theo:main-desingularization}.

We claim that the kernel of the map $\Z^{d+2} \to \Z$  given by $x \mapsto \sum_{j=1}^{d+2} \ga_j x_j$ {admits} a $\Z$-basis  $(1,\dots,1), f_1 ,\dots, f_{d}$ such that  
\begin{align}
    \sum_{j=1}^{d+2} f_{kj} \delta_j = \begin{cases}
        1 & \text{ if } k=1, \\
        0 & \text{ if } k=2, \dots, d.
    \end{cases}
\end{align}
Consider the map $m_G :\Z^{d+2} \to \Z^2$ given by $x \mapsto G  x$. 
By hypothesis, the map is surjective. Thus, there exists a splitting
\[
\Z^{d+2} = \operatorname{Kernel} (m_G) \oplus f_1 \Z \oplus \kappa \Z,
\]
where $f_1 $ and $\kappa$ satisfy
\[
G f_1 = \begin{pmatrix}
        0 \\
        1
\end{pmatrix}, \quad G \kappa = \begin{pmatrix}
        1 \\
        0
\end{pmatrix}.
\]
Let $\{(1, \dots ,1), f_2, \dots,f_d \}$ be a $\Z$-basis of $\operatorname{Kernel}(m_G)$. Then, $(1,\dots,1), f_1, \dots, f_d$ is a basis of the kernel of $x \mapsto \sum_{j=1}^{d+2} \ga_j x_j$ with the desired properties. 

We make the change of variables (see Lemma \ref{lem:projective-to-toric})
$w_j = t^{\kappa_j} \prod_{i=1}^{d} x_i^{f_{ij}}$, $j=1,\dots, d+2$.
Note that 
\begin{align*}
     \prod_{j=1}^{d+2} \left( t^{\kappa_j} \prod_{i=1}^{d} x_i^{f_{ij}} \right)^{\ga_j}=t^{\sum_{j=1}^{d+2} \kappa_j \ga_j }   \prod_{i=1}^{d} x_{i}^{\sum_{j=1}^{d+2}f_{ij} \ga_j} &= t, \text { and }\\
    \prod_{j=1}^{d+2} \left( t^{\kappa_j} \prod_{i=1}^{d} x_i^{f_{ij}} \right)^{\delta_j} =  t^{\sum_{j=1}^{d+2} \kappa_j \delta_j }   \prod_{i=1}^{d} x_{i}^{\sum_{j=1}^{d+2}f_{ij} \delta_j} &= x_1.
\end{align*}
Thus, $V_t$ is isomorphic to the subvariety of $\G_m^{d} \times \G_m$ given by
\begin{align*}
 \sum_{j=1}^{d+2} t^{\kappa_j} \prod_{i=1}^{d} x_i^{f_{ij}}&=0,\\   
 x_1 & = z^{N}.
\end{align*}
Substituting $x_1$ with $z^{N}$ gives a hypersurface $Z_t$ in the torus $\G_m^{d}$ defined by the vanishing of a linear combination of $d+2$ monomials. 
We compute the fundamental data for $Z_t$: 
\begin{enumerate}
    \item The Gale dual of $Z_t$ is $\ga.$
    \item The covering matrix $\mathcal{N}$ is the diagonal matrix $\operatorname{diag}(N,1,\dots,1)$.
    \item  The set of solutions $\Lambda(q)$ is given by 
    \[
\Lambda(q) = \left\{ \, \left(\frac{j}{N} q^\times,0, \dots, 0\right) \in \left(\Z/q^\times \Z\right)^{d} \mid j=0, \dots, N-1 \, \right \}.\]
    \item We may take $\rho_{1} =\delta$, since $\sum_{j}\delta_j=0$, $\sum_{j} f_{1j}\delta_j=1$ and $\sum_{j} f_{kj}\delta_j=0$ for $k=2,\dots,d$. Furthermore, for each $\lambda \in \Lambda(q)$, $\delta(\lambda)$ is independent of $\rho_2, \dots,\rho_d$.
    \item The coefficient $\sigma_1 =1$, and the point count formulas do not depend on  $\sigma_2, \dots, \sigma_{d}$.
\end{enumerate}
The theorem follows by substituting this data in Theorem \ref{theo:main-desingularization}.
\end{proof}

\begin{prop}\label{prop:maximal-minors}
Let  $(\al \,; \be)$ be hypergeometric parameters that are not defined over $\Q$. There exists a gamma triple $(\ga,\delta,N)$ such that the greatest common divisor of the maximal minors of
\[
G= \begin{bmatrix}
\ga_1 &\ga_2 & \dots &\ga_{d+2} \\
\delta_1& \delta_2 &  \dots & \delta_{d+2} 
\end{bmatrix}
\]
is equal to $1$, and $(\ga,k\delta,N)$ represents $(\al \, ; \be)$ for some $1 \le k \le N-1$.
If moreover $\sum_{i=1}^n (\al_i -\be_i) \in \frac{1}{2}\mathbb{Z}$, then $\delta$ can be chosen such that $\sum_{j=1}^{d+2} \delta_j =0$.
\end{prop}
\begin{proof}
Let $(\al\,;\be)$ be defined over $\Q(\zeta_N)$. By Proposition~\ref{prop:extension-field-field-of-definition}, it is represented by a gamma triple $(\ga,\delta,N)$. Since the represented parameters depend only on $\delta$ modulo $N$, we may assume that $|\delta_j|\le N-1$ for all $j$.

If $\sum_i(\al_i-\be_i)\in\frac12\Z$, we first arrange that $\sum_j\delta_j=0$. By the computation in the proof of Proposition~\ref{prop:characterization-of-gamma-triple}, $\sum_i(\al_i-\be_i)\equiv\frac1N\sum_j\delta_j+\frac{r-s}2\pmod\Z$, so $\sum_j\delta_j\equiv0$ or $\frac N2\pmod N$. In the latter case, we append $(-2,1,1)$ to $\ga$ and $(0,0,\frac N2)$ to $\delta$, which does not change $(\al\,;\be)$. Now $\sum_j\delta_j\equiv0\pmod N$, and taking $0\le\delta_j\le N-1$ and subtracting $N$ from $\frac1N\sum_j\delta_j$ of the non-zero $\delta_j$ gives $\sum_j\delta_j=0$ and $|\delta_j|\le N-1$.

Since $(\al\,;\be)$ is not defined over $\Q$, we have $\delta\not\equiv0\pmod N$, so $g=\gcd(\delta_1,\dots,\delta_{d+2})$ satisfies $1\le g\le N-1$. By taking $k=g$ and replacing $(\ga,\delta,N)$ with $(\ga,\frac1g\delta,N)$, we may assume that $\gcd(\delta_1,\dots,\delta_{d+2})=1$ and that $(\ga,k\delta,N)$ represents $(\al\,;\be)$. If the greatest common divisor of the maximal minors of $G$ is not equal to $1$, we consider $\tilde\ga=(\ga_1,\dots,\ga_{d+2},1,-1)$ and $\tilde\delta=(\delta_1,\dots,\delta_{d+2},0,0)$. The greatest common divisor of the maximal minors of
\[
\tilde{G}= \begin{bmatrix}
\ga_1 &\ga_2 & \dots &\ga_{d+2} & 1 & -1 \\
\delta_1& \delta_2 &  \dots & \delta_{d+2} & 0 &0
\end{bmatrix}
\]
is a common divisor of $\delta_1,\dots,\delta_{d+2}$; hence it is $1$. Moreover, $(\tilde\ga,k\tilde\delta,N)$ also represents $(\al\,;\be)$. Both steps do not change $\sum_j\delta_j$, so we are done.
\end{proof}

Let $(\al \, ; \be)$ be hypergeometric parameters such that $\sum_{i=1}^n (\al_i -\be_i) \in \frac{1}{2}\mathbb{Z}$, and let $N$ be the smallest positive integer such that $(\al \, ; \be)$ is defined over $\Q(\zeta_N)$. 
If $N=1$, then Theorem \ref{theo:BCM} gives a variety $\overline{V_t}$ such that for all but finitely many primes $p$, the finite hypergeometric sum $F_{p^m}(\al \, ; \be \mid t/\ga^\ga)$ appears as a summand in the point count $\#\overline{V_t}(\F_{p^m})$ for all $m \ge 1$.   
If $N>1$, then $(\al \, ; \be)$ is not defined over $\Q$, so by Proposition \ref{prop:maximal-minors} there exists a gamma triple $(\ga,\delta,N)$ with $\sum_{j=1}^{d+2} \delta_j =0$ such that the greatest common divisor of the maximal minors of $G$ is $1$ and $(\ga,k\delta,N)$ represents $(\al \, ; \be)$ for some $1 \le k \le N-1$. Applying Theorem \ref{theo:realisation} to $(\ga,\delta,N)$ gives a variety $\overline{V_t}$ such that for all but finitely many primes $p$ of the form $lN+1$, for some integer $l$, the finite hypergeometric sum $F_{p^m}(\al \, ; \be \mid t/\ga^\ga)$ appears in the point count $\#\overline{V_t}(\F_{p^m})$ for all $m \ge 1$, namely in the summand with $j=k$.
The restriction on the prime powers when $N>1$ is not surprising, since one expects these finite hypergeometric sums to come from motives over $\Q(\zeta_N)$; the norm $\operatorname{N}(\mathfrak{p})$ of any prime $\mathfrak{p} \nmid N$ in $\Z[\zeta_N]$ satisfies $\operatorname{N}(\mathfrak{p}) \equiv 1 \pmod{N}$.

In the rest of this section, we present two examples to illustrate how the notion of gamma triples allows us to realise the finite hypergeometric sum in different varieties.  We consider two different gamma triples that represent the hypergeometric parameters $((1/3,2/3)\,;(1,1))$ (see Example \ref{exa:example-2}). 

\begin{exa}
Consider the gamma triple
\[
(\ga,\delta,N)= ((-3,1,1,1),(0,0,0,0),1).
\]
An associated family of varieties $Z_t$ in $\Ps^{3}$ is given by
\begin{align*}
x_1x_2x_3x_4 &\neq 0,\\
x_1 + x_2 + x_3 +x_4 &= 0,\\
x_1^{-3} x_2 x_3 x_4 &= t.
\end{align*}
The point count of the compactification $\overline{Z}_t$ of $Z_t$ is given by 
\[
\# \overline{Z}_t(\F_q) = q+1 -  F_q\left((-3,1,1,1),0,1 \, \bigg| \,  -27t\right)
\]
when $q$ is coprime to $3$ (see also \cite[Corollary 1.7]{BCM}).   
\end{exa}

\begin{exa}
Consider the gamma triple
\[
(\ga,\delta,N)= ((-1,-1,1,1),(1,-1,0,0),3).
\]
An associated family of varieties $X_t$ is given by 
\begin{align*}
x_1 x_2 x_3 x_4 &\neq 0,\\
x_1 +x_2 +x_3 +x_4 &= 0,\\
x_1^{-1}x_2^{-1}x_3 x_4&=t,\\
x_1 x_2^{-1} &= z^{3}
\end{align*}
in $\Ps^{3} \times \G_m$.
Let $q$ be a prime power such that $q\equiv 1 \pmod{3}$. Then, $q\equiv 1 \pmod{6}$ or $q=4^l$ for some $l \ge 1$, in which case $-1=1$ in $\F_q$. By Theorem \ref{theo:realisation}, the point count of the compactification $\overline{X}_t$ of $X_t$ is given by 
    \begin{align*}
    \#\overline{X}_t(\F_q) &= q+1 -\sum_{j=0}^{2} q^{s_{j\delta}(0)-1} g\left(j\frac{q^\times}{3} \delta \right) F_q\left(\gamma, j \delta, 3 \mid \frac{t}{\gamma^\gamma}\right)\\
    &= q+1 -  2\chi(-1)^{q^{\times}/3} F_q\left(\ga,\delta,3 \mid t\right)+ \begin{cases}
    q & \text{ if } t=1,\\
    0 & \text{ otherwise}.
    \end{cases}\\
    &= q+1 -  2 F_q\left(\ga,\delta,3 \mid t\right)+ \begin{cases}
    q & \text{ if } t=1,\\
    0 & \text{ otherwise}.
    \end{cases}
    \end{align*}
We note that the family $\overline{X}_t$ is a family of genus $2$ curves, and the special fibre at $t=1$ is reducible. It is not hard to see that $\overline{X}_1$ is isomorphic to the curve defined by the bihomogeneous equation
\[
(y_0+y_1) (y_1x_0^3+y_0x_1^3) =0 
\]
in $\Ps^1\times \Ps^1$.
It is the union of two projective lines that intersect at the three $\F_q$-points 
\[
([1,1],[1,-1]), ([\zeta,1],[1,-1]),([\zeta^2,1],[1,-1]),
\]
where $\zeta \in \F_q$ is a primitive third root of unity. Thus, the number of $\F_q$-points of $\overline{X}_1$ is $2q-1$.
\end{exa}

\section{Application: the Dwork family}\label{section:dwork}
In this section, we show how the Dwork family fits into our framework.
Let $X_u$ be the Dwork family defined by
\[
y_1^{d+1}+\dots + y_{d+1}^{d+1} - u^{-1} (d+1) y_1 \cdots y_{d+1} = 0
\]
in the projective space $\Ps^d$ with homogeneous coordinates $y_1, \dots, y_{d+1}$.  Let $Z_u$ be the intersection of $X_u$ with the maximal torus $\{y_1\cdots y_{d+1} \neq 0\}$ of $\Ps^{d}$. Then, $Z_u$ is given by the vanishing of
\[
f(x)= x_1^{d+1}+\dots + x_{d}^{d+1} +1  - u^{-1} (d+1) x_1 \cdots x_{d}
\]
in the torus $\G_m^{d}$ with coordinates $x_1, \dots, x_{d}$. 
The polynomial $f$ has $d+2$ monomials, so it fits our framework.  The Newton polytope $\Delta$ of $Z_{u}$ is a dilation of the standard $d$-simplex, namely 
\[
\big\{(z_1,\dots ,z_d)\in \R^d \mid  z_1+\dots +z_d \le {d+1} \text{ and }z_i \ge 0 \text{ for each }i \big\}.
\]
Thus, the toric variety $\Ps_{\Delta}$ is the projective space $\Ps^{d}$. Furthermore, the compactification $\overline{Z}_u$ of $Z_{u}$ coincides with $X_u$. Now, the Gale dual of $Z_u$ is the vector $\ga= (1,\dots,1,-(d+1)) \in \Z^{d+2}$. Furthermore, a choice of primitive hypersurface $Z_u^{\operatorname{prim}}$ for $\ga$ is given by
\begin{equation}\label{eq:dwork2}
  y_1^{d+1}y_2^{-1}\cdots y_d^{-1}+y_2 + \dots +y_{d} + 1 -(d+1) u^{-1} y_1 =0.  
\end{equation}
It is clear that $Z_u$ is isomorphic to the variety given by
\begin{align*}
     y_1^{d+1}y_2^{-1}\cdots y_d^{-1}+y_2 + \dots +y_{d} + 1 -(d+1) u^{-1} y_1 &=0,\\
     y_1&= x_1\cdots x_d,\\
     y_j&= x_j^{d+1}  \text{ for } j =2,\dots, d
\end{align*}
in $\G_m^{d} \times \G_m^d$.
Thus, the matrix $\mathcal{N}$ encoding the covering $Z_u\to Z_u^{\operatorname{prim}}$ is equal to the $d \times d$ matrix
\[
\begin{bmatrix}
1 & 1 & 1& \cdots & 1 \\
0 & d+1 &0& \cdots & 0 \\
0 & 0 &d+1& \cdots & 0 \\
\vdots & \vdots & \vdots & \ddots & \vdots \\
0 & 0 & 0& \cdots & d+1
\end{bmatrix}.
\]
Let $q$ be a prime power coprime to $d+1$, and let $e=\gcd(q-1,d+1)$.
Then, 
\begin{align*}
    \Lambda(q) &= \left\{ \left(0, \frac{q^\times}{e} \lambda_1, \dots, \frac{q^\times}{e}\lambda_{d-1}\right) \mid \lambda_i \in \{1,\dots,e\} \text{ for } i=1, \dots,d-1\right\}.
\end{align*}
Meanwhile, for $k=1,\dots,d$, the vector $\rho_k$ may be taken to be the $k$-th row of the $d\times(d+2)$ matrix
\[
\begin{bmatrix}
0& 0 & 0 & \cdots & 0 & -1 & 1 \\
0& 1 & 0 & \cdots & 0 & -1 & 0 \\
0& 0 & 1 & \cdots & 0 & -1 & 0 \\
\vdots &\vdots & \vdots & \ddots & \vdots & \vdots & \vdots \\
0& 0 & 0 & \cdots & 1 & -1 & 0
\end{bmatrix}.
\]
Thus, if $\eta= \frac{q^\times}{e}(0,\lambda_1,\dots,\lambda_{d-1}) \in \Lambda(q)$, then $\delta(\eta)  = \frac{q^\times}{e} (0,\lambda_1,\dots,\lambda_{d-1},-\lambda_1-\dots -\lambda_{d-1},0) \in \left( \Z/q^\times\Z\right)^{d+2}$. 
Furthermore, we have
\begin{align*}
    \sigma &=(-(d+1)u^{-1},1,\dots,1) \in \left(\F_q^\times\right)^{d},\\
    t &= (-(d+1) u^{-1})^{-(d+1)},\\
    \ga^{\ga} &= (-(d+1))^{-(d+1)}.
\end{align*}
Substituting these data into Corollary \ref{cor:simplices-hg}, and noting that $s_{\delta(\eta)}(0)=1$ (since $\delta(\eta)_1=\delta(\eta)_{d+2}=0$) and $\chi^{-\eta}(\sigma)=1$ (since $\eta_1=0$ and $\sigma_2=\dots=\sigma_d=1$), we have, for non-zero $u$,
\begin{align*}
\#X_u(\F_q) &= \frac{q^{d}-1}{q-1} - \sum_{\eta \in \Lambda(q)} g\left(\delta(\eta )\right) F_q\left(\gamma,\delta(\eta), q^\times \bigm\vert u^{d+1}\right).
\end{align*}
As in the introduction, write $\delta(\lambda)  = (0,\lambda_1,\dots,\lambda_{d-1},-\lambda_1-\dots -\lambda_{d-1},0) \in \left( \Z/e\Z\right)^{d+2}$. Then,
\begin{align*}
    \#X_u(\F_q) &= \frac{q^{d}-1}{q-1} - \sum_{\lambda \in (\Z/e\Z)^{d-1}} g\left(\frac{q^\times}{e}\delta(\lambda)\right) F_q\left(\gamma, \frac{q^\times}{e} \delta(\lambda ), q^\times \bigm\vert u^{d+1}\right)\\
    &=\frac{q^{d}-1}{q-1} - \sum_{\lambda \in (\Z/e\Z)^{d-1}} g\left(\frac{q^\times}{e}\delta(\lambda)\right) F_q\left(\gamma,  \delta(\lambda ), e \bigm\vert u^{d+1}\right).
\end{align*}
This proves Theorem \ref{theo:Dwork} from the introduction. 
Note that the hypergeometric parameters $(\al \,;\be)$ associated to the gamma triple $(\ga,\delta(\lambda),e)$ can be computed by cancelling the common factors of the numerator and denominator of 
\[
Q= \frac{T^{d+1}-1}{(T- \zeta_{e}^{-\lambda_1}) \dots (T- \zeta_{e}^{-\lambda_{d-1}}) (T- \zeta_{e}^{\lambda_1+\dots +\lambda_{d-1}}) (T-1)}.
\]

Now we specialise to $d+1=4$ and prove Theorem \ref{theo:dwork-3}.
The Gale dual $\ga$ is $(1,1,1,1,-4)$.
First we consider the case $q\equiv 3 \pmod{4}$: in this case $e=\gcd(q-1,4) = 2$, and we have
    \[
    Q= \frac{T^{4}-1}{(T- \zeta_2^{-\lambda_1}) (T- \zeta_2^{-\lambda_{2}}) (T- \zeta_2^{\lambda_1+\lambda_{2}})(T-1) }.
    \]
    Thus, 
    \begin{align*}
     \#X_u(\F_q) 
     &= \frac{q^{3}-1}{q-1} - \sum_{\lambda \in (\Z/2\Z)^{2}} g\left(\frac{q-1}{2}\delta(\lambda)\right)F_q\left(\gamma,\delta(\lambda),2 \mid u^{4}\right)\\
    &= \frac{q^{3}-1}{q-1} - g(0)^5 F_q(\ga,0,2 \mid u^4) - 3g(0)^3 g\left(q^\times/2 \right)^2 F_q(\ga, (0,0,1,1,0),2 \mid u^4)\\
     &= q^2+q+1 +  F_q\biggl(\begin{matrix} \frac{1}{4}, \frac{2}{4}, \frac{3}{4} \\ 1,  1, 1 \end{matrix} \:\bigg|\:\: u^{4} \biggr)- 3 q F_q\biggl(\begin{matrix} \frac{1}{4},  \frac{3}{4} \\ \frac{1}{2},  1 \end{matrix} \:\bigg|\:\: u^{4} \biggr).
    \end{align*}
    
Now suppose that $q\equiv1 \pmod{4}$. Then, $e=\gcd(q-1,4) = 4$. We have
    \[
    Q= \frac{T^{4}-1}{(T- \zeta_4^{-\lambda_1}) (T- \zeta_4^{-\lambda_{2}}) (T- \zeta_4^{\lambda_1+\lambda_{2}})(T-1) }, \text{ and }
    \]
     \begin{equation}\label{eq:dwork-d-3}
        \#X_u(\F_q) 
     = \frac{q^{3}-1}{q-1} - \sum_{\lambda \in (\Z/4\Z)^{2}} g\left(\frac{q-1}{4}\delta(\lambda)\right)F_q\left(\gamma,\delta(\lambda),4 \mid u^{4}\right). 
     \end{equation}
     
    In the following table, we summarise the possible $\delta(\lambda)$, up to permutation, the value of $g\bigl(q^\times\delta(\lambda)/4\bigr)$, and the associated parameters $(\al \, ; \be)$.
    \begin{table}[ht]
    \centering
    \begin{tabular}{c c c c}
    \hline
    $\delta(\lambda)$ & \# perm. & $g\bigl(q^\times\delta(\lambda)/4\bigr)$ & $(\alpha \, ; \beta)$ \\
    \hline
    $(0,0,0,0,0)$      & 1 & $g(0)^5=-1$ &  $((1/4,2/4,3/4)\, ;(1,1,1))$\\
    $(0,2,2,0,0)$      & 3 & $g(0)^3g(q^\times/2)^2=-q$ & $((1/4,3/4)\, ;(1/2,1))$ \\
    $(0,1,-1,0,0)$     & 6 & $g(0)^3g(-q^\times/4)g(q^\times/4)= -\chi(-1)^{q^\times/4}q$  & $((1/2) \, ; (1) )$ \\
    $(0,1,1,2,0)$      & 3 &$g(0)^2g\left(q^\times/4\right)^2g\left(q^\times/2\right)$  & $((1/4)\,;(3/4))$ \\
    $(0,-1,-1,2,0)$    & 3 &$g(0)^2g\left(-q^\times/4\right)^2   g\left(q^\times/2\right)$  & $((3/4)\,;(1/4))$ \\
    \end{tabular}
    \end{table}
    
    Substituting into Equation \eqref{eq:dwork-d-3}, we get 
    \begin{align*}
    \#X_u(\F_q)  &= q^2+q+1 + F_q\biggl(\begin{matrix} \frac{1}{4}, \frac{2}{4}, \frac{3}{4} \\ 1,  1, 1 \end{matrix} \:\bigg|\:\: u^{4} \biggr)+3 q \cdot  F_q\biggl(\begin{matrix} \frac{1}{4},  \frac{3}{4} \\ \frac{1}{2},  1 \end{matrix} \:\bigg|\:\: u^{4} \biggr)+6\chi(-1)^{q^{\times}/4} q \cdot F_q\biggl(\begin{matrix} \frac{1}{2} \\ 1 \end{matrix} \:\bigg|\:\: u^{4} \biggr)\\
     &\quad -3g\left(q^\times/4\right)^2 g\left(q^\times/2\right)F_q\biggl(\begin{matrix} \frac{1}{4} \\ \frac{3}{4} \end{matrix} \:\bigg|\:\: u^{4} \biggr)-3g\left(-q^\times/4\right)^2   g\left(q^\times/2\right)F_q\biggl(\begin{matrix} \frac{3}{4} \\ \frac{1}{4} \end{matrix} \:\bigg|\:\: u^{4} \biggr)\\
     &= q^2+q+1 + F_q\biggl(\begin{matrix} \frac{1}{4}, \frac{2}{4}, \frac{3}{4} \\ 1,  1, 1 \end{matrix} \:\bigg|\:\: u^{4} \biggr)+ 3q \cdot F_q\biggl(\begin{matrix} \frac{1}{4},  \frac{3}{4} \\ \frac{1}{2},  1 \end{matrix} \:\bigg|\:\: u^{4} \biggr) +12\chi(-1)^{q^{\times}/4} q\cdot  F_q\biggl(\begin{matrix} \frac{1}{2} \\ 1 \end{matrix} \:\bigg|\:\: u^{4} \biggr).
    \end{align*}
The last equality was obtained using Lemma \ref{lem:gauss-sums-properties} and the following lemma:
\begin{lem}
Suppose that $q\equiv1 \pmod{4}$. Then, 
\begin{align}
F_q\biggl(\begin{matrix} \frac{1}{4} \\ \frac{3}{4} \end{matrix} \:\bigg|\:\: u^{4} \biggr) &= -\chi(-1)^{q^\times/4} \frac{g(q^\times/2)}{g(q^\times/4)^2} F_q\biggl(\begin{matrix} \frac{1}{2} \\ 1 \end{matrix} \:\bigg|\:\: u^{4} \biggr), \label{eq:lem-sec-exa-1}\\
F_q\biggl(\begin{matrix} \frac{3}{4} \\ \frac{1}{4} \end{matrix} \:\bigg|\:\: u^{4} \biggr) &= -\chi(-1)^{q^\times/4} \frac{g(q^\times/2)}{g(-q^\times/4)^2} F_q\biggl(\begin{matrix} \frac{1}{2} \\ 1 \end{matrix} \:\bigg|\:\: u^{4} \biggr).\label{eq:lem-sec-exa-2}
\end{align}
\end{lem}
\begin{proof}
    The proof follows from Definition \ref{defi:finite-hg-function} by making the change of variables $m= \frac{1}{4}q^\times +n$ to prove Equation \eqref{eq:lem-sec-exa-1}, and $m= -\frac{1}{4}q^\times +n$ to prove Equation \eqref{eq:lem-sec-exa-2}.
\end{proof}
\bibliographystyle{amsalpha}
\bibliography{References}
\end{document}